\documentclass{amsart}
\usepackage{amsmath,amssymb,amsthm,mathrsfs}

\renewcommand{\theequation}{\arabic{section}.\arabic{equation}}

\newtheorem{theorem}{Theorem}[section]

\newtheorem{proposition}[theorem]{Proposition}
\newtheorem{lemma}[theorem]{Lemma}
\newtheorem{corollary}[theorem]{Corollary}
\newtheorem{remark}[theorem]{Remark}
\newtheorem{example}[theorem]{Example}

\newcommand{\sH}{{\mathcal H}}
\newcommand{\sK}{{\mathcal K}}
\newcommand{\sM}{{\mathcal M}}
\newcommand{\sF}{{\mathcal F}}
\newcommand{\sU}{{\mathcal U}}
\newcommand{\sV}{{\mathcal V}}
\newcommand{\sG}{{\mathcal G}}
\newcommand{\sL}{{\mathcal L}}
\newcommand{\sE}{{\mathcal E}}
\newcommand{\sD}{{\mathcal D}}

\newcommand{\sY}{{\mathcal Y}}
\newcommand{\sX}{{\mathcal X}}
\newcommand{\sZ}{{\mathcal Z}}

\newcommand{\sR}{{\mathcal R}}
\newcommand{\sA}{{\mathcal A}}

\newcommand{\sS}{{\mathcal S}}

\def\a{{\alpha}}
\def\b{{\beta}}
\def\d{\delta}
\def\de{\Delta}

\def\g{\gamma}
\def\ga{\Gamma}

\def\l{\lambda}
\def\la{\Lambda}
\def\s{\sigma}
\def\si{\Sigma}
\def\t{\tau}
\def\va{\varphi}
\def\o{\omega}
\def\om{\Omega}

\def\tht{\Theta}

\def\z{\zeta}
\def\ts{\times}

\def\iy{\infty}
\def\im{{\rm Im\, }}

\def\kr{{\rm Ker\, }}

\def\col{{\rm col\, }}

\def\lg{\langle}
\def\rg{\rangle}
\def\wh{\widehat}
\def\wt{\widetilde}

\def\Up{\Upsilon}

\newcommand{\ands}{\quad\mbox{and}\quad}

\newcommand{\BC}{{\mathbb C}}
\newcommand{\BD}{{\mathbb D}}
\newcommand{\BT}{{\mathbb T}}

\newcommand{\mat}[2]{\ensuremath{\left[\begin{array}{#1}
#2
\end{array} \right]}}
\newcommand{\sbm}[1]{\left[\begin{smallmatrix} #1\end{smallmatrix}\right]}
\newcommand{\wtil}[1]{\widetilde{#1}}
\newcommand{\what}[1]{\widehat{#1}}
\newcommand{\half}{\frac{1}{2}}
\newcommand{\tu}[1]{\textup{#1}}
\newcommand{\spec}{r_\textup{spec}}

\newcommand{\nn}{\nonumber}

\begin{document}

\title{All solutions  to an operator  Nevanlinna-Pick  interpolation  problem}

%---------Author 1
\author[A.E. Frazho]{A.E. Frazho}

\address{%
Department of Aeronautics and Astronautics, Purdue University\\
West Lafayette, IN 47907, USA}

\email{frazho@ecn.purdue.edu}

%---------Author 2

\author[S. ter Horst]{S. ter Horst}

\address{%
Department of Mathematics, Unit for BMI, North-West University\\
Private Bag X6001-209, Potchefstroom 2520, South Africa}

\email{sanne.terhorst@nwu.ac.za}

%---------Author 3

\author[M.A. Kaashoek]{M.A. Kaashoek}

\address{%
Department of Mathematics,
VU University Amsterdam\\
De Boelelaan 1081a, 1081 HV Amsterdam, The Netherlands}

\email{m.a.kaashoek@vu.nl}

\thanks{This work is based on the research supported in part by the
National Research Foundation of South Africa (Grant Number 90670 and 93406).}

\begin{abstract}
The main results presented in this paper provide a complete and explicit description of all solutions to  the left tangential  operator Nevanlinna-Pick interpolation problem assuming the associated  Pick operator is strictly positive.  The complexity of the solutions is similar to that found in descriptions of the sub-optimal Nehari problem and variation on the Nevanlinna-Pick interpolation problem in the Wiener class that have been obtained through the band method.  The main techniques used to derive the formulas are based on  the theory of co-isometric realizations, and  use the Douglas factorization lemma and  state space calculations.  A new feature is that we do not assume an additional stability assumption on our data, which allows us to view the Leech problem and a large class of commutant lifting problems as special cases. Although the  paper has partly the character of a survey article, all results are  proved in detail and some background material has been added to make the paper accessible to a large audience  including engineers.
\end{abstract}



\maketitle

%%%%%%%%%%%%%%%%%%%%%%%%%%%%%%%%%%%%%%%%%%%%%%%%%%%%%%%%%%%%%%%%%%%%%%%%%%%%%%%%%%%%%%%%%%%%%%%% Add material
%%%%%%%%%%%%%%%%%%%%%%%%%%%%%%%%%%%%%%%%%%%%%%%%%%%%%%%%%%%%%%%%%%%%%%%%%%%%%%%%%%%%%%%%%%%%%%%%

\setcounter{section}{0}
\setcounter{equation}{0}

\setcounter{section}{0}
\setcounter{equation}{0}

\section{Introduction}\label{sec:intro}

Nevanlinna-Pick interpolation problems have a long and interesting history which  goes back to the papers of G. Pick \cite{P16} and R. Nevanlinna \cite{N19} for scalar functions. Since then interpolation problems  with metric constraints involving matrix or operator-valued functions, in one or several variables,  has been a topic of intense study with rich applications to system and control theory, prediction theory and geophysics.  See, for example, the introductions  of the books  \cite{FF90, FFGK98},  Chapter 7  in the book \cite{AD08}, the papers  \cite{KKY07} and \cite{Kh98}, several variable papers \cite{am1,am2}, and references therein.

In the present  paper we deal with  the left tangential  Hilbert space operator  Nevanlinna-Pick interpolation problem   in one variable  with the unkowns being opeators.  Our aim is  to give a self-contained presentation combining  the best techniques from commutant lifting \cite{FF90,FFGK98}, the band method \cite{GKW89a,GKW89b,GKW91}, state space analysis \cite{am1,am2,btv},  and other interpolation methods \cite{AAK71,cs,KKY07,Kh98,D2009,Sz.-nk}.  In particular, the technique of extending a partial isometry used in the present paper goes back to work of Sz.-Nagy-Koranyi  \cite{Sz.-nk} and also appears in the so-called ``lurking Isometry''  method of Ball and co-authors \cite{BK70} and Arov-Grossman \cite{AG92}, to name only a few. In \cite{BB08} this problem was considered in the more general setting of the Drury-Arveson space and solved via a modification of the Potapov methodology.

Our proofs are not based on the commutant lifting method, and the approach taken here avoids the complications that arise in describing the solutions when the isometric lifting is not minimal, as is typically the case in the commutant lifting reformulation of the operator interpolation problem. As main tools we use the theory of co-isometric realizations, the Douglas factorization lemma and state space calculations, which are common in mathematical system theory.

As a by-product of   our method we  present  in Subsection \ref{Assec:co-iso} an alternative way to construct co-isometric realizations of Schur class functions,  which seems to be new and could be of interest in the multi-variable case. In the appendix we also give an alternative proof of  the  Beurling-Lax-Halmos theorem and present a new approach to the maximum entropy principle. We made an effort for  the paper    to be readable by someone whose has  an elementary knowledge of Hilbert space operator theory with state space techniques from systems and control theory.  On the other hand in order to achieve  self-containedness, the appendix provides  background material  that is used throughout the paper.

Let us now introduce the Hilbert space operator Nevanlinna-Pick interpolation problem we shall be dealing with and review some of our main new results. The data  for the problem is a triplet of bounded linear Hilbert space operators  $\{W, \wt{W}, Z\}$, where, for given Hilbert spaces $\sZ$, $\sY$ and $\sU$, we have
\[
Z:\sZ\to\sZ,\quad
W: \ell_+^2(\sY)\to\sZ,\quad
\wt{W}: \ell_+^2(\sU)\to\sZ,
\]
with $\ell^2_+(\sY)$ (respectively $\ell^2_+(\sU)$) the Hilbert space of square summable  unilateral  sequences of vectors from $\sY$ (respectively $\sU$), and where the following intertwining relations are satisfied
\begin{equation}\label{data1}
ZW=WS_\sY \ands
Z\wt{W}=\wt{W}S_\sU.
\end{equation}
Here $S_\sU$ and $S_\sY$ are the unilateral forward shift operators on $\ell_+^2(\sU)$ and $\ell_+^2(\sY)$, respectively.

We say that $F$ is a \emph{solution to the operator Nevanlinna-Pick (LTONP for short) interpolation problem  with data set}   $\{W, \wt{W}, Z\}$ if
\begin{equation}\label{defNP}
F\in \sS(\sU, \sY) \ands WT_F=\wt{W}.
\end{equation}
Here $T_F$ is the Toeplitz operator with defining function $F$ mapping $\ell_+^2(\sU)$ into $\ell_+^2(\sY)$. Moreover, $\sS(\sU, \sY)$ is the Schur class of operator-valued functions  whose values map  $\sU$ into $\sY$, that is, the set of all operator-valued analytic functions  $F$  in the open unit disc $\BD$  whose values map  $\sU$ into $\sY$ such that $\|F\|_\infty=\sup\{\|F(\lambda)\|:\lambda \in \mathbb{D}\} \leq 1$.

Note that this class of Nevanlinna-Pick interpolation problems has the same point evaluation interpolation condition as the one considered in Section 1.4 of \cite{FFGK98}, but is larger in the sense that, unlike in \cite{FFGK98}, we do not assume the spectral radius of $Z$ to be strictly less that one. To see that the point evaluation condition coincides with that of \cite{FFGK98}, note that the fact that $W$ and $\wt{W}$ satisfy \eqref{data1} implies that they are the controllability operators (cf., \cite[page 20]{FFGK98}) of the pairs $\{Z,B\}$ and $\{Z,\wt{B}\}$, respectively, where $B$ and $\wt{B}$ are the operators given by
\begin{equation}\label{def:BBtilde}
B=WE_\sY:\sY\to \sZ\quad\mbox{and}\quad \wt{B}=\wt{W}E_\sU:\sU\to \sZ.
\end{equation}
Here $E_\sY$  and $E_\sU$ are the operators embedding $\sY$ and $\sU$, respectively, into the first  component of  $\ell_+^2(\sY)$ and $\ell_+^2(\sU)$, respectively; see the final paragraph of this section for more details. Then for $F\in \sS(\sU,\sY)$, the operator $WT_F$ is also a controllability operator, namely for the pair $\{Z,(BF)(Z)_{\textup{left}}\}$, where
\[
(BF)(Z)_{\textup{left}}=\sum_{k=0}^\infty Z^k B F_k,
\]
with   $F_0, F_1, F_2, \ldots$ being  the Taylor coefficients of $F$  at zero. Then $WT_F=\wt{W}$ is equivalent to the left tangential operator argument condition $(BF)(Z)_{\textup{left}}=\wt{B}$.

Although the LTONP interpolation problem has a simple formulation, it covers two relevant special cases that will be discussed in Sections \ref{CL-ONP} and \ref{sec:Leech} below. In both cases it is essential that we do not demand that the spectral radius is strictly less than one. In Section \ref{CL-ONP} we discuss a large class of commutant lifting problems that can be written in the form of a LTONP interpolation problem. Conversely, any LTONP interpolation problem can be rewritten as a commutant lifting problem from this specific class. Hence the problems are equivalent in this sense. In this case, the operator $Z$ will be a compression of a unilateral forward shift operator and will typically not have spectral radius less than one. The connection with commutant  lifting is already observed in \cite[Section II.2]{FFGK98} and also appears  in the more general setting of the Drury-Arveson space in \cite{BB08}.

The second special case, discussed in Section \ref{sec:Leech}, is the Leech problem. This problem, and  its  solution, originates from a paper by R.B. Leech, which was written in 1971-1972, but published only recently \cite{L14}; see \cite{KR14} for an account of the history behind this paper. The Leech problem is another nontrivial example of a LTONP interpolation problem for which the operator $Z$ need not have spectral radius less than  one, in fact, in this case, the operator $Z$ is equal to a unilateral forward shift operator and hence its spectral radius is equal to one. Our analysis of the rational Leech problem \cite{FtHK14a, FtHK14b, FtHK15} inspired us to study  in detail  the class of LTONP interpolation problems. It led to new results and  improvements on our earlier results on the Leech problem.

Next we will present our main results. This requires some preparation.  Let $\{W, \wt{W}, Z\}$ be a LTONP data set.  Set $P=WW^*$ and $\wt{P}= \wt{W}\wt{W}^*$.   The intertwining relations in \eqref{data1} imply  that
\begin{align}
&P-ZPZ^*=BB^*, \  \mbox{where $B=WE_\sY:\sY\to \sZ$}, \label{basicid1a}\\
&\wt{P}-Z\wt{P}Z^*=\wt{B}\wt{B}^*, \  \mbox{where  $\wt{B}=\wt{W}E_\sU:\sU\to \sZ$}. \label{basicid1b}
\end{align}
Here,  as before (see \eqref{def:BBtilde}),    the maps  $E_\sY$  and $E_\sU$ are the  operators embedding  $\sY$ and $\sU$, respectively, into the first  component of  $\ell_+^2(\sY)$ and $\ell_+^2(\sU)$, respectively; see the final paragraph of this section for more details.   The operator $\la=P-\wt{P}$  is called the \emph{Pick operator} associated with the data set $\{W, \wt{W}, Z\}$.

If the LTONP interpolation problem is solvable, then necessarily the  Pick operator is non-negative. Indeed, assume there exists a function $F$  in $S(\sU,\sY)$ satisfying $WT_F = \widetilde{W}$. Then  $T_F$ is a contraction so that
\[
\lg\wt{P} x,x\rg = \|\wt{W}^* x\|^2 = \|T_F^* W^*x\|^2
\leq \|W^* x\|^2 = \lg P x,x\rg, \quad x\in \sZ.
\]
Hence $\la = P - \widetilde{P} \geq 0$.

The converse is also true. If the Pick operator is non-negative, then the  LTONP interpolation problem is solvable   (see Theorem \ref{thm:allsol1} in the next section). In this paper our aim is to describe all solutions, in particular for the case when  $\la$  is strictly positive.

To state our first main theorem we need two auxiliary operators.  Assume $P=WW^*$ is strictly positive, which is the case if $\la$ is strictly positive. Then there exist a Hilbert space $\sE$ and a pair of operators  $C:\sZ\to \sE$ and $D:\sY\to \sE$ such that
\begin{align}
&\begin{bmatrix}
  D & C \\
  B & Z
\end{bmatrix}
\begin{bmatrix}
I_\sY & 0 \\
  0 &  P
\end{bmatrix}
\begin{bmatrix}
D^* & B^* \\
  C^* & Z^*
\end{bmatrix}
= \begin{bmatrix}
  I_\sE  & 0 \\
  0 & P
\end{bmatrix},\label{semiunit1}\\[.2cm]
&\begin{bmatrix}
D^* &   B^* \\
  C^* & Z^*
\end{bmatrix}
\begin{bmatrix}
 I_\sE  & 0  \\
  0 &P^{-1}
\end{bmatrix}
\begin{bmatrix}
  D & C \\
  B & Z
\end{bmatrix}
 = \begin{bmatrix}
  I _\sY & 0  \\
  0 & P^{-1}
\end{bmatrix}. \label{semiunit2}
\end{align}
We shall call such a pair  $C$ and $D$  an \emph{admissible pair  of complementary operators} determined by the data set $\{W, \wt{W}, Z\}$.  In \eqref{semiunit1} and \eqref{semiunit2} the symbols $I_\sE$ and $I_\sY$ denote the identity operators on the spaces $\sE$ and $\sY$, respectively. In general,   when it is clear from the context  on which space the identity operator is acting,  the subscript is omitted and we simply write~$I$.

An application of  Lemma XXVIII.7.1   in \cite{GGK2} shows that admissible pairs exist and that such a pair is unique up to multiplication by a unitary operator from the left. There are various  ways to construct admissible pairs in a concrete way,  also in a multivariable setting (see, e.g., \cite{BBF07}). In this introduction we  mention only one way  to obtain such a pair of operators,  namely as follows.   Since  $Z W = W S_\sY$, the space    $\kr W$ is an invariant subspace for the forward shift $S_{\sY}$. But then, by   the Beurling-Lax-Halmos theorem,    there exists an inner function $\tht\in \sS(\sE, \sU)$, for some Hilbert space $\sE$, such that  $\kr W=\im T_\tht$.   Now put
\begin{equation}
\label{defCD}
C=E_\sE^* T_\tht^* S_\sY W^* P^{-1}: \sZ\to \sE \ands  D=\tht(0)^*:\sY\to \sE.
\end{equation}
Then $C$ and $D$ form  an admissible pair  of complementary operators. Another  method to construct   admissible pairs of complementary operators, which has the advantage that it   can be readily used in Matlab in the finite dimensional case, is given   Section \ref{Assec:CD}.  We are now ready to state our  first main result.

\begin{theorem}\label{thm:main1}  Let $\{W, \wt{W}, Z\}$ be a  data set for a LTONP interpolation problem. Assume $\la=WW^*-\wt{W}\wt{W}^* $ is strictly positive. Then $P=WW^*$ is strictly positive and the operator $\la^{-1}-P^{-1}$ is  non-negative, the operator $Z^*$ is pointwise stable and  its spectral radius is less than or equal to one.  Furthermore,  all solutions to  the LTONP interpolation problem are given by
\begin{equation}
\label{allsol1a}
F(\lambda)= \Big(\Upsilon_{11}(\l)  X(\l)+\Upsilon_{12}(\l)\Big)\Big(\Upsilon_{21}(\l) X(\l) +\Upsilon_{22}(\l)\Big)^{-1},\quad \l\in\BD
\end{equation}
where the free parameter $X$ is an arbitrary Schur class function,  $X\in \sS(\sU, \sE)$, and the coefficients in \eqref{allsol1a}  are the analytic functions on $\BD$  given by
\begin{align}
\Upsilon_{11}(\l)&=D^*Q_\circ  +\l B^*(I -\l Z^*)^{-1}\la^{-1}{P}C^* Q_\circ, \label{defUp11} \\
\Upsilon_{12}(\l)&=B^*(I -\l Z^*)^{-1}\la^{-1}\wt{B}R_\circ,\label{defUp12}
\\
\Upsilon_{21}(\l)&=\l \wt{B}^*(I -\l Z^*)^{-1}\la^{-1} PC^*Q_\circ, \label{defUp21} \\
\Upsilon_{22}(\l)&=R_\circ  +\wt{B}^*(I -\l Z^*)^{-1}\la^{-1}\wt{B} R_\circ.\label{defUp22}
\end{align}
Here the operators $B$ and $\wt{B}$ are given by \eqref{basicid1a} and \eqref{basicid1b}, respectively, the operators $C:\sZ\to \sE$ and $D:\sY\to \sE$  form an admissible   pair of complementary  operators, and $Q_\circ$ and  $R_\circ$ are the strictly positive operators given by
\begin{equation}\label{defQR0}
\begin{aligned}
Q_\circ&=\left(I_\sE+CP(\la^{-1}-P^{-1})PC^*\right)^{-\frac{1}{2}}: \sE\to \sE,\\ R_\circ&=(I_\sU+\wt{B}^*\la^{-1}\wt{B})^{-\frac{1}{2}}: \sU\to \sU.
\end{aligned}
\end{equation}
The parameterization  given by \eqref{allsol1a}  is proper, that is, the map $X\mapsto F$ is one-to-one.
\end{theorem}

Note that \eqref{allsol1a} implicitly contains the statement that the operator $\Upsilon_{21}(\l)X(\l)+\Upsilon_{22}(\l)$ is invertible for each $\l \in \BD$. In particular, taking $X\equiv 0$ in \eqref{allsol1a}, we see that under the conditions of the above theorem, the operator $\Upsilon_{22}(\l)$ is invertible for each $\l\in \BD$.

Furthermore, setting $X\equiv 0$ in \eqref{allsol1a}, we obtain the so-called central solution $F_\circ(\l)=\Upsilon_{12}(\l)\Upsilon_{22}(\l)^{-1}$, which is introduced, in a different way, in Remark \ref{rem:centrsol}. See also  Theorem \ref{thm:centrsol} and Proposition \ref{quotientf}. In  Section \ref{sec:max}  we show that the central solution is the unique  Schur class function that maximizes a notion of entropy among all solutions; see Theorem \ref{thm:maxent} below.

By Theorem \ref{thm:main1} the set of all solutions is parameterised by the Schur class $\sS(\sU, \sE)$. Hence the LTONP interpolation problem has a single solution if and only if  $\sE=\{0\}$; [we assume that $\sU$ is non-trivial]. On the other hand we know that $\sE$ can be chosen in such a way that $\kr W= \im T_\Theta$, where $\Theta\in S(\sE, \sU)$  is an inner function. Thus $\sE=\{0\}$ holds if and only if  $\kr W=\{0\}$, i.e., $W$ is one-to-one. On the other hand, since we assume $\la$ to be strictly positive, $WW^*$ is also strictly positive. Thus there exists a single solution if and only if $W$ is invertible.

In Section \ref{sec:Leech} we specialize Theorem \ref{thm:main1} for the Leech problem case, yielding Theorem~\ref{thm:Leech}   below, which gives a  generalization and a further improvement of the description of all solutions of the rational Leech problem given in  \cite{FtHK15}.

The explicit  formulas for  the functions $\Upsilon_{ij}$, $1\leq i,j \leq 2$,  given in \eqref{defUp11}--\eqref{defUp22} are  new. The formulas are of the same complexity as the  corresponding formulas for the coefficients appearing in the linear fractional representation of all solutions of the sub-optimal Nehari problem  presented in the classical  Adamjan-Arov-Kre\v{\i}n   paper \cite{AAK71}. See also  Theorem XXXV.4.1 in \cite{GGK2} where the linear fractional representation  of all solutions of the sub-optimal Nehari problem in the Wiener class setting is obtained as an application of the band method \cite{GKW89a}  and \cite{GKW89b}.  The variation of the band method for solving  extension problems presented in \cite{KZ99} and the related unpublished manuscript \cite{K99} inspired us  to  derive the formulas in Theorem \ref{thm:main1}.

When the inner function $\tht$ determined by  $\kr W=\im T_\tht$ is bi-inner, then  the LTONP  interpolation  problem is equivalent to a Nehari extension problem. But even in this special case, it requires some work to derive the formulas \eqref{defUp11} -- \eqref{defUp22}; cf., \cite[Section XXXV.5]{GGK2}.

The next theorem is an addition to Theorem \ref{thm:main1} which will allow us to derive    further properties for  the coefficients   $\Upsilon_{ij}$, $1\leq i,j \leq 2$, in the linear fractional representation \eqref{allsol1a}; see Proposition \ref{prop:propsUps} below and Section \ref{sec:propsUps}. The theorem also shows that the functions  \eqref{defUp11} -- \eqref{defUp22} are the natural  analogs of the formulas appearing \cite[Theorem XXXV.4.1]{GGK2} for the Nehari problem.

\begin{theorem}\label{thm:main1a} Let $\{W, \wt{W}, Z\}$ be a  data set for a  LTONP interpolation problem. Assume $\la=WW^*-\wt{W}\wt{W}^* $ is strictly positive. Then $P=WW^*$ is strictly positive, the operator
\begin{equation}\label{defA}
A=W^*{P^{-1}}\wt{W}:\ell_+^2(\sU)\to \ell_+^2(\sY)
\end{equation}
is a strict contraction,  and  the functions  defined by \eqref{defUp11} -- \eqref{defUp22}  are also given by
\begin{align}
\Upsilon_{11}(\l)  &= D^*Q_\circ + \lambda  E_{\sY}^*
(I - \l S_{\sY}^*)^{-1}(I - A A^*)^{-1}   W^*C^* Q_\circ,\label{Upsilon11}\\
\Upsilon_{12}(\l)   &=E_{\sY}^*\big(I - \l  S_{\sY}^*\big)^{-1} A (I - A^*A)^{-1}E_{\sU}R_\circ,\label{Upsilon12}\\
\Upsilon_{21}(\l)    &=   \lambda E_{\sU}^*
(I  -\l  S_{\sU}^*)^{-1} A^* (I - AA^*)^{-1}  W^*C^*Q_\circ, \label{Upsilon21}\\
\Upsilon_{22}(\l)  &=
E_{\sU}^*\big(I - \l  S_{\sU}^*\big)^{-1} (I - A^*A)^{-1}E_{\sU}R_\circ. \label{Upsilon22}
\end{align}
Here, as in the preceding theorem, $C:\sZ\to \sE$ and $D:\sY\to \sE$  form an admissible   pair of complementary  operators determined by the data. Furthermore, the strictly positive operators $Q_\circ$ and $R_\circ$ defined by \eqref{defQR0} are also given by
\begin{align}
Q_\circ  &= \Big(I_\sE + C W A \big(I - A^*A\big)^{-1}A^* W^* C^*\Big)^{-\frac{1}{2}},\label{Q0}\\
R_\circ & = \Big(E_{\sU}^*\big(I - A^*A\big)^{-1} E_{\sU}\Big)^{-\frac{1}{2}}. \label{R0}
\end{align}
\end{theorem}

In the following result we list a few properties of the coefficients of the linear fractional transformation \eqref{allsol1a}.

\begin{proposition}\label{prop:propsUps}  Let $\{W, \wt{W}, Z\}$ be a  data set for a LTONP interpolation problem. Assume $\la=WW^*-\wt{W}\wt{W}^* $ is strictly positive. Then  the functions $\Upsilon_{ij}$, $1\leq  i,j\leq 2$, given by  \eqref{defUp11}--\eqref{defUp22}   are $H^2$-functions. More precisely, we have
\begin{align}
&\Upsilon_{11}(\cdot) x\in H^2(\sY)  \ands \Upsilon_{21}(\cdot) x\in H^2(\sU), \quad x\in \sE, \label{Hardy1}\\
&\Upsilon_{12}(\cdot) u\in H^2(\sY)  \ands \Upsilon_{22}(\cdot) u\in H^2(\sU), \quad u\in \sU. \label{Hardy2}
\end{align}
Moreover,  the functions $\Upsilon_{i,j}$ form a $2 \times 2$ $J$-contractive operator function, that is,  for all $\l\in\BD$ we have
\begin{equation}\label{Jcontractive}
\begin{bmatrix}
\Upsilon_{11}(\l)^*  &   \Upsilon_{21}(\l) ^*\\
\Upsilon_{12}(\l)^*  &   \Upsilon_{22}(\l) ^*
  \end{bmatrix}  \mat{cc}{I_\sY&0\\0&-I_{\sU}}
  \mat{cc}{\Up_{11}(\l)&\Up_{12}(\l)\\ \Up_{21}(\l) & \Up_{22}(\l)}
  \leq  \mat{cc}{I_\sE&0\\0&-I_{\sU}},
\end{equation}
with equality for each $\l$ in the intersection of the resolvent set of $Z$ and the unit circle  $\BT$. Furthermore, $\Upsilon_{22}(\l)$ is invertible for each $\l\in \BD$ and $\Upsilon_{22}(\l)^{-1}$ is a Schur class function.
\end{proposition}

Here for any Hilbert space $\sV$ the symbol $H^2(\sV)$ stands for the Hardy space of $\sV$-valued measurable functions on the unit circle $\BT$ that are square integrable and whose negative Fourier coefficients are equal to zero. Equivalently,  $\va\in H^2(\sV)$ if and only if $\va$ is an $\sV$-valued analytic function on the unit $\BD$ and  its Taylor coefficients $\va_0, \va_1, \va_2, \dots$ are square summable in norm.

Assume that  $\la=WW^*-\wt{W}\wt{W}^* $ is strictly positive.  Then  $A= W^* P^{-1} \widetilde{W}$ is a strict contraction. Because $P = WW^*$ is strictly positive, we see that $\im W$ is closed and $\im A\subset  \im W^*$. Furthermore, $WA=\wt{W}$, and hence $W(T_F-A)=0$ for any solution $F$ to the LTONP  interpolation  problem. In other words, if $F$ is a solution to the LTONP  interpolation  problem, then necessarily
\[
T_F= \begin{bmatrix}A\\ \star \end{bmatrix}:
\ell_+^2(\sU)\to \begin{bmatrix}\im W^*\\ \kr W \end{bmatrix}.
\]
The converse is also true.  This observation enables us to rephrase the LTONP interpolation problem as a commutant lifting problem.  On the other hand, as we shall see in Section \ref{CL-ONP}, a large class of  commutant lifting problems can be viewed as LTONP interpolation problems, and  hence Theorem \ref{thm:main1a} can be used to describe all solutions  of  a large class commutant lifting problems.  This will lead to  a  commutant lifting version of Theorem \ref{thm:main1a}; see Theorem \ref{thm:main2}  below.

\medskip
\noindent\textsc{Contents.}
The paper consists of nine sections, including the present introduction, and an appendix. In Section \ref{sec:ONP-coiso} we develop our primary techniques that are used to prove the main results, namely observable, co-isometric realizations from system theory, and we show how solutions can be obtained from a specific class of observable, co-isometric realizations, referred to as $\la$-preferable. The main result, Theorem \ref{thm:allsol1}, presents yet another description of the solutions to the LTONP interpolation problem. This description is less explicit, but on the other hand only requires the  Pick operator  to be  non-negative. In Section \ref{sec: prfThm2.1} we prove the main result of Section \ref{sec:ONP-coiso}, Theorem \ref{thm:allsol1}. Starting with Section \ref{sec:strictlypos} we add the assumption that the  Pick operator  is strictly positive. The main results, Theorems \ref{thm:main1} and \ref{thm:main1a} are proven in Sections \ref{sec:strictlypos} and \ref{sec:Thm1.2}, respectively. The next section is devoted to the proof of Proposition \ref{prop:propsUps}. Here we also show that the central solution, introduced in Remark \ref{rem:centrsol}, is indeed given by the quotient formula mentioned in the first paragraph after Theorem~\ref{thm:main1}; see Proposition \ref{quotientf}. In Section \ref{sec:max} we introduce a notion of entropy associated with the LTONP interpolation problem and show that the central solution is the unique solution that maximizes the entropy. This result is in correspondence with similar results on  metric  constrained interpolation; cf., Section IV.7 in \cite{FFGK98}. The new feature in the present paper is that we can rephrase the entropy of a solution in terms of its $\la$-preferable, observable, co-isometric realizations. In the last two sections, Sections \ref{CL-ONP} and \ref{sec:Leech}, we describe the connections with the commutant lifting problem and the Leech problem, respectively. Finally, the appendix consists of  seven subsections containing various preliminary results   that are used throughout the paper, with proofs often added for the sake of completeness.

\smallskip
\noindent\textsc{Terminology and Notation.}
We conclude this introduction with a few words on terminology and notation. With the term {\em operator} we will always mean a bounded linear operator. Moreover,  we say that  an operator  is \emph{invertible}  when it is both injective and surjective, and in that case  its  inverse is an operator, and hence bounded. An operator $T$ on a Hilbert space $\sH$ is called \emph{strictly positive} whenever it is  non-negative  ($T\geq 0$) and invertible; we denote this by $T\gg 0$. The unique  non-negative  square root of a  non-negative operator $T$ is denoted by $T^\half$. Furthermore, an operator $T$ on $\sH$ is said to be \emph{exponentially stable} whenever its  \emph{spectrum} $\s(T)$ is inside the open unit disc $\BD$, in other words, when the \emph{spectral radius} $\spec(T)$ of $T$ is strictly less than one. Moreover, we say that $T$ is \emph{pointwise stable} whenever $T^n h\to 0$ for each $h\in \sH$;  by some authors (see, e.g., Definition 4.5 in \cite{BB13}) this kind of stability is referred  to as strongly stable.   Clearly, a exponentially stable  operator is also pointwise stable. A subspace $\sM$ of a Hilbert space $\sH$ is by definition a closed linear manifold in $\sH$. Given a subspace $\sM$ of $\sH$ we write $P_{\sM}$ for the orthogonal projection on $\sH$ along $\sM$. We will also use the embedding operator $\tau_\sM:\sM\to \sH$, which maps $m\in\sM$ to $m\in\sH$.  Its adjoint $\tau_\sM^*:\sH\to\sM$ will also be denoted by $\Pi_\sM$, and thus $\Pi_\sM^*$ is the embedding operator $\tau_\sM$.  Recall that $S_\sU$ denotes the unilateral forward shift operator on $\ell^2_+(\sU)$, for a given Hilbert space $\sU$. We will also need the operator $E_\sU: \sU\to \ell^2_+(\sU)$ which is the embedding operator that embeds $\sU$ into the first entry of $\ell^2_+(\sU)$, that is, $E_\sU u=\mat{cccc}{u&0&0&\cdots}^\top \in\ell^2_+(\sU)$. Here, and in the sequel,  the symbol  ${}^\top$ indicates the block transpose. Hence for a (finite or infinite) sequence $C_1,C_2,\ldots$ of vectors or operators we have
\[
\mat{ccc}{C_1&C_2& \cdots}^\top=\mat{c}{C_1\\ C_2\\ \vdots} \ands
\mat{c}{C_1\\ C_2\\ \vdots}^\top=\mat{ccc}{C_1&C_2& \cdots}.
\]
Finally, for any $y= \col \big[y_j\big]_{j=0}^\iy$ in $\ell_+^2(\sY)$ we have
\begin{equation}\label{Ftransf}
\wh{y}(\l):=E_\sY^*(I-\l S_\sY^*)^{-1} \begin{bmatrix}
                             y_0 \\
                             y_1 \\
                             y_2 \\
                             \vdots \\
                           \end{bmatrix}
= \sum_{n=0}^\infty \lambda^n y_n,\quad \l\in \BD.
\end{equation}
It follows that $\wh{y}$ belongs to the Hardy space $H^2(\mathcal{Y})$, and any function in the Hardy space $H^2(\mathcal{Y})$ is obtained in this way. The map $y\mapsto  \wh{y}$ is the   Fourier transform  mapping $\ell_+^2(\mathcal{Y})$ onto the Hardy space $H^2(\mathcal{Y})$.

%:Sec2
\setcounter{equation}{0}
\section{Operator Nevanlinna-Pick interpolation  and co-isometric realizations}\label{sec:ONP-coiso}

 Throughout this section $\{W,\wt{W},Z\}$  is a data set for a LTONP interpolation problem, and $\la$ is the associate Pick operator. We assume  that  $\Lambda$ is a non-negative  operator, but not necessarily strictly positive, and we define  $\mathcal{Z}_\circ$ to be the closure of the range of $\Lambda$. Thus
 \begin{equation}\label{defZcirc}
 \sZ=\sZ_\circ \oplus \kr \la.
 \end{equation}
The main result of this section, Theorem \ref{thm:allsol1} below, provides a Redheffer type description of the set of all solutions of the LTONP interpolation problem with data set $\{W,\wt{W},Z\}$.  The proof of this result will be given in Section \ref{sec: prfThm2.1}, but much of the  preparatory  work is done in the current section.

From the definition of the Pick operator and  the two identities \eqref{basicid1a} and  \eqref{basicid1b} it follows that
\begin{equation}\label{Pick1a}
 \la - Z \la Z^* =BB^* - \wt{B}\wt{B}^*.
\end{equation}
Since $\Lambda \geq 0$,   the identity \eqref{Pick1a}  can be rewritten as $K_1 K_1^*=K_2 K_2^*$, where
\begin{equation}\label{Pick1b}
K_1=\begin{bmatrix} \wt{B}  &   \la^{\frac{1}{2}} \end{bmatrix}:
\begin{bmatrix}\sU\\ \sZ\end{bmatrix} \to \sZ \ands
K_2=\begin{bmatrix} B  & Z \la^{\frac{1}{2}} \end{bmatrix}:
\begin{bmatrix}\sY\\ \sZ\end{bmatrix} \to \sZ.
\end{equation}
This allows to apply Lemma \ref{lem:KK1}. Let $\mathcal{F}$ and $\mathcal{F}^\prime$ be the subspaces defined by
\begin{equation}\label{defFFprime}
\mathcal{F}=    \overline{\im K_1^*}
\ands \mathcal{F}^\prime= \overline{\im K_2^*}.
\end{equation}
Notice that $\sF$ is a subspace of $\sU\oplus  \sZ_\circ$ while $\sF^\prime$ is a subspace of $\sY\oplus  \sZ_\circ$, where  $\mathcal{Z}_\circ$ is the subspace of  $\sZ$ given by \eqref{defZcirc}.  Applying Lemma \ref{lem:KK1} we see  that there exists a unique   operator  $\o$ from $\sF$ into $\sF^\prime$ such that
\begin{equation}\label{defomega1}
\begin{bmatrix} B  & Z \la^{\frac{1}{2}} \end{bmatrix}
\begin{bmatrix} B^*   \\ \la^{\frac{1}{2}}Z^* \end{bmatrix} =\begin{bmatrix} B  & Z \la^{\frac{1}{2}} \end{bmatrix}\o \begin{bmatrix}
  \wt{B}^*    \\ \Lambda^{\frac{1}{2}} \end{bmatrix}=
 \begin{bmatrix} \wt{B}  &   \la^{\frac{1}{2}} \end{bmatrix}
\begin{bmatrix} \wt{B}^*    \\ \la^{\frac{1}{2}} \end{bmatrix}.
 \end{equation}
Moreover,  $\o$ is   a unitary operator mapping   $\sF$ onto $\sF^\prime$. We shall refer to $\o$ as the \emph{unitary operator determined by the data set $\{W,\wt{W},Z\}$}.  Note that the two identities in \eqref{defomega1} imply that
\begin{equation}\label{defomega2}
\o \begin{bmatrix} \wt{B}^*    \\ \la^{\frac{1}{2}} \end{bmatrix}=
\begin{bmatrix} B^*   \\ \la^{\frac{1}{2}}Z^* \end{bmatrix}
\ands \begin{bmatrix} B  & Z \la^{\frac{1}{2}} \end{bmatrix}\o
=  \left.\begin{bmatrix} \wt{B}  &   \Lambda^{\frac{1}{2}}\end{bmatrix}\right|\mathcal{F}.
\end{equation}
In fact each of the identities in \eqref{defomega2} separately can be used as the definition of $\o$.

In the sequel $\sG$ and $\sG^\prime$ will denote the orthogonal complements of  $\sF$ and $\sF^\prime$  in $\sU\oplus \sZ_\circ$ and $\sY\oplus \sZ_\circ$, respectively, that is,
\begin{equation}\label{defGGprime}
\mathcal{G}
= \left(\mathcal{U}\oplus  \mathcal{Z}_\circ\right)\ominus \mathcal{F}
\quad \mbox{and} \quad
\mathcal{G}^\prime
= \left(\mathcal{Y}\oplus  \mathcal{Z}_\circ\right)\ominus \mathcal{F}^\prime.
\end{equation}
In particular, $\sF\oplus\sG=\sU\oplus \sZ_\circ$ and  $\sF^\prime \oplus\sG^\prime=\sY\oplus \sZ_\circ$. The fact that $\sG$ is perpendicular to $\sF$  and  $\sG^\prime$ is perpendicular to $\sF^\prime$ implies that
\begin{equation}\label{zeroGGprime}
K_1 \sG= \{0\} \ands K_2 \sG^\prime=\{0\}.
\end{equation}
The following result, which is  the main theorem of this section, will be used in the later sections to derive our main theorems.

\begin{theorem}\label{thm:allsol1} Let $\{W,\wt{W},Z\}$ be a data set  for a LTONP interpolation problem with   $Z^*$  being  pointwise stable,  and assume that the Pick operator $\la $ is non-negative. Furthermore, let $\o$ be the unitary operator determined by the data set. Then  the LTONP interpolation problem is solvable  and  its solutions are given by
\begin{equation}\label{allsol1}
F(\l)=G_{11}(\l)+\l G_{12}(\l)\left(I_{\sZ_\circ}-\l G_{22}(\l)\right)^{-1}G_{21}(\l), \quad \l \in \BD,
\end{equation}
where
\begin{equation}\label{cond1}
G=\begin{bmatrix} G_{11}&G_{12}\\ G_{21}&G_{22} \end{bmatrix}
\in \sS(\sU\oplus \sZ_\circ, \sY\oplus \sZ_\circ) \hspace{.1cm} \mbox{and}\hspace{.15cm} G(0)|{\sF}=\o.
\end{equation}
Moreover, there is a one-to-one correspondence between the set of all solutions $F$ and the set of  all Schur class functions $G$ satisfying the two conditions in \eqref{cond1}.
\end{theorem}

\begin{remark}\label{rem:centrsol}
\textup{Let $G_\circ$  be the function identically equal to $\o P_\sF$. Then $G_\circ$ is a Schur class function, $G_\circ\in \sS(\sU\oplus \sZ_\circ, \sY\oplus \sZ_\circ)$, and $G_\circ(\l)| \mathcal{F}=\omega$ for each $\l \in \BD$. Thus the two conditions in \eqref{cond1} are satisfied for $G=G_\circ$. The corresponding  solution   $F$    is denoted by $F_\circ$ and referred to as  the \emph{central solution}. Note that $F_\circ\in S(\sU, \sY)$.}
\end{remark}

The following corollary is an immediate consequence of Theorem \ref{thm:allsol1} using  the definition of  the  central solution given above.

\begin{corollary}\label{cor:central}Let $\{W,\wt{W},Z\}$ be a data set  for a LTONP interpolation problem with   $Z^*$  being  pointwise stable,  and assume that the Pick operator $\la $ is non-negative. Furthermore, let $M_\circ$ be the  operator mapping $\sU\oplus \sZ_\circ$ into $\sY\oplus \sZ_\circ$ defined by  $M_\circ=\o P_\sF$, where $\o$ is the unitary operator determined by the data set. Write  $M_\circ$ as a $2\ts 2$ operator matrix as follows:
\[
M_\circ=\begin{bmatrix} \d_\circ &\g_\circ\\  \b_\circ &\a_\circ\end{bmatrix}
:\begin{bmatrix}\sU\\   \sZ_\circ\end{bmatrix}\to
\begin{bmatrix}\sY\\   \sZ_\circ\end{bmatrix}.
\]
Then the central solution $F_\circ$ is given by
\begin{equation}\label{eq:Fcirc1}
F_\circ(\l)=\d_\circ+\l\g_\circ(I_{\sZ_\circ}-\l \a_\circ)^{-1}\b_\circ, \qquad \l\in \BD.
\end{equation}
\end{corollary}

Since $M_\circ$ is a contraction, one calls the right side of  \eqref{eq:Fcirc1} a \emph{contractive realization}. The next example is  a trivial one to a certain extend, on the other hand it tells us how one can construct  a contractive realization    for any Schur class function.

\begin{example}\label{ex:SchurF}
\textup{Let $F$ be a Schur class function, $F\in  S(\sU, \sY)$, and let $T_F$ be the Toeplitz operator defined by $F$. Put $\sZ_1=\ell_+^2(\sY)$, and consider the operators
\begin{equation}\label{exdataset}
W_1=I_{\ell_+^2(\sY)}, \quad \wt{W}_1=T_F, \quad Z_1=S_\sY.
\end{equation}
Then
\[
Z_1W_1=S_\sY=W_1 S_\sY, \quad  Z_1\wt{W}_1=S_\sY T_F=T_F S_\sU= \wt{W}_1S_\sU.
\]
Thus $\{W_1, \wt{W}_1, Z_1\}$ is a data set for a LTONP  interpolation problem.  Moreover, $Z_1^*=S_\sY^*$, and hence $Z_1^*$ is pointwise stable. Note that $\Psi \in S(\sU, \sY)$ is a solution to the related LTONP  interpolation  problem  if and only if   $W_1 T_\Phi= \wt{W}_1$.  But   $W_1 T_\Phi= \wt{W}_1$  if and only if   $T_\Phi=T_F$. It follows that the LTONP  interpolation  problem for the data set $\{W_1, \wt{W}_1, Z_1\}$  is solvable, and the solution is unique, namely $\Phi=F$. But then $F$ is the central solution of the  LTONP  interpolation  problem for  the data set $\{W_1, \wt{W}_1, Z_1\}$, and Corollary~\ref{cor:central}  tells us that $F$ admits a representation of the
form
\begin{equation}\label{eq:F}
F (\l)=\d_1 +\l\g_1 (I_{\sZ_{1, \circ}}-\l \a_1)^{-1}\b_1, \qquad \l\in \BD.
\end{equation}
Moreover, the operator matrix $M_1$ defined by
\[
M_1=\begin{bmatrix} \d_1 &\g_1\\  \b_1 &\a_1\end{bmatrix}
:\begin{bmatrix}\sU\\   \sZ_{1, \circ}\end{bmatrix}\to
\begin{bmatrix}\sY\\   \sZ_{1, \circ}\end{bmatrix},
\]
is given by $M_1= \o_1 P_{\sF_1}$, where $\o_1$ is the unitary operator determined by the data set $\{W_1, \wt{W}_1, Z_1\}$. Since $M_1$ is a contraction, the right hand side is a contractive realization of $F$. Thus given any $F\in  S(\sU, \sY)$  Corollary~\ref{cor:central} provides a way to construct a contractive realization for $F$. Finally, it is noted that in this setting the
corresponding subspace $\sG_1^\prime:=\sG^\prime = \{0\}$, and thus, $M_1 = \omega_1 P_{\sF_1}$ is in fact a co-isometry. Indeed, to see that this is the case, note that $\sZ_{1,\circ}:=\sZ_\circ=\overline{\im (I-T_FT_F^*)^\half}$ and $\sF_{1}'$ is the closure of the range of
\[
\mat{c}{E_{\sY}^* I_{\ell^2_+(\sY)}\\ (I-T_FT_F^*)^\half S_\sY^*}
= \mat{cc}{I_\sY &0\\ 0& (I-T_FT_F^*)^\half}
\mat{c}{E_{\sY}^* \\ S_\sY^*}.
\]
Since the block column operator on the right hand side is unitary it follows that $\sF_1'$ is equal to the closure of the range of the $2 \times 2$ block operator on the right hand side, which equals $\sY\oplus \sZ_{1,\circ}$. Therefore, $\sG_1'=(\sY\oplus \sZ_{1,\circ})\ominus \sF_1'=\{0\}$, as claimed.
We shall come back to this construction  in Subsection \ref{Assec:co-iso} of the appendix.}
\end{example}

Describing the solution set of an interpolation problem with a map of the form \eqref{allsol1} with a restriction of $G$ equal to a constant unitary operator  is one  of ``standard'' methods    of parameterizing all solutions of interpolation problems. For instance, this type of formula is used in the description of all solutions to the commutant lifting theorem;   see   Section VI.5 in \cite{FFGK98}, where the unitary operator  $\o$ is defined  by  formula (2.2)  on page 265, the analogs of the spaces $\sF$, $\sF^\prime$, $\sG$,  $\sG^\prime$ appear on page 266, and the analog of the function $G$  is referred to as a Schur contraction.  Such maps are  also used to describe all solutions \ to the so-called abstract interpolation problem, cf., \cite{KKY07,Kh98}, and these are only a few of many instances.  The operator $\o$ is also closely related to the ``lurking isometry"   used in  \cite{BT98}, which has its roots in \cite{Sz.-nk}.

In the present paper  the  proof of Theorem \ref{thm:allsol1} is based purely on state space methods, using the theory of co-isometric realizations.  Therefore we first review some  notation, terminology and standard  facts from realization theory, including  the main theorem about observable, co-isometric realizations of Schur class functions.   The reader familiar with system theory may skip this subsection.

\subsection{Preliminaries from realization theory}\label{ssec:realization} We say that a quadruple of Hilbert space operators $\si=\{\alpha, \beta, \gamma, \delta\}$,
\[
\a:\sX\to \sX, \quad \b:\sU\to \sX, \quad \g:\sX\to \sY, \quad \d:\sU\to \sY,
\]
is a \emph{$($state space$)$ realization} for a function $F$ with values in $\mathcal{L}(\mathcal{U},\mathcal{Y})$ if
\begin{equation}\label{realF}
F(\lambda) = \delta +  \lambda \gamma (I - \lambda \alpha)^{-1} \beta
\end{equation}
for all $\l$ in some neighborhood of the origin.  The space  $\sX$  is called the \emph{state space} while $\mathcal{U}$ is the {\em input space} and $\mathcal{Y}$ is the {\em output space}. In systems theory $F$ is referred to as the \emph{transfer function} of the system $\si=\{\alpha, \beta, \gamma, \delta\}$.  Note  that $\{\alpha, \beta, \gamma, \delta\}$ is a realization for $F$  implies that $F$ is analytic in some neighborhood of the origin, and in that case the  coefficients
$\{F_n\}_{n=0}^\infty$ of the Taylor expansion of
$F(\lambda) = \sum_{n=0}^\infty \lambda^n F_n$ at zero are given by
\begin{equation}\label{Fn}
F_0=F(0) = \delta \ands F_n = \gamma \alpha^{n-1} \beta
\qquad (n \geq 1).
\end{equation}
The system $\si=\{\alpha, \beta, \gamma, \delta\}$ or the pair $\{\gamma, \alpha\}$ is said to be   \emph{observable} if $\cap_{n\geq 0} \kr \g\a^n=\{0\}$. Two systems $\{\alpha_1,  \beta_1, \gamma_1, \delta_1\}$ and $\{\alpha_2, \beta_2, \gamma_2, \delta_2\}$ with state spaces $\sX_1$ and $\sX_2$, respectively, are called \emph{unitarily equivalent} if $\delta_1 = \delta_2$ and there exists a unitary operator $U$ mapping $\mathcal{X}_1$ onto $\mathcal{X}_2$ such that
\[ \alpha_2 U = U \alpha_1, \quad \beta_2 = U \beta_1, \quad \gamma_2 U = \gamma_ 1.
\]
Clearly, two unitarily equivalent systems both realize the same transfer function $F$. Given a system $\si=\{\alpha, \beta, \gamma, \delta\}$ the $2\ts 2$ operator matrix $M_\si$  defined by
\begin{equation}\label{Msi}
M_\si = \begin{bmatrix}
      \delta & \gamma \\
      \beta  & \alpha \\
    \end{bmatrix}:
    \begin{bmatrix}
    \mathcal{U} \\
      \mathcal{X} \\
    \end{bmatrix}\rightarrow\begin{bmatrix} \mathcal{Y}\\\mathcal{X} \\ \end{bmatrix},
\end{equation}
is called the \emph{system matrix} defined by $\si$.  If the system matrix $M_\si$ is a contraction, then its transfer function  is a \emph{Schur class function}, $F\in \sS(\sU, \sY)$, that  is, $F$ is analytic on the open unit disc $\BD$ and $\sup _{\l \in \BD}\|F(\l)\|\leq 1$. The converse is also true. More precisely,  we have the following classical result.

\begin{theorem}\label{thm:co} A function $F$ is in $S(\mathcal{U},\mathcal{Y})$ if and only if $F$ admits an observable, co-isometric realization. Moreover, all observable, co-isometric realizations of  $F$ are unitarily equivalent.
\end{theorem}

The `if part'' of the  above theorem is rather straight forward to prove, the ``only if part''  is much less trivial and has  a long and interesting history, for example  involving  operator model theory (see \cite{sz-nf} and the revised and enlarged edition \cite{sz-nfBK}) or the theory of  reproducing kernel Hilbert spaces (see, \cite{deBrR66a} and \cite{deBrR66b}) . We also mention  Ando's Lecture Notes \cite[Theorem 3.9 and formulas (3.16), (3.17)]{Ando90}, and recent work in a multivariable setting due to J.A. Ball and co-authors  \cite{b}-- \cite{btv}.

An alternative new proof  of Theorem \ref{thm:co} is given in  Subsection \ref{Assec:co-iso} in the Appendix.

If  the system $\si=\{\alpha, \beta, \gamma, \delta\}$  has   a contractive system matrix, then
\begin{equation}\label{defGa0}
\ga:= \col \big[\g \a^j\big]_{j=0}^\iy  =
\begin{bmatrix}
  \gamma \\
  \gamma \alpha \\
  \gamma \alpha^2 \\
  \vdots \\
\end{bmatrix}
:  \sX\to \ell_+^2(\sY)
\end{equation}
is a well defined operator and $\ga$ is a contraction. This classical result  is  Lemma 3.1 in \cite{FFK02};  see also Lemma \ref{lem-obs} in the Appendix where the proof is given  for completeness.  We call $\ga$ the \emph{observability operator} defined by
$\si$, or simply by the pair $\{\g,\a\}$. Note that in this case $\si$ is observable if and only if $\ga$ is one-to-one. We conclude with the following lemma.

\begin{lemma}\label{lem:co-iso}
If $F\in \sS(\sU, \sY)$ has a  co-isometric realization, then $F$ has an  observable, co-isometric realization.
\end{lemma}

\begin{proof}[\bf Proof]
Assume that $\si=\{\a, \b, \g, \d\}$ is a  co-isometric realization of $F$.  Let $\ga$ be the observability operator defined \eqref{defGa0}, and let $\sX_0$ be the closure of the range of $\ga^*$. Thus $\sX= \sX_0\oplus \kr \ga$, and relative to this  Hilbert space direct sum the operators $\a$, $\b$, $\g$ admit the following partitions:
\begin{align*}
\a&=\begin{bmatrix}\a_0&0\\ \star &\star\end{bmatrix}:
\begin{bmatrix}\sX_0\\ \kr \ga  \end{bmatrix}\to \begin{bmatrix}\sX_0\\ \kr \ga  \end{bmatrix}, \quad \b=\begin{bmatrix}\b_0\\ \star \end{bmatrix}: \sU\to \begin{bmatrix}\sX_0\\ \kr \ga  \end{bmatrix},\\
\g&=\begin{bmatrix}\g_0 & 0 \end{bmatrix}:\begin{bmatrix}\sX_0\\ \kr \ga  \end{bmatrix}\to \sY.
\end{align*}
Then the system $\si_0=\{\a_0, \b_0, \g_0, \d\}$  is an observable realization of $F$.

The  system matrix $M_0 = M_{\si_0}$  for $\si_0$ is also co-isometric.
To see this notice that $M_\si$ admits
a matrix representation of the form
\[
M_\si =\begin{bmatrix}
 \d &\g_0&0\\
 \b_0&\a_0 &0\\
\star&\star&\star
 \end{bmatrix}
=\begin{bmatrix}M_0&0\\ \star&\star \end{bmatrix},\hspace{.2cm}
\mbox{and hence}
\hspace{.2cm}  M_\si M_\si^*=\begin{bmatrix}M_0M_0^*&\star\\ \star&\star \end{bmatrix}.
\]
Since $M_\si$ is a co-isometry, $M_\si M_\si^*$ is the identity operator on the space $ \sY\oplus \sX_0\oplus \kr \ga$, and thus  $M_0M_0^*$  is the identity operator on $ \sY\oplus \sX_0$.  Therefore, $\si_0=\{\a_0, \b_0, \g_0, \d\}$  is an observable, co-isometric realization of $F$.
\end{proof}

\subsection{Solutions of the  LTONP interpolation problem  and $\la$-preferable realizations}
As before $\{W,\widetilde{W},Z\}$ is a data set  for a LTONP interpolation problem,  and  we assume that the Pick operator $\Lambda$ is non-negative.

Let $\si=\{\a,\b,\g,\d\}$ be a co-isometric realization of $F$
with state space $\sX$ and system matrix $M= M_\si$. We call the realization \emph{$\la$-preferable} if  $\sX=\sZ_\circ\oplus \sV $ for some Hilbert space $\sV$ and $M|\sF=\o$.  As before,  $\sZ_\circ$  equals the closure of the range of $\Lambda$; see \eqref{defZcirc}, and $\o$ is the unitary operator from $\sF$ onto $\sF^\prime$ determined by the data set  $\{W,\widetilde{W},Z\}$. In particular, $\sF$ and $\sF^\prime$ are the subspaces of $\sU\oplus \sZ_\circ$ and $\sY\oplus \sZ_\circ$, respectively, defined by \eqref{defFFprime}.  Note that  $\sX=\sZ_\circ\oplus \sV $ implies that $\sF\subset \sU\oplus \sX$, and thus $M|\sF$ is well defined. Furthermore,  $M$ partitions as
\begin{equation}\label{M25}
M = \begin{bmatrix}
      \delta & \gamma_1 & \gamma_2 \\
      \beta_1 & \alpha_{11} & \alpha_{12} \\
      \beta_2 & \alpha_{21} & \alpha_{22} \\
    \end{bmatrix}:
    \begin{bmatrix}
    \mathcal{U} \\
      \mathcal{Z}_\circ \\
      \mathcal{V} \\
    \end{bmatrix}\rightarrow\begin{bmatrix} \mathcal{Y} \\
      \mathcal{Z}_\circ \\
      \mathcal{V} \\
    \end{bmatrix},
\end{equation}
and the constraint $M|\sF=\o$  is equivalent to
\begin{equation}\label{M25a}
\omega = \left.\begin{bmatrix}
      \delta & \gamma_1   \\
      \beta_1 & \alpha_{11}   \\
    \end{bmatrix}\right|\mathcal{F}.
\end{equation}
To see the latter, observe that $M|\sF=\o$ implies that $M\sF=\o\sF=\sF^\prime \subset \sY\oplus \sZ_\circ$, and hence $\begin{bmatrix}\b_2 & \a_{21}\end{bmatrix}|\sF=\{0\}$. Conversely, if \eqref{M25a} holds, then the restriction of the first two bock rows of $M$ in \eqref{M25} to $\sF$ is equal to $\o$. Since $\o$ is unitary, the restriction of the last block row to $\sF$ must be zero, for otherwise $M$ would not be a contraction. Hence $M|\sF=\o$.

The following theorem  is the main result of the present subsection.

\begin{theorem}\label{thm:allco} Let $\{W,\widetilde{W},Z\}$
be a data set  for a LTONP interpolation problem with
 $Z^*$  being  pointwise stable,  and assume that the Pick operator $\Lambda$ is non-negative.  Then all solutions $F$ of the LTONP interpolation problem are given by
\begin{equation}\label{allsol2}
F(\lambda)=
\delta + \lambda \gamma(I - \lambda \alpha)^{-1}\beta, \quad \lambda \in \mathbb{D},
\end{equation}
where $\si=\{\alpha,\beta,\gamma,\delta\}$ is an  observable, co-isometric realization of $F$ which is $\la$-preferable. Moreover,
\begin{equation}\label{Lam8}
 \Lambda = W \ga \ga^*W^*,
\end{equation}
where $\ga$ is the observability operator mapping
$\mathcal{X}$ into $\ell_+^2(\mathcal{Y})$ determined by $\{\gamma, \alpha\}$.
Finally,  up  to  unitary equivalence of  realizations  this parameterization of all solutions to the  LTONP interpolation problem  via  $\la$-preferable,  observable, co-isometric realizations $\si=\{\alpha,\beta,\gamma,\delta\}$   is  one-to-one  and onto.
\end{theorem}

\begin{remark}\label{rem:co}
\textup{If one  specifies  Theorem  \ref{thm:allco} for the case when the data set is the set $\{W_1, \wt{W}_1, Z_1\}$,  where   $W_1$,  $\wt{W}_1$ and $ Z_1$ are given by \eqref{exdataset}, then Theorem~\ref{thm:co} is obtained. Note however   that Theorem \ref{thm:co} is used  in the proof   of Theorem  \ref{thm:allco}, and therefore Theorem \ref{thm:co} does not appear as a  corollary  of Theorem  \ref{thm:allco}. On the other hand, if one uses the arguments  in  the proof  of   Theorem \ref{thm:allco} for the data set  $\{W_1, \wt{W}_1, Z_1\}$ only,   then one obtains a new direct proof  of the fact that any Schur class function admits an observable co-isometric realization.  This proof is given in  Subsection \ref{Assec:co-iso}; cf.,   Example \ref{ex:SchurF}.}
\end{remark}

The proof  of Theorem  \ref{thm:allco} will be based on two lemmas.

\begin{lemma}\label{lem:la-prefer1} Let $\{W,\widetilde{W},Z\}$ be a data set  for a  LTONP interpolation problem with $Z^*$  being  pointwise stable,  and assume that the Pick operator $\Lambda$ is non-negative. Let $F\in \sS(\sU, \sY)$, and assume that  $\si=\{\a,\b,\g,\d\}$  is a $\la$-preferable,  co-isometric realization  of $F$. Then $F$ is a solution to the LTONP interpolation problem. Moreover,
\begin{equation}\label{sqrt1}
\Lambda^{\frac{1}{2}}\Pi_{\mathcal{Z}_\circ} = W \ga,
\end{equation}
where $\ga$ is the observability operator  defined by $\{\g, \a\}$ and $\mathcal{Z}_\circ = \overline{\im \Lambda}$; see  \eqref{defZcirc}.
\end{lemma}

\begin{proof}[\bf Proof]
Using $Z W = WS_\mathcal{Y}$
and $S_\mathcal{Y}^* \ga  = \ga  \alpha$, we obtain
\[
W \ga  - Z W \ga  \alpha =
W\left(I - S_\mathcal{Y}S_\mathcal{Y}^* \right)\ga  =
W E_\sY E_\sY ^*\ga  = B \gamma.
\]
In other words,
\begin{equation}\label{lyap_obs}
W \ga  -Z W \ga  \alpha = B \gamma.
\end{equation}
Because $Z^*$ is pointwise stable, it follows that $W \ga $ is the unique solution to the Stein equation  $\Omega - Z \Omega \alpha = B \gamma$;
see Lemma \ref{lem_lyap} in the Appendix.

Since the system $ \si= \{\a, \b, \g, \d\}$ is $\la$-preferable,  we know that the state space $\sX$ is equal to $\sZ_\circ\oplus \sV$ for some Hilbert space $\sV$,  where $\sZ_\circ= \overline{\im \Lambda}$.
Let $\Pi_{\sZ_\circ}$ be the orthogonal projection from  $\sX=\sZ_\circ \oplus \sV$ onto $\sZ_\circ$.  We shall prove  that
\begin{equation}\label{Ax=b}
\begin{bmatrix}
  B  & Z \la^{\frac{1}{2}}\Pi_{\mathcal{Z}_\circ} \\
\end{bmatrix} \begin{bmatrix}
      \delta & \gamma \\
      \beta  & \alpha \\
    \end{bmatrix}  =  \begin{bmatrix}
  \widetilde{B}  &  \la^{\frac{1}{2}}\Pi_{\mathcal{Z}_\circ} \\
\end{bmatrix}.
\end{equation}
Let   $M=M_\si$ be the system matrix of the realization $\si$, i.e.,  the $2\ts 2$ operator matrix appearing in the left hand side of \eqref{Ax=b}. To prove  the identity \eqref{Ax=b}  we first note that the second identity in \eqref{defomega2} and  $M|\sF = \o$  imply that the two sides of \eqref{Ax=b} are equal when restricted to $\sF$. Next, consider the orthogonal complements
\begin{equation*}
\sF^\perp =(\sU\oplus \sZ_\circ \oplus \sV) \ominus \sF=\sG\oplus \sV \quad
\sF^{\prime \perp}=(\sY\oplus \sZ_\circ \oplus \sV) \ominus  \sF^\prime =\sG^\prime\oplus \sV.
\end{equation*}
Since  $M$ is a contraction with  $M\sF=\sF^\prime$ and  $M|\sF$ is unitary, we have $M\sF^\perp \subset \sF^{\prime \perp}$. Therefore it remains to show that the two sides of \eqref{Ax=b} are also equal when restricted to $\sF^\perp$.  To do this, take
$f=\begin{bmatrix} u_0 &z_0& v_0\end{bmatrix}{}^\top$ in $\sF^\perp$. Here $u_0\in \sU$, $z_0\in \sZ_\circ$, and $v_0\in \sV$. Then
\begin{equation}\label{step1}
\begin{bmatrix} \wt{B}  &  \la^{\frac{1}{2}}\Pi_{\sZ_\circ}
\end{bmatrix}f= \begin{bmatrix} \wt{B}  &  \la^{\frac{1}{2}}|{\sZ_\circ} &0\end{bmatrix}\begin{bmatrix} u_0 \\z_0\\ v_0 \end{bmatrix}= \begin{bmatrix} \wt{B}  &  \la^{\frac{1}{2}}\end{bmatrix}\begin{bmatrix} u_0 \\z_0\end{bmatrix}.
\end{equation}
But the vector $\begin{bmatrix} u_0 &z_0\end{bmatrix}{}^\top$ belongs to  the space $\sG$. Thus  the first identity in \eqref{zeroGGprime} shows that $\begin{bmatrix} \wt{B}  &  \la^{\frac{1}{2}}\Pi_{\sZ_\circ}  \end{bmatrix}f=0$.  Now consider $f^\prime:= Mf\in \sF^{\prime\perp}$. Write $f^\prime=\begin{bmatrix} y_0 &z_0^\prime & v_0^\prime \end{bmatrix}{}^\top$, where $y_0\in \sY$, $z_0^\prime\in \sZ_\circ$, and $v_0^\prime \in \sV$. Then
\[
\begin{bmatrix} B  & Z \la^{\frac{1}{2}}\Pi_{\sZ_\circ}\end{bmatrix} Mf=
\begin{bmatrix} B  & Z \la^{\frac{1}{2}}|\sZ_\circ&0\end{bmatrix}
\begin{bmatrix} y_0  \\z_0^\prime \\ v_0^\prime \end{bmatrix}
=\begin{bmatrix} B  & Z \la^{\frac{1}{2}}\end{bmatrix}\begin{bmatrix} y_0\\z_0^\prime\end{bmatrix}=0,
\]
because    $\begin{bmatrix} y_0  & z_0^\prime\end{bmatrix}{}^\top$  belongs to $\sG^\prime$ and using  the second  identity in \eqref{zeroGGprime}.
We conclude that  $\begin{bmatrix} B  & Z \la^{\frac{1}{2}}\Pi_{\sZ_\circ}\end{bmatrix} Mf=0$. Hence when applied to $f$ both sides of   \eqref{Ax=b} are equal to zero, which completes the proof of \eqref{Ax=b}.

Note that \eqref{Ax=b} is equivalent to the following two identities:
\begin{equation}\label{Ax=b1}
\la^{\frac{1}{2}}\Pi_{\mathcal{Z}_\circ} = Z \la^{\frac{1}{2}}\Pi_{\mathcal{Z}_\circ} \alpha + B \gamma
\quad \mbox{and}\quad \widetilde{B}= Z  \la^{\frac{1}{2}}\Pi_{\mathcal{Z}_\circ} \beta + B \delta.
\end{equation}
Because  $W\ga$ is the unique solution to the Stein equation \eqref{lyap_obs}, as observed above, the first identity in \eqref{Ax=b1} shows that $W\ga=\la^{\frac{1}{2}}\Pi_{\mathcal{Z}_\circ}$, i.e.,  the identity \eqref{sqrt1} holds true.

By consulting the second equation in \eqref{Ax=b1}, we have
  $\widetilde{B} = Z W \ga  \beta + B \delta$.
Using this we obtain
\[
W T_F E_\sU = \begin{bmatrix}
                            W E_\sY  & W  S_\sY\\
                          \end{bmatrix}\begin{bmatrix}
                            F(0)  \\ S_\mathcal{Y}^*T_F E_\sU\\
                          \end{bmatrix}
=\begin{bmatrix}
                            B  & Z W   \\
                          \end{bmatrix}\begin{bmatrix}
                            \delta  \\ \ga  \beta\\
                          \end{bmatrix}
= \widetilde{B}.
\]
Therefore $W T_F E_\sU = \widetilde{B} = \wt{W} E_\sU$.
So for any integer $n \geq 0$, we have
\[
W T_F S_\mathcal{U}^n E_\sU =
W S_\mathcal{Y}^n T_F  E_\sU=
Z^n W T_F  E_\sU = Z^n \widetilde{W} E_\sU
= \widetilde{W} S_\mathcal{U}^n E_\sU.
\]
Because $\{S_\mathcal{U}^n E_\sU \mathcal{U}\}_{n=0}^\infty$ spans
$\ell_+^2(\mathcal{U})$, we see that $W T_F = \widetilde{W}$.
Hence, $F$ is a solution to the LTONP interpolation problem.
\end{proof}

\begin{lemma}\label{lem:la-prefer2}  Let   $F$ be a solution to the  LTONP interpolation problem  with data set $\{W,\widetilde{W},Z\}$, and assume $\si=\{\a,\b\,\g,\d\}$ is  a co-isometric realization  of $F$.  Then up to unitary equivalence the realization  $\si$  is $\la$-preferable.
\end{lemma}

\begin{proof}[\bf Proof]
Throughout   $F(\lambda) = \delta + \lambda \gamma (I - \lambda \alpha)^{-1}\beta$ is a co-isometric  realization of the solution $F$  for the  LTONP interpolation problem with data set $\{W,\widetilde{W},Z\}$. We split the proof into three parts.

\smallskip\noindent
\textsc{Part 1.} In this part we show that
\begin{equation}\label{Ax=b3}
\begin{bmatrix} B & Z W \ga \end{bmatrix}
\begin{bmatrix} \delta & \gamma \\ \beta & \alpha \end{bmatrix} = \begin{bmatrix} \widetilde{B} &  W \ga \end{bmatrix}.
\end{equation}
To prove this equality, note that
\begin{align}
\begin{bmatrix} B & Z W \ga \end{bmatrix}
\begin{bmatrix}\d \\ \b\end{bmatrix}&=B\d+ZW\ga \b
=B\d +WS_\sY\ga \b \nn \\
&= WE_\sY \d+WS_\sY\ga \b = W\left(E_\sY \d+S_\sY\ga \b\right) \nn \\
&= W\begin{bmatrix}\d\\ \g\b \\  \g\a\b\\ \g\a^2\b \\ \vdots\end{bmatrix}=WT_F E_\sU=\wt{W}E_\sU=\wt{B}.\label{Ax=b3a}
\end{align}
Furthermore, we have
\begin{align}
\begin{bmatrix} B & Z W \ga \end{bmatrix}\begin{bmatrix}\g \\ \a \end{bmatrix}&=B \g+ZW\ga \a=WE_\sY\g+WS_\sY\ga \a  \nn \\
&=W\left(  E_\sY\g+S_\sY\ga \a\right)=W\ga.\label{Ax=b3b}
\end{align}
Together  the identities \eqref{Ax=b3a} and \eqref{Ax=b3b} prove the identity \eqref{Ax=b3}.

\smallskip\noindent
\textsc{Part 2.}
In this part we show that $W \ga  \ga ^* W^*$ is equal to the Pick operator $\la $.  Since the realization  $\{\a,\b,\g,\d\}$  is co-isometric, the corresponding system matrix is a co-isometry, and hence \eqref{Ax=b3} implies that
\[
\begin{bmatrix} \widetilde{B} &  W \ga \end{bmatrix}\begin{bmatrix} \wt{B}^*\\ \ga^* W^*\end{bmatrix}=\begin{bmatrix} B & Z W \ga \end{bmatrix}
\begin{bmatrix}B^* \\ \ga^* W^*Z^*\end{bmatrix}.
\]
Now put  $\om = W \ga  \ga ^* W^*$. Then the preceding identity is equivalent to
\[
\om - Z \om Z^* = B B^* - \wt{B} \wt{B}^*.
\]
Hence $\om$ is a solution to the Stein equation \eqref{Pick1a}. Since $Z^*$ is pointwise stable, the solution to this Stein equation is unique (see Lemma \ref{lem_lyap}), and thus, $\om = \la$.

\smallskip\noindent
\textsc{Part 3.}  In this part  we show that up to unitary equivalence the system $\si=\{\a, \b, \g, \d\}$ is $\la$-preferable. Let $\sX$ be the state space of  $\si$, and decompose $\sX$ as $\sX= \sX_\circ \oplus \sV$, where $\sV=\kr W\ga$. Since
\[
(W\ga)(W\ga)^*=\la=
 \Lambda^{\frac{1}{2}}\Pi_{\sZ_\circ} \Pi_{\sZ_\circ}^*\Lambda^{\frac{1}{2}}
\]
by the second part of the proof, the Douglas factorization lemma shows that there exists a unique unitary operator $\t_\circ$ mapping $\sZ_\circ$ onto $\sX_\circ$ such that
\begin{equation}\label{def:tau}
(W\ga|\sX_\circ)\t_\circ=\la^{\frac{1}{2}}|\sZ_\circ.
\end{equation}
Now,   put $\wt{\sX}= \sZ_\circ\oplus \sV$, let  $U$ be the unitary operator
from $\wt{\sX}$ onto $\sX$ defined by
\begin{equation}\label{defU1}
U=\begin{bmatrix}\t_\circ &0\\0&I_\sV \end{bmatrix}:
\begin{bmatrix}\sZ_\circ \\  \sV \end{bmatrix}\to
\begin{bmatrix}\sX_\circ\\ \sV \end{bmatrix},
\end{equation}
and define  the system $\wt{\si}=\{\wt{\a}, \wt{\b}, \wt{\g}, \wt{\d}\}$ by setting
\begin{equation}\label{wtsystem}
\wt{\a}=U^{-1}\a U, \quad \wt{\b}=U^{-1}\b, \quad \wt{\g}=\g U, \quad \wt{\d}=\d.
\end{equation}
Note that the systems $\si$ and $\wt{\si}$ are unitarily equivalent. Thus $\wt{\si}$ is a co-isometric realization of  $F$. Furthermore, the space $\sZ_\circ$ is a subspace of $\wt{\sX}$. Therefore in order to complete the proof  it remains to show that the system matrix $\wt{M}$ of the system $\wt{\si}$  has the following property:
\[
\wt{M}|\sF=\o.
\]
Here $\o$ is the unitary operator  determined by the given data set $\{W,\wt{W}, Z\}$. In particular, $\o: \sF \to \sF^\prime$, with  $\sF$ and  $\sF^\prime$ being defined by  \eqref{defFFprime}.

Let $M$ be the system matrix for $\si$. Multiplying  \eqref{Ax=b3} from the right by $M^*$, using the fact that $M$ is a co-isometry,   and taking adjoints, we see that
\begin{equation}\label{propM}
M\begin{bmatrix}\wt{B}^*\\ (W\ga)^* \end{bmatrix}= \begin{bmatrix}B^*\\ (W\ga)^*Z^* \end{bmatrix}.
\end{equation}
Note that $(W\ga)^*$ maps $\sZ$ into $\sX_\circ$. Hence taking adjoints in  \eqref{def:tau} and using that $\t_\circ$ is a unitary operator, we see that
\[
(W\ga)^*z=\t_\circ \la^{\frac{1}{2}}z \quad \mbox{for each $z\in \sZ$}.
\]
But then, using  the definition of $U$ in \eqref{defU1}, we obtain
\begin{equation}\label{BwtB-ident}
\begin{bmatrix}I_\sU  & 0 \\ 0  &U^{-1} \end{bmatrix}\begin{bmatrix}\wt{B}^*\\ (W\ga)^* \end{bmatrix}=\begin{bmatrix}\wt{B}^*\\ \la^{\frac{1}{2}} \end{bmatrix}\ands
\begin{bmatrix}I_\sY  & 0 \\ 0  &U^{-1} \end{bmatrix}
\begin{bmatrix}B^*\\ (W\ga)^*Z^* \end{bmatrix}=
\begin{bmatrix}B^*\\ \la^{\frac{1}{2}}Z^* \end{bmatrix}.
\end{equation}
From \eqref{wtsystem} it follows that
\[
\wt{M}\begin{bmatrix}I_\sU  & 0 \\ 0  &U^{-1} \end{bmatrix}=
\begin{bmatrix}I_\sY  & 0 \\ 0  &U^{-1} \end{bmatrix}
M.
 \]
 Using the later  identity and the ones in \eqref{propM} and \eqref{BwtB-ident} we see that
 \begin{align*}
 \wt{M} \begin{bmatrix}\wt{B}^*\\ \la^{\frac{1}{2}}\end{bmatrix}&=
 \wt{M}\begin{bmatrix}I_\sU  & 0 \\ 0  &U^{-1} \end{bmatrix}
\begin{bmatrix}\wt{B}^*\\ (W\ga)^* \end{bmatrix}=\begin{bmatrix}I_\sY  & 0 \\ 0  &U^{-1} \end{bmatrix}M\begin{bmatrix}\wt{B}^*\\ (W\ga)^* \end{bmatrix}\\
&=\begin{bmatrix}I_\sY  & 0 \\ 0  &U^{-1} \end{bmatrix} \begin{bmatrix}B^*\\ (W\ga)^*Z^* \end{bmatrix}=\begin{bmatrix}B^*\\ \la^{\frac{1}{2}} Z^*\end{bmatrix}.
\end{align*}
Now recall that $\o$ is the unique operator satisfying the first identity in \eqref{defomega2}.   Thus $\wt{M} $ and $\o$ coincide on $\sF$,  that is,  $\wt{M}|\sF=\o$.
\end{proof}

\begin{corollary}\label{cor:la-prefer1}
If  $F\in \sS(\sU, \sY)$ has  a $\la$-preferable, co-isometric   realization, then $F$ has  a $\la$-preferable, observable, co-isometric realization.
\end{corollary}

\begin{proof}[\bf Proof]
The fact that $F$ has  $\la$-preferable, co-isometric realization implies (use Lemma \ref{lem:la-prefer1}) that $F$ is a solution to the  LTONP interpolation problem. Moreover, from Lemma \ref{lem:co-iso} we know that $F$ has an observable, co-isometric realization. Since observability is preserved under unitarily equivalence, Lemma \ref{lem:la-prefer2} tells us that $F$ has  a $\la$-preferable, observable, co-isometric realization.
\end{proof}

\begin{proof}[\bf Proof of Theorem \ref{thm:allco}]
Let $\si$ be an observable, co-isometric system which is $\la$-preferable, and let $F$ be its transfer function. Then Lemma \ref{lem:la-prefer1}  tells is that $F$ is a solution to the LTONP interpolation problem. Moreover, since $\sZ_\circ$ is the closure of the range of $\la$, the identity \eqref{sqrt1} shows that
\[
W\ga \ga^*W^* =\Lambda^{\frac{1}{2}}\Pi_{\sZ_\circ}
\Pi_{\sZ_\circ}^*\Lambda^{\frac{1}{2}}
=\Lambda,
\]
which proves \eqref{Lam8}. Conversely, by  Theorem
\ref{thm:co} and  Lemma \ref{lem:la-prefer2}, if $F$ is a solution to the  LTONP interpolation problem, then $F$  has a $\la$-preferable, co-isometric realization. But then   $F$  also has a $\la$-preferable, observable, co-isometric realization by Corollary \ref{cor:la-prefer1}. Finally, by Theorem \ref{thm:co},  two observable, co-isometric realizations have the  same transfer function $F$ if and only if they are unitarily equivalent.  This proves that up to unitary equivalence the parametrization is one-to-one and onto.
\end{proof}

For later purposes, namely  the proof of Theorem \ref{thm:allsol1} in the next section,  we conclude this subsection with the following  corollary of Lemma \ref{lem:la-prefer1}.

\begin{corollary}\label{cor:uniequiv}
Let $F\in \sS(\sU, \sY)$, and let the systems $\si=\{\a, \b, \g, \d\}$ and $\si'=\{\a', \b', \g', \d \}$ be $\la$-preferable, co-isometric realizations of  $F$ with state spaces $\sX=\sZ_\circ\oplus \sV$ and $\sX'=\sZ_\circ\oplus \sV'$, respectively. If $U: \sX\to \sX'$ is a unitary operator such that
\begin{equation}\label{uniequiv2}
\a'U=U\a, \quad \b'=U\b, \quad \g'U=\g.
\end{equation}
Then $U|\sZ_\circ$ is the identity operator on $\sZ_\circ$ and $U\sV=\sV'$.
\end{corollary}

\begin{proof}[\bf Proof]
Let $\ga$ and $\ga'$ be the observability operators of $\si$ and $\si'$, respectively. From \eqref{uniequiv2} it follows that  $\ga' U=\ga$. Furthermore,  using  the identity  \eqref{sqrt1} for both  $\si$ and $\si'$ we see that
\[
\la^{\frac{1}{2}}\Pi_{\sZ_\circ}= W\ga \ands  \la^{\frac{1}{2}}\Pi_{\sZ_\circ}= W\ga'.
\]
Taking adjoints, it follows that  $U\Pi_{\sZ_\circ}^*\la^{\frac{1}{2}}=U\ga^*W^*= {\ga'}^*W^*=\Pi_{\sZ_\circ}^*\la^{\frac{1}{2}}$. Since the range of  $\la^{\frac{1}{2}}$ is dense in $\sZ_\circ$, we conclude that the operator $U$ acts as the identity operator on  $\sZ_\circ$, i.e.,  $U|\sZ_\circ=I_{\sZ_\circ}$. But then, using the fact that $U$ is unitary, we see that $U\sV=\sV'$.
\end{proof}

%:Sec3
\setcounter{equation}{0}
\section{Proof of Theorem \ref{thm:allsol1}}\label{sec: prfThm2.1}
In this section we prove Theorem \ref{thm:allsol1}.   Thus throughout  $\{W,\wt{W},Z\}$ is a data set  for a LTONP interpolation problem with   $Z^*$  being  pointwise stable,  and we assume that the Pick operator $\la $ is non-negative. Furthermore,  we use freely the notation and terminology introduced in the first three paragraphs of Section \ref{sec:ONP-coiso}. In particular, $\o$ is the unitary operator determined by the data set.

We begin with two lemmas. The first shows show how Schur class functions $F$ and $G$ that satisfy \eqref{allsol1} can be constructed from contractive realizations, and hence, in particular,  from  co-isometric realizations.

%%%%%%%%%%%%
\begin{lemma}\label{L:FGrel}
Let $M$ be a contractive linear operator  mapping $\sU\oplus \sZ_\circ\oplus \sV$ into $\sY\oplus \sZ_\circ\oplus \sV$, for some Hilbert space $\sV$, partitioned as in \eqref{M25}. Define
%\begin{equation}\label{GFforms}
\begin{align}
F(\lambda)& = \delta + \lambda \begin{bmatrix}
       \gamma_1 & \gamma_2 \\
    \end{bmatrix}\left(I - \lambda \begin{bmatrix}
        \alpha_{11} & \alpha_{12} \\
        \alpha_{21} & \alpha_{22}
    \end{bmatrix}\right)^{-1}\begin{bmatrix}
      \beta_1   \\
      \beta_2
    \end{bmatrix}, \nn \\
 G(\lambda)&=\mat{cc}{G_{11}(\l) & G_{12}(\l)\\ G_{21}(\l)& G_{22}(\l)} \nn \\
 &= \begin{bmatrix}
      \delta & \gamma_1  \\
      \beta_1 & \alpha_{11}
    \end{bmatrix} + \lambda
    \begin{bmatrix}
        \gamma_2 \\
        \alpha_{12}
    \end{bmatrix}\left(I - \lambda \alpha_{22}\right)^{-1}
    \begin{bmatrix}
      \beta_2 & \alpha_{21}
    \end{bmatrix}.\label{GFforms}
\end{align}
Then $F$, $G$ and the functions $G_{ij}$, $1\leq i,j\leq 2$, are Schur class functions, and
\begin{equation}\label{FGrel}
F(\l)=G_{11}(\l)+\l G_{12}(\l)\left(I_{\sZ_\circ}-\l G_{22}(\l)\right)^{-1}G_{21}(\l), \quad \l \in \BD.
\end{equation}
\end{lemma}

\begin{proof}[\bf Proof]
  Since $M$ is contractive,  the system matrices of the realizations of $F$ and $G$ in \eqref{GFforms} are also contractive, and hence $F$ and $G$ are Schur class functions. Note that the second identity in \eqref{GFforms} tells us that
\begin{align*}
G_{11}(\l)&=\d +\l \g_2\left(I-\l \a_{22}\right)^{-1}\b_2,\quad G_{12}(\l)=\g_1 +\l \g_2\left(I-\l \a_{22}\right)^{-1}\a_{21},\\
G_{21}(\l)&=\b_1 +\l \a_{12}\left(I-\l \a_{22}\right)^{-1}\b_2,\  G_{22}(\l)=\a_{11} +\l \a_{12}\left(I-\l \a_{22}\right)^{-1}\a_{21}.
\end{align*}
Again using $M$ is contractive, we see that the system matrices of the realizations of  $G_{ij}$, $1\leq i,j\leq 2$,  are also contractive, and hence the functions $G_{ij}$, $1\leq i,j\leq 2$,  are also Schur class functions.

Now let $F$ be given by  the first identity in \eqref{GFforms}. Fix $\l\in \BD$ and $u\in \sU$. Put $y=F(\l)u$, and define
\[
\begin{bmatrix} x_1\\ x_2\end{bmatrix}:=
\left(I_{{\sZ_\circ}\oplus \sV}-\l \begin{bmatrix} \a_{11}& \a_{12}\\ \a_{21}& \a_{22}\end{bmatrix}\right)^{-1}
\begin{bmatrix} \b_1\\ \b_2\end{bmatrix}u.
\]
Then the identity $F(\l)u=y$ is equivalent to the following three identities:
\begin{align}
&\hspace{2cm}y=\d u +\l \g_1 x_1+\l \g_2 x_2,\label{3ids1}\\
x_1&=\b_1 u +\l \a_{11} x_1+\l \a_{12} x_2, \quad x_2=\b_2 u +\l \a_{21} x_1+\l \a_{22} x_2.\label{3ids2}
\end{align}
The second  identity in \eqref{3ids2}  implies that
\begin{equation}\label{formx2}
x_2=\left(I-\l \a_{22}\right)^{-1}\b_2 u+\l \left(I-\l \a_{22}\right)^{-1}\a_{21}x_1.
\end{equation}
Inserting this formula for $x_2$ into the first  identity  in \eqref{3ids2} yields
\begin{align*}
x_1 &=\b_1 u +\l \a_{11} x_1+\l \a_{12}\left(I-\l \a_{22}\right)^{-1}\b_2 u+
\l^2\a_{12}\left(I-\l \a_{22}\right)^{-1}\a_{21}x_1\\
&=G_{21}(\l)u+\l \a_{11} x_1+\l\big(G_{22}(\l)x_1-\a_{11}x_1\big)\\
&=G_{21}(\l)u+\l G_{22}(\l)x_1,
\end{align*}
and thus
\begin{equation}\label{formx1}
x_1=\left(I-\l G_{22}(\l)\right)^{-1}G_{21}(\l)u.
\end{equation}
Using the identity \eqref{3ids1} together with the identities \eqref{formx2} and \eqref{formx1} we obtain
\begin{align*}
F(\l)u&=\d u +\l \g_1 x_1+\l \g_2 x_2\\
&= \d u +\l \g_1 x_1+ \l \g_2\left(I-\l \a_{22}\right)^{-1}\b_2 u +\l^2 \g_2\left(I-\l \a_{22}\right)^{-1}\a_{21}x_1\\
&=G_{11}(\l)u+\l\left(\g_1+  \l \g_2\left(I-\l \a_{22}\right)^{-1}\a_{21}\right)x_1\\
&= G_{11}(\l)u+\l G_{12}(\l)\left(I-\l G_{22}(\l)\right)^{-1}G_{21}(\l)u.
\end{align*}
Hence  \eqref{FGrel} holds as claimed.
\end{proof}

\begin{lemma}\label{lem:obssysts}  Let $M$ be a contractive linear operator  mapping $\sU\oplus \sZ_\circ\oplus \sV$ into $\sY\oplus \sZ_\circ\oplus \sV$, for some Hilbert space $\sV$, partitioned as in \eqref{M25}. Consider the systems
\begin{align}
\si&= \left\{\begin{bmatrix}
 \alpha_{11} & \alpha_{12} \\  \alpha_{21} & \alpha_{22}  \end{bmatrix},
\begin{bmatrix}\b_1\\ \b_2  \end{bmatrix}, \begin{bmatrix} \g_1&\g_2  \end{bmatrix},\d \right\}, \label{syst1}\\
\wt{\si}&=\left\{\a_{22}, \begin{bmatrix} \b_2&\a_{21}  \end{bmatrix},  \begin{bmatrix}\g_2\\ \a_{12}  \end{bmatrix},
\begin{bmatrix}\d & \g_{1} \\  \b_{1} & \alpha_{11}  \end{bmatrix}\right\}. \label{wtsyst1}
\end{align}
Then $\si$ is observable if and only if  $\wt{\si}$ is observable and \begin{equation}\label{partobs-cond}
\begin{bmatrix} \g_1 & \g_2\end{bmatrix}\begin{bmatrix}
 \alpha_{11} & \alpha_{12} \\  \alpha_{21} & \alpha_{22}  \end{bmatrix}^n
\begin{bmatrix} z \\v \end{bmatrix}=0 \quad (n=0,1,2,\ldots) \  \Longrightarrow \ z=0.
\end{equation}
\end{lemma}

\begin{proof}[\bf Proof]
We split the proof into two parts. In the first part we assume $\si$ is observable, and we prove that $\wt{\si}$ is observable and  that \eqref{partobs-cond} holds. The second part deals with the reverse implication.

\smallskip
\noindent\textsc{Part 1.} Let $\si$ be observable. In that case the identities on the left side of the arrow in \eqref{partobs-cond} imply that $z=0$ and $v=0$. In particular, the implication in \eqref{partobs-cond} holds. To see that $\wt{\si}$ is observable, fix a $v\in \sV$, and assume that
\[
\begin{bmatrix} \g_2\\ \a_{12}  \end{bmatrix}\a_{22}^n v=0, \quad n=0,1,2, \ldots.
\]
In other words, we assume that
\begin{equation}\label{wt-obs1}
\g_2 \a_{22}^n v=0\ands \a_{12}\a_{22}^n v=0, \quad n=0,1,2, \ldots.
\end{equation}
We want to show that $v=0$. We first show that
\begin{equation}\label{wt-obs2}
\begin{bmatrix}
 \alpha_{11} & \alpha_{12} \\  \alpha_{21} & \alpha_{22}  \end{bmatrix}^n\begin{bmatrix}0  \\v \end{bmatrix}= \begin{bmatrix} 0 \\\a_{22}^n v \end{bmatrix}, \quad n=0,1,2, \ldots.
\end{equation}
For $n=0$ the statement is trivially true. Assume  that the identity in \eqref{wt-obs2} holds for some integer $n\geq 0$. Then, using the second part of  \eqref{wt-obs2}, we obtain
\[
\begin{bmatrix}
 \alpha_{11} & \alpha_{12} \\[.1cm]  \alpha_{21} & \alpha_{22}  \end{bmatrix}^{n+1}\begin{bmatrix}0  \\v \end{bmatrix}=\begin{bmatrix}
 \alpha_{11} & \alpha_{12} \\[.1cm]   \alpha_{21} & \alpha_{22}  \end{bmatrix}
 \begin{bmatrix} 0 \\[.1cm] \a_{22}^n v \end{bmatrix}=\begin{bmatrix} \a_{12} \a_{22}^n v \\[.1cm] \a_{22}^{n+1} v \end{bmatrix}=\begin{bmatrix}0\\[.1cm] \a_{22}^{n+1} v \end{bmatrix}.
\]
By induction \eqref{wt-obs2}  is proved.  Using the second part of  \eqref{wt-obs2}, we conclude that
\[
\begin{bmatrix} \g_1 & \g_2\end{bmatrix}\begin{bmatrix}
 \alpha_{11} & \alpha_{12} \\  \alpha_{21} & \alpha_{22}  \end{bmatrix}^n \begin{bmatrix}0  \\v \end{bmatrix}=\begin{bmatrix} \g_1 & \g_2\end{bmatrix}\begin{bmatrix} 0 \\\a_{22}^n v \end{bmatrix}= \g_2 \a_{22}^n v=0, \ \  n=0,1,2,\ldots.
\]
Since the system $\si$ is observable, we conclude that $v=0$, and hence $\wt{\si}$ is observable.

\smallskip
\noindent\textsc{Part 2.} Assume that $\wt{\si}$ is observable and that \eqref{partobs-cond} holds. Let $\ga$ be the observability operator defined by $\si$. Thus
\[
\ga=\begin{bmatrix}
\g\\
\g \a\\
\g\a^2\\
\vdots
\end{bmatrix}: \sX\to \ell_+^2(\sY), \quad\mbox{where $\sX=\sZ_\circ\oplus \sV$ and}
\]
\[
\g =\begin{bmatrix} \g_1 & \g_2\end{bmatrix}: \begin{bmatrix}\sZ_\circ\\ \sV\end{bmatrix}\to \sY, \quad \a =\begin{bmatrix}
 \alpha_{11} & \alpha_{12} \\  \alpha_{21} & \alpha_{22}  \end{bmatrix}:
 \begin{bmatrix}\sZ_\circ \\ \sV\end{bmatrix}\to \begin{bmatrix}\sZ_\circ\\ \sV \end{bmatrix}.
 \]
Since $M$ is a contraction, the operator $\ga$ is a  well defined contraction; see Lemma \ref{lem-obs}.   We want to prove that $\ga$ is one-to-one.

Let $x=z\oplus v\in \kr \ga$. Then condition  \eqref{partobs-cond} tells us that $z=0$.  Thus $\kr \ga \subset \sV$.  It remains to prove that $v=0$.

Observe that $S_\sY^*\ga=\ga \a $. Thus $\a^n x \in  \kr \ga\subset \sV$ for each $n=0,1,2, \ldots$ which, by induction,  implies that
\begin{equation}\label{alpha-v1}
\a^n\begin{bmatrix}0\\ v \end{bmatrix}= \begin{bmatrix}0\\ \a_{22}^n v \end{bmatrix}, \quad n=0,1,2, \ldots.
\end{equation}
We see that
\[
0=\g \a^n\begin{bmatrix}0\\ v \end{bmatrix}=\begin{bmatrix}\g_1 &\g_2   \end{bmatrix}\begin{bmatrix}0\\ \a_{22}^n v \end{bmatrix}=\g_2   \a_{22}^n v, \quad n=0,1,2, \ldots.
\]
Furthermore, again using \eqref{alpha-v1}, we have $\a_{12}\a_{22}^n v=0$ for each $n\geq 0$. Thus
\begin{equation}\label{alpha-v2}
\begin{bmatrix}\g_2 \\ \a_{12}   \end{bmatrix}\a_{22}^n v=0 , \quad n=0,1,2, \ldots.
\end{equation}
But, by assumption,  $\wt{\si}$ is observable. Thus \eqref{alpha-v2} implies that $v=0$, as desired.
\end{proof}

\begin{proof}[\bf Proof of Theorem \ref{thm:allsol1}.]
First assume $F\in\sS(\sU,\sY)$ is a solution to the LTONP  interpolation problem. By Theorem \ref{thm:co}, the function $F$ admits an observable co-isometric realization $\si=\{\a,\b,\g,\d\}$. Since $F$ is a solution of the LTONP  interpolation  problem, by Lemma \ref{lem:la-prefer2}, the realization $\si$ is $\la$-preferable, up to unitary equivalence. Hence, we may assume $\si$ is $\la$-preferable. This implies that the system matrix $M$ of $\si$ has a decomposition as in \eqref{M25} and $M|\sF=\o$. Now define $G$ as in \eqref{GFforms}. Then, by Lemma \ref{L:FGrel}, the function $F$ is given by \eqref{allsol1}. Moreover, since the constraint $M|\sF=\o$ is equivalent to \eqref{M25a} the fact that $M|\sF=\o$ implies $G(0)|\sF=\o$.

Conversely, assume $G\in \sS(\sU\oplus\sZ_\circ,\sY\oplus\sZ_\circ)$ with $G(0)|\sF=\o$. We show that $F$ given by \eqref{allsol1} is a solution to the LTONP interpolation  problem. Let $\wtil{\si}=\{\wtil{\a},\wtil{\b},\wtil{\g},\wtil{\d}\}$ be an observable co-isometric realization of $G$ with state space $\sV$. Then $\wtil{\d}|\sF=G(0)|\sF=\o$. Note that the system matrix $\wtil{M}$ of $\wtil{\si}$ admits a decomposition as in \eqref{M25}, that is,
\[
\wtil{M}=\mat{cc}{\wtil{\d} & \wtil{\g}\\ \wtil{\b} & \wtil{\a}}
=\mat{cc|c}{\d&\g_1&\g_2\\ \b_1&\a_{11}&\a_{12}\\ \hline \b_2&\a_{21}&\a_{22}}
:\mat{c}{\sU\\\sZ_\circ\\ \sV}\to\mat{c}{\sY\\ \sZ_\circ\\ \sV}
\]
By Lemma \ref{L:FGrel} we obtain that the system
\begin{equation}\label{realF2}
\si=\left\{\begin{bmatrix}\a_{11}&\a_{12}\\ \a_{21} & \a_{22} \end{bmatrix},\begin{bmatrix}\b_1\\ \b_2 \end{bmatrix},\begin{bmatrix} \g_1 & \g_2\end{bmatrix},\d\right\}
\end{equation}
is  a  co-isometric realization for the function $F\in\sS(\sU,\sY)$ given by \eqref{allsol1}. Furthermore, $\wtil{\d}|\sF=\o$ together with the fact that $\o$ is unitary and $\wtil{M}$ a co-isometry, implies that $\wtil{M}|\sF=\o$. Hence $\si$ is a $\la$-preferable realization. Then, by Lemma \ref{lem:la-prefer1}, it follows that $F$ given by \eqref{allsol1} is a solution to the LTONP  interpolation  problem.

It remains to show that in the characterization of the solutions to the LTONP interpolation  problem given in Theorem \ref{thm:allsol1}, the functions $F$ and $G$ determine each other uniquely. Clearly, $F$ is uniquely determined by $G$ via \eqref{allsol1}. Thus the proof is complete when we show that for each solution $F$ there exists a unique $G$ as in \eqref{cond1} such that \eqref{allsol1} holds.

As in the second paragraph of the present proof, let $G$ be in the Schur class $\sS(\sU\oplus\sZ_\circ,\sY\oplus\sZ_\circ)$ with $G(0)|\sF=\o$, and let the system
\[
\wt{\si}=\left\{\a_{22}, \begin{bmatrix} \b_2&\a_{21}  \end{bmatrix},  \begin{bmatrix}\g_2\\ \a_{12}  \end{bmatrix},
\begin{bmatrix}\d & \g_{1} \\  \b_{1} & \alpha_{11}  \end{bmatrix}\right\}
\]
be an observable co-isometric  realization  of  $G$. Define $F$ by \eqref{allsol1}. Then  the system \eqref{realF2} is a $\la$-preferable co-isometric realization of $F$. We claim that this realization is  also observable.  To see this, we use the identity \eqref{sqrt1}. Taking adjoints in \eqref{sqrt1} we see that $\sZ_\circ\subset \overline{\im \ga^*}$, where $\ga$ is the observability operator defined by the pair $\{\g,\a\}$, i.e., as in \eqref{defGa0}, and hence $\kr \ga\subset \sV$. In other words, condition \eqref{partobs-cond} in Lemma \ref{lem:obssysts} is satisfied. But then, since $\wt{\si}$ is observable,  using Lemma \ref{lem:obssysts},  we conclude that the system $\si$ is also observable.

Now assume $G'\in \sS(\sU\oplus\sZ_\circ,\sY\oplus\sZ_\circ)$ with $G'(0)|\sF=\o$ is such that $F$ is also given by \eqref{allsol1} with $G$ replaced by $G'$. Let
\[
\wt{\si}'=\left\{\a_{22}', \begin{bmatrix} \b_2'&\a_{21}'  \end{bmatrix},  \begin{bmatrix}\g_2'\\ \a_{12}'  \end{bmatrix},
\begin{bmatrix}\d' & \g_{1}' \\  \b_{1}' & \alpha_{11}'  \end{bmatrix}\right\}
\]
be an observable co-isometric realization for $G'$. Then
\begin{equation*}
\si'=\left\{\begin{bmatrix}\a_{11}'&\a_{12}'\\ \a_{21}' & \a_{22}' \end{bmatrix},\begin{bmatrix}\b_1'\\ \b_2' \end{bmatrix},\begin{bmatrix} \g_1' & \g_2'\end{bmatrix},\d'\right\}
\end{equation*}
is a $\la$-preferable co-isometric realization for $F$, which is observable by the same argument as used for $\si$. Since all observable, co-isometric realizations of $F$ are unitarily equivalent, by Theorem \ref{thm:co}, we obtain that there exists a unitary operator $U$ from the state space $\sZ_\circ\oplus\sV$ of $\si$ to the state space $\sZ_\circ\oplus \sV'$ of $\si'$ such that \eqref{uniequiv2} holds, where
\begin{align*}
  &\a=\mat{cc}{\a_{11}& \a_{12}\\ \a_{21} &\a_{22}},\quad   \b=\mat{c}{\b_1\\\b_2},\quad   \g=\mat{cc}{\g_1& \g_2}, \\
  & \a'=\mat{cc}{\a_{11}'& \a_{12}'\\ \a_{21}' &\a_{22}'},\quad   \b'=\mat{c}{\b_1'\\\b_2'},\quad  \g'=\mat{cc}{\g_1' & \g_2'}.
\end{align*}
By Corollary \ref{cor:uniequiv}, we obtain that $U|\sZ_\circ =I_{\sZ_\circ}$ and $U$ maps $\sV$ onto $\sV'$. Let $\wt{U}=U|\sV:\sV\to\sV'$. Then \eqref{uniequiv2} takes the form
\begin{align*}
&\mat{cc}{\a_{11}& \a_{12}\\ \wt{U}\a_{21} &\wt{U}\a_{22}}
=\mat{cc}{\a_{11}'& \a_{12}'\wt{U}\\ \a_{21}' &\a_{22}'\wt{U}}, \quad
\mat{c}{\b_1'\\\b_2'}=\mat{c}{\b_1\\ \wt{U}\b_2},\\
&\hspace{2.4cm}\mat{cc}{\g_1' & \g_2'\wt{U}} = \mat{cc}{\g_1& \g_2}.
\end{align*}
This yields
\[
\wt{U}\a_{22}=\a_{22}'\wt{U},\quad
\wt{U}\mat{cc}{\b_2 & \a_{21}}=\mat{cc}{\b_2 & \a_{21}},\quad
\mat{c}{\g_2\\\a_{12}}= \mat{c}{\g_2'\\ \a_{12}'}\wt{U}.
\]
However, this shows that the realizations $\wt{\si}$ and $\wt{\si}'$ of $G$ and $G'$, respectively, are unitarily equivalent. Hence $G=G'$. We conclude that there exists only one $G\in \sS(\sU\oplus\sZ_\circ,\sY\oplus\sZ_\circ)$ with $G(0)|\sF=\o$ such that $F$ is also given by \eqref{allsol1}.
\end{proof}

We conclude this section with the construction of an observable co-isometric realization of the central solution $F_\circ$ introduced in Remark \ref{rem:centrsol}. Decompose $\o P_\sF$ as
\begin{equation}\label{sys-max}
\omega P_\mathcal{F} =
\begin{bmatrix} \delta_\circ & \gamma_\circ \\
 \beta_\circ & \alpha_\circ \end{bmatrix}
:\begin{bmatrix}
        \mathcal{U}         \\
        \mathcal{Z}_\circ
      \end{bmatrix} \rightarrow
      \begin{bmatrix}
        \mathcal{Y}         \\
        \mathcal{Z}_\circ
        \end{bmatrix}.
\end{equation}
Then we know from \eqref{eq:Fcirc1} in Corollary \ref{cor:central}   that
\begin{equation}\label{sys-max1}
F_\circ(\lambda) =
\delta_\circ + \lambda \gamma_\circ(I - \lambda \alpha_\circ)^{-1} \beta_\circ.
\end{equation}
However, \eqref{sys-max} does not provide an observable co-isometric realization of $F_\circ$.

\begin{lemma}\label{L:centreal}
Assume that the Pick operator $\la$ is non-negative. Let $\o P_\sF$ decompose as in \eqref{sys-max},  and define
\begin{equation}\label{def:Mcirc}
M=\mat{cc}{\d & \g\\ \b & \a}=\mat{c|cc}{\d_\circ&\g_\circ&  \Pi_\sY E_{\sG'}^*\\
\hline\b_\circ & \a_\circ & \Pi_{\sZ_\circ}E_{\sG'}^* \\
0&0&S_{\sG'}^*}
:\mat{c}{\sU\\ \hline \sZ_\circ\\ \ell^2_+(\sG')}\to \mat{c}{\sY\\ \hline \sZ_\circ\\ \ell^2_+(\sG')}.
\end{equation}
Here $\Pi_{\sZ_\circ}$ and $\Pi_\sY$ are the orthogonal projections of $\sY\oplus \sZ_\circ$ onto $\sZ_\circ$ and $\sY$ respectively.  Then $\{\a,\b,\g,\d\}$ is a $\la$-preferable  observable  co-isometric realization of $F_\circ$.  Moreover, $\kr M=\sG$.
\end{lemma}

\begin{proof}[\bf Proof]
Since $\sF\oplus \sG= \sU\oplus \sZ_\circ$ and $\sF^\prime\oplus \sG^\prime= \sY\oplus \sZ_\circ$, the  system matrix $M$ can be rewritten as
\begin{equation}\label{sys-max2}
M   = \begin{bmatrix}
        \omega & 0  & 0 \\
        0      & 0  &E_{\sG'}^* \\
        0      & 0  & S_{\mathcal{G}^\prime}^*
      \end{bmatrix}:  \begin{bmatrix}
        \mathcal{F}         \\
        \mathcal{G} \\
        \ell_+^2(\mathcal{G}^\prime)
      \end{bmatrix} \rightarrow
 \begin{bmatrix}
        \mathcal{F}^\prime         \\
        \mathcal{G}^\prime \\
        \ell_+^2(\mathcal{G}^\prime)
      \end{bmatrix}.
\end{equation}
The fact that
\begin{equation}\label{uni-ops}
\o:\sF\to\sF' \ands \mat{c}{E_{\sG'}^* \\S_{\mathcal{G}^\prime}^*}:\ell^2_+(\sG')\to \mat{c}{\sG'\\ \ell^2_+(\sG')}
\end{equation}
are both unitary maps,  implies that $M$ is a co-isometry. Moreover, we have
\begin{align*}
&\delta_\circ + \lambda \begin{bmatrix}
         \gamma_\circ  & \Pi_\sY E_{\sG'}^*  \\
      \end{bmatrix}
      \left(I - \lambda \begin{bmatrix}
        \alpha_\circ  &\Pi_{\sZ_\circ}E_{\sG'}^*\\
        0  & S_{\mathcal{G}^\prime}^* \\
      \end{bmatrix}\right)^{-1}
      \begin{bmatrix}
        \beta_\circ  \\
        0     \\
      \end{bmatrix} =\\
&\hspace{4cm}= \delta_\circ + \lambda \gamma_\circ(I - \lambda \alpha_\circ)^{-1} \beta_\circ
      = F_\circ(\lambda).
\end{align*}
Here $\Pi_{\sZ_\circ}$ is the orthogonal projection from $\sY\oplus\sZ_\circ=\sF'\oplus \sG'$ onto the subspace $\sZ_\circ$.
Hence $M$ is the system matrix of a co-isometric realization of $F_\circ$. It is also clear from \eqref{sys-max2} that the realization $\si=\{\a,\b,\g,\d\}$ of $F_\circ$ is $\la$-preferable.

To prove that $\si$ is observable, let  $\ga$  be the observability operator for the pair $\{\gamma, \alpha\}$. Note that
\[
\ga^*=\begin{bmatrix} \g^*&\a^* \g^*&(\a^* )^2\g^*& \cdots\end{bmatrix}:\ell_+^2(\sY)\to \begin{bmatrix}\sZ_\circ\\ \ell_+^2(\sG') \end{bmatrix}.
\]
Furthermore, we have
\[
\g^*=  \begin{bmatrix} \g_\circ^*\\ E_{\sG^\prime}\Pi_{\sG^\prime} \end{bmatrix}: \sY\to \begin{bmatrix}\sZ_\circ\\ \ell^2_+(\sG')\end{bmatrix}, \quad \a^*=
  \begin{bmatrix}\a_\circ^*&0\\ E_{\sG'}\Pi_{\sG'} &S_{\sG^\prime}\end{bmatrix} \hspace{.2cm}\mbox{on}\hspace{.2cm}
   \begin{bmatrix}\sZ_\circ\\ \ell^2_+(\sG')\end{bmatrix}.\]
Here $\Pi_{\sG^\prime}$ is the orthogonal projection from $\sY\oplus\sZ_\circ=\sF'\oplus \sG'$ onto the subspace $\sG^\prime$.
Let   $\mathcal{X}_\tu{obs}$ be the closure of the range of  $\ga^*$.  We have to prove that  $\sX_\tu{obs}=\sZ_\circ\oplus  \ell^2_+(\sG')$.

Next observe  that  $\mathcal{X}_\tu{obs}$ is an invariant subspace for $\alpha^*$. By Lemma \ref{lem:la-prefer1} we have $\Pi_{\mathcal{Z}_\circ}^*\Lambda^{\frac{1}{2}} = \ga^* W^* $.  From the latter identity together with the fact that  the range of $\la$ is dense in $\sZ_\circ$,  we conclude that
$\mathcal{Z}_\circ$ is a subspace of  $\mathcal{X}_\tu{obs}$.   It follows that
\[
\g_\circ^* y\in \sZ_\circ\subset  \sX_\tu{obs} \ands \g_\circ^* y+E_{\sG'}\Pi_{\sG'}y=\g^*y \in \sX_\tu{obs},  \quad y\in \sY.
\]
These inclusions show that $E_{\sG'}\Pi_{\mathcal{G}^\prime} \mathcal{Y}$ is a subset of $\mathcal{X}_\tu{obs}$.   Next we prove that $E_{\sG'}\Pi_{\sG^\prime} \sZ_\circ$  is a subset of $\mathcal{X}_\tu{obs}$. To do this recall that
$\sX_\tu{obs}$ is invariant under the operator $\a^*$. But then the relation $\mathcal{Z}_\circ \subset \mathcal{X}_\tu{obs}$
implies that  $\alpha^*\mathcal{Z}_\circ$ is a subset of $\mathcal{X}_\tu{obs}$.
Hence
\begin{equation}\label{incluXobs1}
\begin{bmatrix}
  \{0\} \\
  E_{\sG'}\Pi_{\sG'}\sZ_\circ\\
\end{bmatrix}\subset \sX_\tu{obs}\bigvee \begin{bmatrix}
  \alpha_\circ^*  \\
 E_{\sG'}\Pi_{\sG'} \\
\end{bmatrix} \sZ_\circ \subset
\sX_\tu{obs} \bigvee \alpha^*\sZ_\circ  \subset \sX_\tu{obs}.
\end{equation}
Here $\sL\bigvee\sK$ denotes the closure of the linear hull of  the linear spaces $\sL$ and $\sK$.   We know now that  both $E_{\sG'}\Pi_{\sG'}\sY$ and $E_{\sG'}\Pi_{\sG'}\sZ_\circ$  are contained in $\sX_\tu{obs}$.
Hence $\{0\} \oplus E_{\mathcal{G}^\prime} \mathcal{G}^\prime$ is a subspace for $\mathcal{X}_\tu{obs}$. But then
\[
\mathcal{X}_\tu{obs} \supset \bigvee_{n=0}^\infty \alpha^{*n}\begin{bmatrix}
  \{0\} \\
 E_{\sG^\prime} \mathcal{G}^\prime\\
\end{bmatrix}  = \bigvee_{n=0}^\infty \begin{bmatrix}
  \{0\} \\
 S_{\mathcal{G}^\prime}^n E_{\sG^\prime} \mathcal{G}^\prime\\
\end{bmatrix} = \bigvee_{n=0}^\infty \begin{bmatrix}
  \{0\} \\
 \ell_+^2(\mathcal{G}^\prime)\\
\end{bmatrix}.
\]
So $\mathcal{X}_\tu{obs}$ contains the whole state space
$\mathcal{Z}_\circ \oplus \ell_+^2(\mathcal{G}^\prime)$.
Therefore $\{\gamma,\alpha\}$ is observable, and $M$ is
a $\la$-preferable observable co-isometric systems matrix.

Finally, from \eqref{sys-max2} and the fact that the operators in \eqref{uni-ops} are unitary it follows that $\kr M= \sG$.
\end{proof}

%:Sec4
\setcounter{equation}{0}
\section{The case when the Pick operator is strictly positive and the proof of Theorem~\ref{thm:main1}}\label{sec:strictlypos}

In this section we prove Theorem \ref{thm:main1}. Throughout $\{W, \wt{W}, Z\}$  is a data set for a LTONP interpolation problem,  and we assume that the Pick operator $\la$ is strictly positive. We start with a lemma that  proves the first statements in Theorem \ref{thm:main1} and presents a useful formula for the unitary operator  $\o$ determined by the  data set $\{W, \wt{W}, Z\}$.

%%%%%%%%%%%%%%%%%%%%%%%%%
\begin{lemma}\label{L:useobs1}
Let $\{W, \wt{W}, Z\}$ be a data set for a LTONP interpolation problem, and assume that the Pick operator $\la$ is strictly positive. Then
\begin{itemize}
\item[\textup{(i)}] $P$ is strictly positive and  $\la^{-1}-P^{-1}$ is nonnegative,
\item[\textup{(ii)}] $Z^*$ is pointwise stable, in particular, its spectral radius is less than  or equal to one,
\item[\textup{(iii)}] $\widetilde{B}\widetilde{B}^*+\la=BB^*+Z\la Z^*$ and this operator is strictly positive.
\end{itemize}
Moreover,  the unitary operator $\o:\sF\to\sF'$ determined by the data set $\{W, \wt{W}, Z\}$ is given by
\begin{equation}\label{OmPF}
\o P_\sF=\begin{bmatrix}B^*\\  \la^{\frac{1}{2}}Z^*   \end{bmatrix}K
\begin{bmatrix}\wt{B} & \la^{\frac{1}{2}}\end{bmatrix}: \begin{bmatrix}\sU\\ \sZ \end{bmatrix}\to  \begin{bmatrix}\sY\\ \sZ \end{bmatrix}
\end{equation}
with $K=(BB^*+Z\la Z^*)^{-1}=(\widetilde{B}\widetilde{B}^*+\la)^{-1}$.
\end{lemma}

\begin{proof}[\bf Proof]
Since $\la =P-\wt{P}$ is strictly positive and $\wt{P}\geq 0$, we have $P=\la +\wt{P}\geq \la$. Thus $P \geq \la$, and the operator $P$ is  also strictly positive. But then  $P \geq \la$ implies $\la^{-1}\geq P^{-1}$. To see this, note that $P \geq \la$ yields $I-P^{-\frac{1}{2}}\la P^{-\frac{1}{2}}\geq 0$, and hence $\la^{\frac{1}{2}}P^{-\frac{1}{2}}$ is a contraction. Taking adjoints, we see that
$P^{-\frac{1}{2}}\la^{\frac{1}{2}}$ is also a contraction, and thus $I-\la^{\frac{1}{2}}P^{-1}\la^{\frac{1}{2}}$ is non-negative. Multiplying both sides with $\la^{-\frac{1}{2}}$  we obtain $\la^{-1}\geq P^{-1}$ as desired. Finally, note  that $\Lambda^{-1} - P^{-1}$ is not necessarily strictly positive. For example, choose $\widetilde{W} =0$, then $\Lambda = P$ and $\Lambda^{-1} - P^{-1}=0$.

To see that $Z^*$ is pointwise stable, note that $P=WW^*$ is strictly positive by item (i).  From $ZW=WS_\sY$ it follows that $S_\sY^*W^*=W^*Z^*$.
Because $P=WW^*$ is strictly positive, $\|W^* x\|^2 = (WW^* x,x) \geq \epsilon \|x\|^2$
for some $\epsilon >0$ and all $x$ in $\mathcal{Z}$. Thus the range
$\sH$ of $W^*$ is closed and $W^*$ can be viewed as and invertible
operator from $\mathcal{Z}$ onto $\sH$. In particular, the identity
$S_\sY^*W^*=W^*Z^*$ shows that $\sH$ is an invariant subspace for
the backward shift $S_\sY^*$ and $Z^*$ is similar to $S_\sY^*|\sH$.
So the spectral radius of $Z^*$ is less than or equal to one.
Since  $S_\sY^*|\sH$ is pointwise stable, $Z^*$ is also pointwise stable.

The identity in the first part  of  item (iii) follows from \eqref{Pick1a}.  Since $BB^*+\la \geq\la$ and $\la$ is strictly positive,   the operator $BB^*+\la $ is also strictly positive,  which proves the second part of item (iii). Finally, formula \eqref{OmPF} is a direct corollary of Lemma~\ref{lem:KK3}  by applying this lemma with
$K_1=\begin{bmatrix} \wt{B}  &   \la^{\frac{1}{2}} \end{bmatrix}$
and
$K_2=\begin{bmatrix} B  & Z \la^{\frac{1}{2}} \end{bmatrix}$, see \eqref{Pick1b}, and with $N= \widetilde{B}\widetilde{B}^*+\la=BB^*+Z\la Z^* $.
\end{proof}

Using formula \eqref{OmPF} we obtain the following explicit formula for the central solution $F_\circ$.

\begin{theorem}\label{thm:centrsol}
Let $\{W,\wt{W}, Z\}$ be a data set  for a LTONP interpolation problem,  and assume that the Pick operator $\la $ is strictly positive.  Then the central solution $F_\circ$ is given by
\begin{equation}\label{Fcentral}
F_\circ(\l)=B^*(\wt{B}\wt{B}^*+\la )^{-1}(I_\sZ- \l T)^{-1}\wt{B}, \hspace{.2cm}\mbox{where}\hspace{.2cm}  T=\la Z^*(\wt{B}\wt{B}^*+\la )^{-1}.
\end{equation}
Moreover, the spectral radius
$r_{\textup{spec}}(T)$  of $T$ is at most $1$.  Finally, if $\mathcal{Z}$ is finite dimensional, then $T$ is exponentially stable , that is, $r_{\textup{spec}}(T) < 1$.
\end{theorem}

\begin{proof}[\bf Proof]
Because $\Lambda$ is strictly positive,  $\sZ_\circ = \sZ$.
Let $G_\circ$ be the  function identically equal to $\o P_\sF$. Using \eqref{OmPF} we see that
\[
G_\circ(\l)=\begin{bmatrix}
  B^* K  \widetilde{B} &
  B^* K \la^{\frac{1}{2}} \\
 \la^{\frac{1}{2}}Z^*K \widetilde{B}
  & \la^{\frac{1}{2}}Z^* K \la^{\frac{1}{2}} \\
\end{bmatrix}:\begin{bmatrix}
    \mathcal{U} \\
      \mathcal{Z} \\
    \end{bmatrix}\rightarrow\begin{bmatrix} \mathcal{Y} \\
      \mathcal{Z} \\
    \end{bmatrix}.
\]
Hence, by Theorem \ref{thm:allsol1}, the central solution $F_\circ$ (see also
Corollary \ref{cor:central} and Remark \ref{rem:centrsol}) is given by
\[
F_\circ (\lambda) =
B^* K \widetilde{B} + \lambda B^* K \la^{\frac{1}{2}}
\left(I - \lambda \la^{\frac{1}{2}} Z^* K \la^{\frac{1}{2}}\right)^{-1}
\la^{\frac{1}{2}} Z^* K\widetilde{B}.
\]
Using $\la^{\frac{1}{2}}
\left(I - \lambda \la^{\frac{1}{2}} Z^* K \la^{\frac{1}{2}}\right)^{-1}=\left(I - \lambda  \la Z^* K  \right)^{-1}\la^{\frac{1}{2}}$, we have
\begin{align*}
F_\circ(\lambda) &=B^* K \widetilde{B}
+ \lambda B^* K
\left(I - \lambda  \la Z^* K  \right)^{-1}
 \la Z^* K\widetilde{B} \\
 &=B^* K \wt{B} +   B^* K
\left(I - \l   \la Z^* K  \right)^{-1}\Big(I-(I-\l \la Z^* K)\Big)\widetilde{B} \\
 &=B^* K \left(I - \lambda  \la Z^* K  \right)^{-1}\widetilde{B}.
\end{align*}
Since $K=(\wt{B}\wt{B}^*+\la)^{-1}$, this proves \eqref{Fcentral}.

Since $G_\circ(\lambda) = \omega P_\mathcal{F}$
is a contraction, its component $A= P_\mathcal{Z}\omega P_\mathcal{F}|\mathcal{Z} =
\Lambda^{\frac{1}{2}}Z^* K \Lambda^{\frac{1}{2}}$
is also a contraction. Because
$T = \Lambda^{\frac{1}{2}}(\Lambda^{\frac{1}{2}}Z^* K \Lambda^{\frac{1}{2}})\Lambda^{-\frac{1}{2}}$
is similar to $A$, it follows
that $\spec(T)  =\spec(A) \leq 1$.

Now assume that $\mathcal{Z}$ is finite dimensional, and
$\lambda$ is an eigenvalue for $T$ on the unit circle.
Because $T$ is similar to  $A$,
it follows that $A x = \lambda x$ for some nonzero $x$ in $\mathcal{Z}$.
In particular, $\|A x \| = \|\lambda x\| = \|x\|$.
Since $A$ is contained in the lower right hand corner of $\omega P_\mathcal{F}$
and $\omega$ is unitary, we have  $\omega P_\mathcal{F}(0 \oplus x) = 0 \oplus \lambda x$.
To see this notice that
\[
\|x\|^2 \geq \|\omega P_\mathcal{F}(0 \oplus x)\|^2 =
\|P_\mathcal{U}\omega P_\mathcal{F}(0 \oplus x)\|^2 + \|A x\|^2
=  \|P_\mathcal{U}\omega P_\mathcal{F}(0 \oplus x)\|^2 + \| x\|^2.
\]
Hence $P_\mathcal{U}\omega P_\mathcal{F}(0 \oplus x) =0$ and
$\omega P_\mathcal{F}(0 \oplus x) = 0 \oplus  A x = 0 \oplus \lambda x$. Since $\omega$ is a unitary operator,  $0 \oplus x$ must be in  $\mathcal{F}$.  So
$0 \oplus x = \widetilde{B}^* \xi \oplus \Lambda^{\frac{1}{2}} \xi$ for
some nonzero $\xi$ in $\mathcal{Z}$, that is, $x = \Lambda^{\frac{1}{2}} \xi$.
This with the definition of $\omega$ in \eqref{defomega2} readily implies that
\[
\begin{bmatrix}
                     0 \\
                     \lambda \Lambda^{\frac{1}{2}} \xi \\
                    \end{bmatrix}
    = \begin{bmatrix}
                     0 \\
                     \lambda x \\
                    \end{bmatrix}
                    = \omega P_\mathcal{F}\begin{bmatrix}
                      0 \\
                      x \\
                    \end{bmatrix} = \omega P_\mathcal{F}\begin{bmatrix}
                      \widetilde{B}^* \\
                      \Lambda^{\frac{1}{2}} \\
                    \end{bmatrix} \xi =
\begin{bmatrix}
                      {B}^* \\
                      \Lambda^{\frac{1}{2}} Z^* \\
                    \end{bmatrix} \xi.
\]
In other words, $\lambda \Lambda^{\frac{1}{2}} \xi = \Lambda^{\frac{1}{2}} Z^* \xi$,
or equivalently,  $\lambda   \xi =   Z^* \xi$. This says that $Z^*$ has an
eigenvalue  on the unit circle. However, $\mathcal{Z}$ is finite dimensional
and $Z^{*n}$ converges to zero. Hence $Z^*$ is exponentially stable , and thus
 all the eigenvalues of $Z^*$
are contained in the open unit disc. Therefore $\lambda$ cannot be an
eigenvalue for $Z^*$ and $T$ must be exponentially stable .
\end{proof}

\begin{lemma}\label{L:tau12}
Let $\{W, \wt{W}, Z\}$ be a data set for a LTONP interpolation problem and assume the Pick operator $\la$ is strictly positive. Further, let $C:\sZ\to \sE$ and $D:\sY\to \sE$ form an admissible pair of complementary operators, i.e., such that \eqref{semiunit1} and \eqref{semiunit2} holds. Then the operators $\t_1$ and $\t_2$ given by
\begin{equation}\label{tau12}
\t_1
= \begin{bmatrix} I \\  - \la^{-\frac{1}{2}} \widetilde{B}\\
            \end{bmatrix}R_\circ:\mathcal{U}\rightarrow
            \begin{bmatrix}
              \mathcal{U} \\
              \mathcal{Z}
            \end{bmatrix}, \quad
\t_2
=\begin{bmatrix}
              D^* \\
              \la^{-\frac{1}{2}}P C^*\\
            \end{bmatrix}Q_\circ:\mathcal{E}\rightarrow
            \begin{bmatrix}
              \mathcal{Y} \\
              \mathcal{Z}
            \end{bmatrix},
\end{equation}
with $R_\circ$ and $Q_\circ$ given by \eqref{defQR0},
are isometries, the range of $\t_1$ is $\sG$ and  the range of $\t_2$ is $\sG^\prime$.
\end{lemma}

\begin{proof}[\bf Proof]
We split the proof into two parts.  In the first part we deal with $\t_1$ and in the second part with $\t_2$.

\smallskip
\noindent\textsc{Part 1.}
Using the definition of $R_\circ$ in \eqref{defQR0}, we have
\[
\t_1^*\t_1=R_\circ \begin{bmatrix} I & -\wt{B}^*\la^{-\frac{1}{2}}\end{bmatrix}
\begin{bmatrix} I \\  - \la^{-\frac{1}{2}} \widetilde{B}  \end{bmatrix}R_\circ=R_\circ\left (I+\wt{B}^*\la^{-1}\wt{B}\right)R_\circ=I_\sU.
\]
Thus $\t_1$ is an isometry. In particular, the range of $\t_1$ is closed.
Furthermore, note that
\[
\t_1^*\begin{bmatrix}\wt{B}^*\\ \la^{\frac{1}{2}} \end{bmatrix}= R_\circ \begin{bmatrix} I & -\wt{B}^*\la^{-\frac{1}{2}} \end{bmatrix}
\begin{bmatrix}\wt{B}^*\\ \la^{\frac{1}{2}} \end{bmatrix}=R_\circ(\wt{B}^*-\wt{B}^*)=0.
\]
Recall that  in the present case, when $\Lambda$ is strictly positive,
we have
\[
\sF=    \im\begin{bmatrix}
  \wt{B}^*    \\ \la^{\frac{1}{2}} \\
\end{bmatrix}  \\
 \ands    \sG=\sF^\perp=\kr \begin{bmatrix} \wt{B} & \la^{\frac{1}{2}} \end{bmatrix}.
\]
The fact that  $\widetilde{B}\widetilde{B}^* + \Lambda$ is strictly positive, implies that the range of the previous $2\times 1$ operator matrix is closed.  It follows  that $\sF\subset \kr \t_1^*$, and hence $\im \t_1\subset \sF^\perp=\sG$.  To prove that $\im \t_1=\sG$, consider the operator
\begin{equation}\label{defMt1}
N=\begin{bmatrix} R_\circ&\wt{B}^*\\
-\la^{-\frac{1}{2}}\wt{B}R_\circ&\la^{\frac{1}{2}}\end{bmatrix}:
\begin{bmatrix}\sU\\ \sZ \end{bmatrix}\to \begin{bmatrix}\sU\\ \sZ \end{bmatrix}.
\end{equation}
This operator matrix is invertible because the operator $\la^{\frac{1}{2}}$ and    the  Schur complement $N^\ts$ of $\la^{\frac{1}{2}}$ in $N$ are both invertible. To see that $N^\ts$ is invertible, note that
\[
N^\ts=R_\circ+\wt{B}^*\la^{-\frac{1}{2}} \la^{-\frac{1}{2}}\wt{B}R_\circ=(I+\wt{B}^*\la^{-1}B)R_\circ=R_\circ^{-1}.
\]
Next observe that the first column of $N$ is the operator $\t_1$ while  the range  of the second column of $N$ is $\sF$. Since $N$ is invertible, $\im N=\sU\oplus\sZ=\sG\oplus\sF$. It follows that $\sG$ must be included in the range of the first column of $N$, that is, $\sG\subset \im \t_1$.  But then $\im \t_1=\sG$.

\smallskip
\noindent\textsc{Part 2.} First observe that $Q_\circ$ is also given by
\begin{equation}\label{altdefQ0}
 Q_\circ =(DD^*+CP\la^{-1}PC^*)^{-\half}.
\end{equation}
To see this, note that  \eqref{semiunit1} implies that $DD^*+CPC^*=I_\sE$, and thus
\[
DD^*+CP\la^{-1}PC^*\!\! =\! I_\sE-CPC^*+CP\la^{-1}PC^*\!\! =\! I_\sE+CP\left(\la^{-1}-P^{-1}\right)PC^*.
\]
Using the definition of $\t_2$  in \eqref{tau12}  and the formula for $Q_\circ$ in \eqref{altdefQ0}, we obtain
\[
\t_2^*\t_2=Q_\circ \begin{bmatrix}D& CP \la^{-\frac{1}{2}}\end{bmatrix}
\begin{bmatrix}D^*\\ \la^{-\frac{1}{2}} PC^*\end{bmatrix}Q_\circ=
Q_\circ(DD^* +CP \la^{-1}PC^*)Q_\circ= I_\sE.
\]
Thus $\t_2$ is an isometry. In particular,  the range of $\t_2$ is closed. From the identity  \eqref{semiunit1} we know that $BD^*+ZPC^*=0$. This implies that
\[
\t_2^* \begin{bmatrix} B^*   \\ \la^{\frac{1}{2}}Z^*  \end{bmatrix}
=Q_\circ\begin{bmatrix} D& CP \la^{-\frac{1}{2}}\end{bmatrix}
\begin{bmatrix} B^*   \\ \la^{\frac{1}{2}}Z^*   \end{bmatrix} =Q_\circ(BD^*+CPZ^*)=0.
\]
Recall that in the present strictly positive case
\[
\sF^\prime= \im \begin{bmatrix}
   B^*   \\ \la^{\frac{1}{2}}Z^* \\
\end{bmatrix} \quad  \mbox{so that}\quad
 \mathcal{G}^\prime=\sF^{\prime\perp} =\kr \begin{bmatrix} B& Z\la^{\frac{1}{2}}\end{bmatrix}.
\]
We conclude that $\sF^\prime\subset  \kr \t_2^*$, and hence $\im \t_2\subset {\sF^\prime}^\perp=\sG^\prime$.  To prove $\im \t_2=\sG^\prime$ we take $y\in \sY$ and $z\in \sZ$, and assume that $y\oplus z\in \sG^\prime$ and $y\oplus z\perp \im \t_2$. In other words, we assume that
\begin{equation}\label{2ortho-cond}
\begin{bmatrix}y\\ z  \end{bmatrix}  \perp\ \sF^\prime= \im \begin{bmatrix}   B^*\\ \la^{\frac{1}{2}} Z^* \end{bmatrix}  \ands  \begin{bmatrix}y\\ z  \end{bmatrix}  \perp\ \im \t_2=\im \begin{bmatrix} D^* \\ \la^{-\frac{1}{2}}PC^*  \end{bmatrix}.
\end{equation}
But then
\[
\begin{bmatrix} B & Z\end{bmatrix}\begin{bmatrix}y\\ \la^{\frac{1}{2}}z  \end{bmatrix}=0 \ands
\begin{bmatrix} D & CP\la^{-1}\end{bmatrix}\begin{bmatrix}y\\ \la^{\frac{1}{2}}z  \end{bmatrix}=0.
\]
In other words,
\begin{equation}\label{opmatrix}
\begin{bmatrix}
 D & C  P \Lambda^{-1}  \\
 B & Z \\
\end{bmatrix}\begin{bmatrix} y \\\la^{\frac{1}{2}}z \\  \end{bmatrix}=
\begin{bmatrix}0 \\0\\  \end{bmatrix}.
\end{equation}
Now observe that
\begin{equation}\label{inv678}
\begin{bmatrix}
 D & C  P \Lambda^{-1}  \\
 B & Z \\
\end{bmatrix} \begin{bmatrix}
 D^*    & B^*  \\
 P C^*  & P Z^* \\
\end{bmatrix} =
\begin{bmatrix}
Q_\circ^{-2}   & \star  \\
 0             & P \\
\end{bmatrix}
\end{equation}
where $\star$ represents an unspecified entry.
The identities \eqref{semiunit1} and   \eqref{semiunit2} imply that the operator matrix
\[
\begin{bmatrix}
 D^*    & B^*  \\
 P C^*  & P Z^* \\
\end{bmatrix} =\begin{bmatrix}
 I    & 0  \\
 0    & P   \\
\end{bmatrix} \begin{bmatrix}
 D^*    & B^*  \\
   C^*  &   Z^* \\
\end{bmatrix}
\]
is invertible. Because $Q_\circ$ and $P$ are both invertible, the
matrix on the right hand side of \eqref{inv678} is invertible. So the operator matrix on
the left hand side of \eqref{inv678} or \eqref{opmatrix} is invertible. Thus
 $y  \oplus \la^{\frac{1}{2}}z = 0$.  Since $\la^{\half}$ is invertible,
 both $y$ and $z$ are zero.  This can only happen when  $\sG^\prime =\im \t_2$.
\end{proof}

%%%%%%%%%%%%%%%%%
\begin{corollary}\label{C:Gform}
Let $\{W, \wt{W}, Z\}$ be a data set for a LTONP interpolation problem and assume the Pick operator $\la$ is strictly positive. Then all functions $G$ in $\sS(\sU\oplus\sZ, \sY\oplus\sZ)$ with $G(0)|_\sF=\o$ are given by
\begin{align}\label{Gform}
&G(\lambda) =\mat{cc}{G_{11}(\l)&G_{12}(\l)\\ G_{21}(\l) & G_{22}(\l)}=\\
&\small{
=\!\!\begin{bmatrix}
 \! B^* K  \widetilde{B} + D^* Q_\circ  X(\l) R_\circ \! & \!
  B^* K \la^{\frac{1}{2}} - D^* Q_\circ X(\l) R_\circ \widetilde{B}^* \la^{-\frac{1}{2}}\! \\[.1cm]
\! \la^{\frac{1}{2}}Z^*K \widetilde{B} + \la^{-\frac{1}{2}}P C^* Q_\circ X(\l)  R_\circ
  \! & \! \la^{\frac{1}{2}}Z^* K \la^{\frac{1}{2}}
   -\la^{-\frac{1}{2}}P C^* Q_\circ X(\l) R_\circ \widetilde{B}^* \la^{-\frac{1}{2}}\!
\end{bmatrix}}\notag
\end{align}
with $X$ is an arbitrary Schur class function in $\sS(\mathcal{U}, \mathcal{E})$. Moreover, $G$ and $X$ determine each other uniquely. Furthermore, we have
\begin{equation}\label{GijRel}
\begin{aligned}
G_{12}(\l)&=(B^*-G_{11}(\l)\wtil{B}^*)\la^{-\half},\\
G_{22}(\l)&=\la^{\half}(Z^*-\la^{-\half}G_{21}(\l)\wtil{B}^*)\la^{-\half}.
\end{aligned}
\end{equation}
%\begin{equation}\label{GijRel}
% G_{12}(\l)=(B^*-G_{11}(\l)\wtil{B}^*)\la^{-\half}\!\ands \!
% G_{22}(\l)=\la^{\half}(Z^*-\la^{-\half}G_{21}(\l)\wtil{B}^*)\la^{-\half}.
%\end{equation}
\end{corollary}

\begin{proof}[\bf Proof]
The fact that $\o:\sF\to\sF'$ is unitary implies that $G\in \sS(\sU\oplus\sZ, \sY\oplus\sZ)$ satisfies  $G(0)|\sF=\o$  if and only if $G(\lambda) = \o  P_\mathcal{F} + \wtil{X}(\lambda) P_\sG$, $\l\in \BD$, for some $\wtil{X} \in \sS(\sG,\sG')$. Since the operators $\t_1$ and $\t_2$ introduced in Lemma \ref{L:tau12} are isometries with ranges equal to $\sG$ and $\sG'$, respectively, it follows (see Lemma \ref{lem:KK2})  that $\wtil{X}$ is in $\sS(\sG,\sG')$ if and only if $\wtil{X}(\l)=\t_2 X(\l) \t_1^*$, $\l\in \BD$, for a $X\in \sS(\sU,\sE)$, namely $X(\l)\equiv \t_2^*\wtil{X}(\l)\t_1$. Hence the Schur class functions $G\in \sS(\sU\oplus\sZ, \sY\oplus\sZ)$ with $G(0)|\sF=\o$ are characterized by $G(\lambda) = \o  P_\mathcal{F} + \t_2 X(\lambda) \t_1^*$ with $X\in \sS(\sU,\sE)$. It is clear from the above constructions that $G$ and $\wtil{X}$ determine each other uniquely, and that $\wtil{X}$ and $X$ determine each other uniquely. Hence $G$ and $X$ determine each other uniquely. Using the formulas for $\o P_\sF$ and $\t_1$ and $\t_2$ obtained in Lemmas \ref{L:useobs1} and \ref{L:tau12} we see that $\o  P_\mathcal{F} + \t_2 X(\lambda) \t_1^*$ coincides with the right-hand side of \eqref{Gform}.

It remains to derive \eqref{GijRel}. Note that
\[
K=(\la + \wtil{B}\wtil{B}^*)^{-1}=\la^{-1}-\la^{-1}\wtil{B} R_\circ^2\wtil{B}^*\la^{-1}.
\]
This implies that
\begin{align*}
K\wtil{B}&=\la^{-1}\wtil{B}(I- R_\circ^2\wtil{B}^*\la^{-1}\wtil{B})=\la^{-1}\wtil{B}R_\circ^2,\\[.1cm]
K\la&=(\la^{-1}-\la^{-1}\wtil{B} R_\circ^2\wtil{B}^*\la^{-1})\la=I-\la^{-1}\wtil{B}R_\circ^2\wtil{B}^*
=I-K\wtil{B}\wtil{B}^*.
\end{align*}
Summarising we have
\begin{equation}
K\wtil{B}=\la^{-1}\wtil{B}R_\circ^2 \ands
K\la=I-K\wtil{B}\wtil{B}^*.\label{KBtilKLambda}
\end{equation}
We now obtain  that
\begin{align*}
G_{12}(\l)
&=\left(B^* K\la - D^* Q_\circ X(\l) R_\circ \widetilde{B}^*\right) \la^{-\half}\\
&=\left(B^*-B^*K\wtil{B}\wtil{B}^* - D^* Q_\circ X(\l) R_\circ \widetilde{B}^*\right) \la^{-\half}\\
&=\left(B^*-\left(B^*K\wtil{B}+ D^* Q_\circ X(\l) R_\circ \right)\wtil{B}^*\right) \la^{-\half}\\
& =\left(B^*-G_{11}(\l)\wtil{B}^*\right)\la^{-\half},
\end{align*}
and
\begin{align*}
G_{22}(\l)
&= \la^{\half}\left(Z^* K\la-\la^{-1}P C^* Q_\circ X(\l) R_\circ \widetilde{B}^*\right) \la^{-\half}\\
&= \la^{\half}\left(Z^* -Z^*  K\wtil{B}\wtil{B}^*-\la^{-1}P C^* Q_\circ X(\l) R_\circ \widetilde{B}^*\right) \la^{-\half}\\
&= \la^{\half}\left(Z^* -\left(Z^*  K\wtil{B}+\la^{-1}P C^* Q_\circ X(\l) R_\circ\right) \widetilde{B}^*\right) \la^{-\half}\\
&=\la^{\half}\left(Z^* -\la^{-\half}G_{21}(\l) \widetilde{B}^*\right) \la^{-\half},
\end{align*}
as claimed.
\end{proof}

\begin{proof}[\bf Proof of Theorem \ref{thm:main1}]
The first statements in Theorem \ref{thm:main1} are covered by Lemma \ref{L:useobs1}. Clearly the operators $Q_\circ$ and $R_\circ $ are well defined. Since the spectral radius of $Z$  is  at most one, the operator-valued functions $\Upsilon_{ij}$,  $i,j=1,2$, given by \eqref{defUp11} --  \eqref{defUp22} are well defined and analytic on $\BD$. Given these functions it remains to prove the main part of the theorem describing all solutions of the LTONP interpolation problem by \eqref{allsol1a}.

Let $X\in \sS(\mathcal{U}, \mathcal{Y})$  be an arbitrary Schur class function. Define $G$ in $\sS(\sU\oplus\sZ,\sY\oplus\sZ)$ by $G(\l) = \o  P_\mathcal{F} + \t_2 X(\lambda) \t_1^*$, $\l\in\BD$, where $\t_1$ and $\t_2$ are given  by \eqref{tau12}. Hence $G$ is given by \eqref{Gform} and we have \eqref{GijRel}. Set
\[
F(\l) = G_{11}(\lambda) + \lambda G_{12}(\lambda)
\Big(I - \lambda G_{22}(\lambda)\Big)^{-1}G_{21}(\lambda),\quad \l\in\BD.
\]

By  item (ii) in Lemma \ref{L:useobs1} the spectral  radius of $Z$ is at most one, and hence the same holds true for spectral  radius of $Z^*$. Thus $I-\l Z^*$  is invertible for each $\l \in \BD$. Now fix a $\l\in\BD$. Since $G\in\sS(\sU\oplus\sZ,\sY\oplus\sZ)$, we have $G_{22}\in \sS(\sZ,\sZ)$ and thus $I-\l G_{22}(\l)$ is invertible. Notice that
\begin{align*}
I-\l G_{22}(\l)
&=\la^{\half}\left( I-\l Z^* +\l\la^{-\half}G_{21}(\l)\wtil{B}^*\right)\la^{-\half}\\
&=\la^{\half}(I-\l Z^*)\left( I +\l(I-\l Z^*)^{-1}\la^{-\half}G_{21}(\l)\wtil{B}^*\right)\la^{-\half}.
\end{align*}
 The above identity shows that $I +\l (I-\l Z^*)^{-1}\la^{-\half}G_{21}(\l)\wtil{B}^*$ is invertible.  Applying the rule that $I+AB$ is invertible if and only if  $I+BA$ is invertible, we obtain that the operator
$I +\l \wt{B}^* (I-\l Z^*)^{-1}\la^{-\half}G_{21}(\l)$  is invertible. Next, using the rule $(I+AB)^{-1}A=A(I+BA)^{-1}$ we obtain
\begin{align*}
&\left(I-\l G_{22}(\l)\right)^{-1}G_{21}(\l)=\\
&\hspace{1cm}=\la^{\half}\left( I +\l(I-\l Z^*)^{-1}\la^{-\half}G_{21}(\l)\wtil{B}^*\right)^{-1} (I-\l Z^*)^{-1}\la^{-\half}G_{21}(\l)\\
&\hspace{1cm}=\la^{\half}(I-\l Z^*)^{-1}\la^{-\half}G_{21}(\l)\left( I +\l\wtil{B}^*(I-\l Z^*)^{-1}\la^{-\half}G_{21}(\l)\right)^{-1}.
\end{align*}
From the first identity in \eqref{GijRel} we obtain
\begin{align*}
&\l G_{12}(\l) \la^{\half}(I-\l Z^*)^{-1}\la^{-\half}G_{21}(\l)=\\
&\qquad=\l\left(B^*-G_{11}(\l)\wtil{B}^*\right)(I-\l Z^*)^{-1}\la^{-\half}G_{21}(\l)\\
&\qquad =\l B^*(I-\l Z^*)^{-1}\la^{-\half}G_{21}(\l) -\l G_{11}(\l)\wtil{B}^*(I-\l Z^*)^{-1}\la^{-\half}G_{21}(\l)\\
&\qquad=\l B^*(I-\l Z^*)^{-1}\la^{-\half}G_{21}(\l) + G_{11}(\l)+\\
&\hspace{4cm} - G_{11}(\l)\left(I+\l \wtil{B}^*(I-\l Z^*)^{-1}\la^{-\half}G_{21}(\l)\right).
\end{align*}
Summarising we have shown that
\begin{align*}
&\left(I-\l G_{22}(\l)\right)^{-1}G_{21}(\l)=\la^{\half}(I-\l Z^*)^{-1}\la^{-\half}G_{21}(\l)\Xi(\l)\\
&\l G_{12}(\l) \la^{\half}(I-\l Z^*)^{-1}\la^{-\half}G_{21}(\l) =\\
&\hspace{2cm}=G_{11}(\l)+\l B^*(I-\l Z^*)^{-1}\la^{-\half}G_{21}(\l) -  G_{11}(\l)\Xi(\l)^{-1},\\
&\hspace{3cm}\mbox{where}\hspace{.35cm} \Xi(\l)=\left( I +\l\wtil{B}^*(I-\l Z^*)^{-1}\la^{-\half}G_{21}(\l)\right)^{-1}.
\end{align*}
It follows that
\begin{align*}
F(\l) &=G_{11}(\l)+\l G_{12}(\l)\left(I-\l G_{22}(\l)\right)^{-1}G_{21}(\l)\\
&=G_{11}(\l)+\l G_{12}(\l)\la^{\half}(I-\l Z^*)^{-1}\la^{-\half}G_{21}(\l)\Xi(\l)\\
&=G_{11}(\l)+\left(G_{11}(\l)+\l B^*(I-\l Z^*)^{-1}\la^{-\half}G_{21}(\l) \right)\Xi(\l)-G_{11}(\l)\\
&=\left(G_{11}(\l)+\l B^*(I-\l Z^*)^{-1}\la^{-\half}G_{21}(\l)\right)\times\\
&\qquad\qquad \times \left(I+\l \wtil{B}^*(I-\l Z^*)^{-1}\la^{-\half}G_{12}(\l)\right)^{-1}.
\end{align*}
To prove the parametrization of solutions through \eqref{allsol1a} it remains to show that
\begin{align}
G_{11}(\l)+\l B^*(I-\l Z^*)^{-1}\la^{-\half} G_{21}(\l)
&=\left(\Up_{12}(\l)+\Up_{11}(\l)X(\l)\right)R_\circ,   \label{part1}\\
I+\l \wtil{B}^*(I-\l Z^*)^{-1}\la^{-\half}G_{12}(\l) &=\left(\Up_{22}(\l)+\Up_{21}(\l)X(\l)\right)R_\circ.\label{part2}
\end{align}
Note that these two identities  show that $F$ is given by \eqref{allsol1a} and, combined with Theorem~\ref{thm:allsol1}, this yields that all solutions to the LTONP interpolation problem are given by \eqref{allsol1a}. Hence we have proved Theorem \ref{thm:main1} once these two identities are established.

Using \eqref{KBtilKLambda} we obtain that
\begin{align*}
&(I-\l Z^*)^{-1}\la^{-\half}G_{21}(\l)=\\
&\qquad=(I-\l Z^*)^{-1}Z^*  K\wtil{B}
+(I-\l Z^*)^{-1}\la^{-1}P C^* Q_\circ X(\l) R_\circ\\
&\qquad=\left((I-\l Z^*)^{-1}Z^*  \la^{-1}\wtil{B}R_\circ+
(I-\l Z^*)^{-1}\la^{-1}P C^* Q_\circ X(\l)\right) R_\circ.
\end{align*}
Therefore, we have
\begin{align}
&I+\l \wtil{B}^*(I-\l Z^*)^{-1}\la^{-\half}G_{21}(\l)=  I+ \nn  \\
& \quad +\left(\l \wtil{B}^*(I-\l Z^*)^{-1}Z^*  \la^{-1}\wtil{B}R_\circ +\l \wtil{B}^*(I-\l Z^*)^{-1}\la^{-1}P C^* Q_\circ X(\l)
\right)R_\circ \nn \\
&\quad =\left( R_\circ^{-1} +\l \wtil{B}^*(I-\l Z^*)^{-1}Z^*  \la^{-1}\wtil{B}R_\circ+\Up_{21}(\l)X(\l)\right)R_\circ.\label{topart2}
\end{align}
From the definition $R_\circ$ in \eqref{defQR0} it follows that $R_\circ^{-2}-\wtil{B}^*\la^{-1}\wtil{B}=I_\sU$, and hence
\begin{align}
&R_\circ^{-1} +\l \wtil{B}^*(I-\l Z^*)^{-1}Z^*  \la^{-1}\wtil{B}R_\circ=\nn \\
&\quad=R_\circ^{-1}+\wtil{B}^*(I-\l Z^*)^{-1}\left(I-(I-\l Z^*)\right)\la^{-1}\wtil{B}R_\circ  \nn \\
&\quad= R_\circ^{-1}-\wtil{B}^*\la^{-1}\wtil{B}R_\circ+\wtil{B}^*(I-\l Z^*)^{-1}\la^{-1}\wtil{B}R_\circ \nn \\
&\quad=\left(R_\circ^{-2}-\wtil{B}^*\la^{-1}\wtil{B}\right)R_\circ+\wtil{B}^*(I-\l Z^*)^{-1}\la^{-1}\wtil{B}R_\circ\nn  \\
&\quad=R_\circ +\wtil{B}^*(I-\l Z^*)^{-1}\la^{-1}\wtil{B}R_\circ=\Upsilon_{22}(\l).\label{altdefUp22}
\end{align}
Inserting the identity \eqref{altdefUp22} in \eqref{topart2} we obtain the identity \eqref{part2}.

We proceed with the  left hand side of  \eqref{part1}.
\begin{align}
&G_{11}(\l)+\l B^*(I-\l Z^*)^{-1}\la^{-\half}G_{21}(\l)=B^*K\wtil{B} + D^* Q_\circ X(\l) R_\circ+\nn \\
&\qquad +
\l B^*\left((I-\l Z^*)^{-1}Z^*  \la^{-1}\wtil{B}R_\circ+
(I-\l Z^*)^{-1}\la^{-1}P C^* Q_\circ X(\l)\right) R_\circ \nn \\
&\hspace{.5cm}=B^*\la^{-1}\wtil{B}R_\circ^2+ \l B^* (I-\l Z^*)^{-1}Z^* \la^{-1}\wtil{B}R_\circ^2 +\nn \\
&\hspace{3cm}+ \Big(D^* Q_\circ+\l B^*(I-\l Z^*)^{-1}\la^{-1}P C^* Q_\circ\Big)X(\l) R_\circ\nn \\
&\hspace{.5cm}=B^*\la^{-1}\wtil{B}R_\circ^2   +\l B^* (I-\l Z^*)^{-1}Z^* \la^{-1}\wtil{B}R_\circ^2
+\Up_{11}(\l)X(\l) R_\circ. \label{topart1}
\end{align}
Next we compute
\begin{align}
&B^*\la^{-1}\wtil{B}R_\circ^2   +\l B^* (I-\l Z^*)^{-1}Z^* \la^{-1}\wtil{B}R_\circ^2= \nn \\
&\qquad= B^*\la^{-1}\wtil{B}R_\circ^2 +B^* (I-\l Z^*)^{-1}
\left(I- (I-\l Z^*)\right)\la^{-1}\wtil{B} R_\circ^2\nn \\
&\qquad=B^*\la^{-1}\wtil{B}R_\circ^2-B^*\la^{-1}\wtil{B}R_\circ^2+B^*
 (I-\l Z^*)^{-1}\la^{-1}\wtil{B} R_\circ^2  \nn \\
&\qquad=
B^* (I-\l Z^*)^{-1}\la^{-1}\wtil{B} R_\circ^2
=\Upsilon_{12}(\l) R_\circ.\label{altdefUp12}
\end{align}
Inserting the identity \eqref{altdefUp12} in \eqref{topart1} we obtain the identity \eqref{part1}.  Hence we have shown that all solutions are obtained through \eqref{allsol1a}.

To complete the proof we show that the map $X\mapsto F$ given by \eqref{allsol1a} is one-to-one. This is a direct consequence of the uniqueness claims in Corollary \eqref{C:Gform} and Theorem \ref{thm:allsol1}. Indeed, by Corollary \eqref{C:Gform}, the map $X\mapsto G$ from $\sS(\sU,\sE)$ to $\sS(\sU\oplus\sZ_\circ,\sY\oplus\sZ_\circ)$ given by \eqref{Gform} is one-to-one, and each $G$ obtained in this way has $G(0)\sF=\o$. By Theorem \ref{thm:allsol1}, the map $G\mapsto F$ from the set of $G\in \sS(\sU\oplus\sZ_\circ,\sY\oplus\sZ_\circ)$ with $G(0)|\sF=\o$ to the set of solutions in $\sS(\sU,\sY)$ given by \eqref{allsol1} is also one-to-one. Since the map $X\mapsto F$ defined here is the composition of these two maps, it follow that this map is one-to-one as well.
\end{proof}

%:Sec5
\setcounter{equation}{0}
\section{Proof of Theorem \ref{thm:main1a}}\label{sec:Thm1.2}
We begin with a general remark  concerning the formulas for the functions $\Upsilon_{ij}$, $1\leq  i,j\leq 2$, appearing in Theorem \ref{thm:main1}.

Let  $\{W, \wt{W}, Z\}$  be a LTONP data set, and assume that the associate Pick operator $\la$ is strictly positive.  Then  $Z^*$ is pointwise stable. Using the definitions of
 $B=W E_\sY$ and $\wt{B}=\widetilde{W} E_\sU$ (see  in \eqref{basicid1a}  and  \eqref{basicid1b}) with  the intertwining relations $S_\sY^* W^* = W^* Z^*$
and $S_\sU^* \widetilde{W}^* = \widetilde{W}^* Z^*$ (see \eqref{data1}), we obtain
\begin{align*}
B^*(I-\l Z^*)^{-1}&=E_\sY^* W^*(I-\l Z^*)^{-1}=E_\sY^*(I-\l S_\sY^*)^{-1}W^*\quad (\l \in \BD),\\
\wt{B}^*(I-\l Z^*)^{-1}&=E_\sU^* \wt{W}^*(I-\l Z^*)^{-1}=E_\sU^* (I-\l S_\sU^*)^{-1}\wt{W}^*\quad (\l \in \BD).
\end{align*}
It follows that  the formulas \eqref{defUp11} -- \eqref{defUp22} can be rewritten  as follows:
\begin{align}
\Upsilon_{11}(\l)&=D^*Q_\circ  +\l E_\sY^*(I-\l S_\sY^*)^{-1}W^*\la^{-1}{P}C^* Q_\circ, \label{defUp11a} \\[.1cm]
\Upsilon_{12}(\l)&=E_\sY^*(I-\l S_\sY^*)^{-1}W^*\la^{-1}\wt{B}R_\circ,\label{defUp12a} \\[.1cm]
\Upsilon_{21}(\l)&=\l E_\sU^* (I-\l S_\sU^*)^{-1}\wt{W}^*\la^{-1} PC^*Q_\circ, \label{defUp21a} \\[.1cm]
\Upsilon_{22}(\l)&=R_\circ +E_\sU^* (I-\l S_\sU^*)^{-1}\wt{W}^*\la^{-1}\wt{B} R_\circ.\label{defUp22a}
\end{align}

\begin{proof}[\bf Proof of  Theorem \ref{thm:main1a}]
As before let
$\{W, \wt{W}, Z\}$  be a LTONP data set,
and assume that the associate Pick operator $\la$ is strictly positive.  Note that
\[
WW^*=\la +\wt{W}\wt{W}^* \gg 0.
\]
Hence $P=WW^*$ is  also strictly positive.  It follows that the operator $A=W^*P^{-1}\wt{W}$ in \eqref{defA} is well-defined. Finally, it is noted that $W A = \widetilde{W}$.

We first show that $A$ is strictly contractive following arguments similar to the ones used in \cite[Remark II.1.4]{FFGK98}. Note that
\begin{align*}
I-A^*A&=I-\wt{W}^*P^{-1}W W^*P^{-1}\wt{W}=I-\wt{W}^*P^{-1}\wt{W}\\
&=I-\left(\wt{W}^*P^{-\frac{1}{2}}\right)\left (P^{-\frac{1}{2}}\wt{W}\right).
\end{align*}
Put $W_0= P^{-\frac{1}{2}} W$ and $\wt{W}_0=P^{-\frac{1}{2}}\wt{W}$. Then $I-A^*A=I-\wt{W}_0^*\wt{W}_0$. Furthermore,
\begin{align*}
I-\wt{W}_0 \wt{W}_0^*=&I-P^{-\frac{1}{2}}\wt{W}\wt{W}^*P^{-\frac{1}{2}}= P^{-\frac{1}{2}}\left(P-\wt{W}\wt{W}^*\right) P^{-\frac{1}{2}}\\
=&P^{-\frac{1}{2}}\la P^{-\frac{1}{2}}\gg 0.
\end{align*}
Thus $\wt{W}_0^*$ is a strict contraction, and hence the same holds true for $\wt{W}_0$. We conclude that
\[
I-A^*A=I-\wt{W}_0^*\wt{W}_0\gg 0,
\]
and $A$ is a strict contraction.

From the above calculations it follows that $I-A^*A$ is invertible and we can obtain the inverse of  $I-A^*A$ by using the standard operator identity:
\begin{equation}\label{gen-id}
(I-ML)^{-1}=I+M(I-LM)^{-1}L.
\end{equation}
Indeed, we have
\begin{align*}
(I-A^*A)^{-1}&=(I-\wt{W}_0^*\wt{W}_0)^{-1}=I+\wt{W}_0^*
\left(I -\wt{W}_0\wt{W}_0^*\right)^{-1}\wt{W}_0 \nn \\
&=I+\wt{W}^*P^{-\frac{1}{2}}\Big(I-P^{-\frac{1}{2}}\wt{W}\wt{W}^*P^{-\frac{1}{2}}\Big)^{-1}P^{-\frac{1}{2}}\wt{W}\nn \\
&=I+\wt{W}^*\left(P-\wt{W}\wt{W}^*\right)^{-1}\wt{W}\nn
= I+\wt{W}^*\la^{-1}\wt{W}.
\end{align*}
This readily implies that
\begin{equation}\label{inversion1}
 (I-A^*A)^{-1} = I+\wt{W}^*\la^{-1}\wt{W}.
\end{equation}

Next we derive formulas \eqref{Q0} and  \eqref{R0}.  We begin with $Q_\circ$. Note that
\begin{align*}
&A(I-A^*A)^{-1}A^*=W^*P^{-1}\wt{W}\left(I+\wt{W}^*\la^{-1}\wt{W}\right)\wt{W}^*P^{-1}W\nn \\
&\quad=W^*P^{-1}\wt{W}\wt{W}^*P^{-1}W+W^*P^{-1}\wt{W}\wt{W}^*\la^{-1}\wt{W}\wt{W}^*P^{-1}W\nn \\
&\quad=W^*P^{-1}(P-\la)P^{-1}W+W^*P^{-1}(P-\la)\la^{-1}(P-\la)P^{-1}W\nn \\
&\quad=W^*P^{-1}(P-\la)P^{-1}W+W^*P^{-1}(P-\la)\la^{-1}W+\nn\\
&\hspace{6cm}-W^*P^{-1}(P-\la)P^{-1}W\nn \\
&\quad=W^*P^{-1}(P-\la)\la^{-1}W= W^*\la^{-1}W-W^*P^{-1}W.
\end{align*}
In other words,
\begin{equation}\label{inversion2}
 A(I-A^*A)^{-1}A^* = W^*\la^{-1}W-W^*P^{-1}W.
\end{equation}
Thus
\[
WA(I-A^*A)^{-1}A^*W^*=P\la^{-1}P-P=P\left(\la^{-1}-P^{-1}\right)P.
\]
Combining this with
 $Q_\circ =\left(I+CP(\la^{-1}-P^{-1})PC^*\right)^{-\frac{1}{2}}$
 (see \eqref{defQR0})  yields the formula
 $Q_\circ =\left(I + C W A \left(I - A^*A\right)^{-1}A^* W^* C^*\right)^{-\frac{1}{2}}$  for $Q_\circ$ in \eqref{Q0}.

We proceed by  deriving formula  \eqref{R0}. According to the right hand side  of  \eqref{defQR0} and using the identity \eqref{inversion1} we have
\begin{align*}
R_\circ&=\left(I_\sU+\wt{B}^*\la^{-1}\wt{B}\right)^{-\frac{1}{2}}
=\left(I_\sU+E_\sU^*\wt{W}^*\la^{-1}\wt{W}E_\sU\right)^{-\frac{1}{2}}\\
&=\left(E_\sU^*\left(I+\wt{W}^*\la^{-1}\wt{W}\right)E_\sU\right)^{-\frac{1}{2}}
=\left(E_\sU^*(I-A^*A)^{-1}E_\sU\right)^{-\frac{1}{2}}.
\end{align*}
We conclude that  \eqref{R0} is proved.

It remains to show that formulas \eqref{defUp11} -  \eqref{defUp22} can be rewritten as \eqref{Upsilon11} --  \eqref{Upsilon22}, respectively. To do this we  use the remark preceding the present  proof.  In other words  we may assume that the functions $\Upsilon_{ij}$, $1\leq i,j \leq 2$, are given by  \eqref{defUp11a} --  \eqref{defUp22a}.  Then, to derive \eqref{Upsilon11} --  \eqref{Upsilon22},  it  suffices to show that
\begin{align}
W^*\la^{-1}P&= (I-AA^*)^{-1}W^*, \qquad W^*\la^{-1}\wt{B}=A(I-A^*A)^{-1}E_\sU,  \label{ident12a}\\[.1cm]
\wt{W}^*\la^{-1}P&=A^* (I-AA^*)^{-1}W^*, \qquad \wt{W}^*\la^{-1}\wt{B}=(I-A^*A)^{-1}E_\sU-E_\sU. \label{ident21a}
\end{align}
Obviously, the first three identities are  enough to derive formulas  \eqref{Upsilon11}, \eqref{Upsilon12}, and \eqref{Upsilon21} from  the formulas   \eqref{defUp11a}, \eqref{defUp12a},  and \eqref{defUp21a}, respectively. To see that a similar result holds true for the second identity in \eqref{ident21a}, note that this second identity in \eqref{ident21a} implies that
\begin{align*}
\Upsilon_{22}(\l)&=R_\circ +E_\sU^* (I-\l S_\sU^*)^{-1}\wt{W}^*\la^{-1}\wt{B} R_\circ\\[.1cm]
&=R_\circ +E_\sU^* (I-\l S_\sU^*)^{-1}\left((I-A^*A)^{-1}-I\right)E_\sU R_\circ
\\[.1cm]
&=R_\circ +E_\sU^* (I-\l S_\sU^*)^{-1}(I-A^*A)^{-1}E_\sU R_\circ-E_\sU^* (I-\l S_\sU^*)^{-1}E_\sU R_\circ \\[.1cm]
&=R_\circ - E_\sU^* E_\sU R_\circ+ E_\sU^* (I-\l S_\sU^*)^{-1}(I-A^*A)^{-1}E_\sU R_\circ\\[.1cm]
&=E_\sU^* (I-\l S_\sU^*)^{-1}(I-A^*A)^{-1}E_\sU R_\circ,
\end{align*}
which proves \eqref{Upsilon22}.

It remains to prove the four identities in \eqref{ident12a}  and \eqref{ident21a}. Note that the second identity in \eqref{ident21a} follows from \eqref{inversion1}. Indeed,
\[
\wt{W}^*\la^{-1}\wt{B}=\wt{W}^*\la^{-1}\wt{W}E_\sU=\left((I-A^*A)^{-1}-I\right)E_\sU=(I-A^*A)^{-1}E_\sU-E_\sU.
\]
To prove the other identities we first use \eqref{inversion1} to show that
\begin{align}
A(I-A^*A)^{-1}&=W^*P^{-1}\wt{W}\left(I+\wt{W}^*\la^{-1}\wt{W}\right)\nn
\\[.1cm]
&=W^*P^{-1}\wt{W}+W^*P^{-1}\wt{W}\wt{W}^*\la^{-1}\wt{W}\nn\\[.1cm]
&=W^*P^{-1}\wt{W}+W^*P^{-1}(P-\la)\la^{-1}\wt{W}\nn\\[.1cm]
&=W^*P^{-1}\wt{W}+W^*\la^{-1}\wt{W}-W^*P^{-1}\wt{W}\nn\\[.1cm]
&=W^*\la^{-1}\wt{W}.\label{inversion3}
\end{align}
Since $W^*\la^{-1} \wt{B}=W^*\la^{-1} \wt{W}E_\sU$, formula \eqref{inversion3} yields the second identity in \eqref{ident12a}.

Next, using the general identity \eqref{gen-id} and the identity \eqref{inversion2}, we see that
\begin{equation} \label{inversion4}
(I-AA^*)^{-1}=I+A(I-A^*A)^{-1}A^*=I+W^*\la^{-1}W-W^*P^{-1}W.
\end{equation}
It follows that
\begin{align}
(I-AA^*)^{-1}W^*&=W^*+W^*\la^{-1}WW^*-W^*P^{-1}WW^*\nn \\
&=W^*+W^*\la^{-1}P-W^*P^{-1}P\nn \\
&=W^*\la^{-1}P.\label{inversion5}
\end{align}
This proves the first identity in \eqref{ident12a}. Finally, using \eqref{inversion5}, we have
\begin{align*}
A^*(I-AA^*)^{-1}W^*&=A^*W^*\la^{-1}P=\wt{W}^*P^{-1}WW^*\la^{-1}P=\wt{W}^*\la^{-1}P.
\end{align*}
Hence  the first identity in \eqref{ident21a} is
 proved.
\end{proof}

%:Sec6
\setcounter{equation}{0}
\section{Proof of  Proposition \ref{prop:propsUps} and the quotient formula for the central solution} \label{sec:propsUps}

{Throughout this section  $\{W, \wt{W}, Z\}$ is a  data set for a LTONP interpolation problem, and we assume that $\la=WW^*-\wt{W}\wt{W}^* $ is strictly positive.}

The section consists of three subsections. In the first subsection we show that  the function  $\Upsilon_{22}$   defined by \eqref{defUp22} is outer, and we derive a quotient formula for the central solution.  In the second subsection we prove our statement concerning   the  $J$-contractiveness  of the coefficient matrix   contained  in Proposition \ref{prop:propsUps}. The final statement in  Proposition \ref{prop:propsUps} about $\Upsilon_{22}^{-1}$ being a Schur class function is covered by the final part of Proposition \ref{prop:Upp22b}.
 The third subsection consists of a few remarks about the case when the operator $Z$ is exponentially stable .

\subsection{The quotient formula}\label{ssec:quotient}
{First notice  that the formulas \eqref{Hardy1} and \eqref{Hardy2} directly follow from the identities \eqref{defUp11a} -- \eqref{defUp22a}. Let us prove this for \eqref{Hardy1}.  Since $W^*$ and $\wt{W}^*$ are bounded linear operators  from $\sZ$ into $\ell_+^2(\sY)$ and $\ell_+^2(\sU)$, respectively, it follows that $W^*\la^{-1}PC^*Q_\circ$ and  $\wt{W}^*\la^{-1}PC^*Q_\circ$  are bounded linear operators mapping  $\sE$ into  $\ell_+^2(\sY)$ and $\ell_+^2(\sU)$, respectively. Thus
\[
W^*\la^{-1}PC^*Q_\circ x\in \ell_+^2(\sY)\ands \wt{W}^*\la^{-1}PC^*Q_\circ x\in \ell_+^2(\sU)\quad (x\in \sE).
\]
But then, applying \eqref{Ftransf} for $\sY$ and for $\sU$ in place of $\sY$,  we see that  the inclusions in  \eqref{Hardy1} are proved. Similar arguments prove \eqref{Hardy2}.}

\begin{proposition} \label{prop:Ups22a} The function $\Upsilon_{22}$ defined by \eqref{Upsilon22} is outer   and for each $\l\in \BD$ the operator $\Upsilon_{22}(\lambda)$ is invertible and
\begin{equation}\label{invu22}
\Upsilon_{22}(\lambda)^{-1}
=R_\circ - \lambda R_\circ
\wt{B}^*\left(I - \lambda Z^*(\Lambda +  \widetilde{B}\widetilde{B}^*)^{-1}\Lambda\right)^{-1}Z^*\la^{-1}{\wt{B}R_\circ^2.}
\end{equation}
In particular, the spectrum of $Z^*(\Lambda +  \widetilde{B}\widetilde{B}^*)^{-1}\Lambda$ is contained in the closed unit disc.
Furthermore,  the function $\Upsilon_{22}(\lambda)^{-1}$ belongs to $H^\iy(\sU,\sU)$, that is, $\Upsilon_{22}(\lambda)^{-1}$ is uniformly bounded on the open unit disk.  Finally, if $\sZ$ is finite dimensional, then both $Z^*$ and
$Z^*(\Lambda +  \widetilde{B}\widetilde{B}^*)^{-1}\Lambda$ are  exponentially stable , and $\Upsilon_{22}(\lambda)$ is an invertible outer function.
\end{proposition}

\begin{proof}[\bf Proof]
From Theorem \ref{thm:centrsol} we know that the operator
 $T=\la Z^*(\Lambda +  \wt{B}\wt{B}^*)^{-1}$  has
 spectral radius  less than or equal to one. Since  $Z^*(\la +  \wt{B}\wt{B}^*)^{-1}\Lambda = \Lambda^{-1}T \Lambda$
is similar to $T$, we see that the operator
$Z^*(\la +  \wt{B}\wt{B}^*)^{-1} \Lambda$  also has spectral radius  less than or equal to one.  In particular, $I-\l Z^*(\la +  \wt{B}\wt{B}^* )^{-1}\la$ is
invertible for each $\lambda \in \BD$.  The remaining part of the proof is done in   four steps.

\smallskip
\noindent\textsc{Step 1.}
In this part we show that for each $\l\in \BD$ the operator $\Upsilon_{22}(\lambda)$ is invertible and that its inverse is given by \eqref{invu22}. The invertibility of  $\Upsilon_{22}(\lambda)$ we already know from Theorem \ref{thm:main1}; see the paragraph directly  after  Theorem \ref{thm:main1}. Here the main point is to prove the identity \eqref{invu22}. To do this notice that
\begin{align*}
\Upsilon_{22}(\lambda)R_\circ^{-1} &= I+\wt{B}^*(I -\l Z^*)^{-1}\la^{-1}\wt{B}\\
&= I + \wt{B}^*\la^{-1}\wt{B}
+ \lambda \wt{B}^*(I -\l Z^*)^{-1}Z^*\la^{-1}\wt{B}.
\end{align*}
Recall the following  state space identity when $D$ is invertible:
\[
\left(D + \lambda C(I - \lambda A)^{-1}B\right)^{-1} =
D^{-1} - \lambda D^{-1} C \left(I - \lambda (A - B D^{-1}C)\right)^{-1}B D^{-1}.
\]
Using this with $R_\circ^2 = (I + \wt{B}^*\la^{-1}\wt{B})^{-1}$, {we see that
\begin{equation}\label{withX}
R_\circ \Upsilon_{22}(\lambda)^{-1} =
R_\circ^2 - \lambda R_\circ^2 \wt{B}^* Y(\l)^{-1}  Z^*\la^{-1}\wt{B}R_\circ^2,
\end{equation}
where
\begin{align*}
Y(\l)&=I - \l  \left(Z^* - Z^*\la^{-1} \wt{B}R_\circ^2\wt{B}^*\right)
=I - \l Z^*\left(I - \la^{-1} \wt{B}R_\circ^2\wt{B}^*\right)\\
&=I - \l Z^*\left(I - \la^{-1} \wt{B}(I + \wt{B}^*\la^{-1}\wt{B})^{-1}\wt{B}^*\right)\\
&=I - \l Z^*\left(I-\la^{-1} \wt{B}\wt{B}^*(I+\la^{-1}\wt{B}\wt{B}^*)^{-1}\right)\\
&=I - \l Z^*\left(I-\left[(I+\la^{-1}\wt{B}\wt{B}^*)-I\right](I+\la^{-1}\wt{B}\wt{B}^*)^{-1}\right)\\
&=I - \l Z^*\left(I+\la^{-1}\wt{B}\wt{B}^*\right)^{-1}=I - \l Z^*\left(\la+\ \wt{B}\wt{B}^*\right)^{-1}\la.
\end{align*}
Inserting this formula for $ Y(\l)$ into \eqref{withX} we obtain} the inverse formula for $\Upsilon_{22}(\lambda)$ in \eqref{invu22}.

\smallskip
\noindent\textsc{Step 2.}
We proceed by  proving  that the function $\Upsilon_{22}(\lambda)$ is outer. To accomplish this we use that $\Upsilon_{22}(\lambda)$  is also given by \eqref{Upsilon22}, with $A=W^*P^{-1}\wt{W}$ as in  \eqref{defA},  and we  apply  Lemma \ref{lem:outer1} in Subsection   \ref{sec-out} in the Appendix. Using $P = Z P Z^* + BB^*$ {and the fact that $P$ is strictly positive,} we see that
\[
I = P^{-\frac{1}{2}}Z P^{\frac{1}{2}}P^{\frac{1}{2}}Z^* P^{-\frac{1}{2}} +
P^{-\frac{1}{2}}BB^* P^{-\frac{1}{2}}.
\]
In particular, $P^{-\frac{1}{2}}Z P^{\frac{1}{2}}$ is a contraction.
Hence
\[
I \geq \left(P^{-\frac{1}{2}}Z P^{\frac{1}{2}}\right)^*P^{-\frac{1}{2}}Z P^{\frac{1}{2}}
= P^{\frac{1}{2}} Z^* P^{-1}Z P^{\frac{1}{2}}.
\]
Multiplying both sides by $P^{-\frac{1}{2}}$, we see that
\begin{equation}\label{invPZ}
Z^* P^{-1}Z \leq P^{-1}.
\end{equation}
Using this with $A^*A = \widetilde{W}^* P^{-1}\widetilde{W}$
and $\widetilde{W} S_\sU = Z\widetilde{W}$, we obtain
\begin{align*}
S_\sU^* A^* A S_\sU &= S_\sU^* \widetilde{W}^* P^{-1}\widetilde{W} S_\sU
= \widetilde{W}^*Z^* P^{-1}Z\widetilde{W} \leq \widetilde{W}^*  P^{-1}\widetilde{W}
= A^*A.
\end{align*}
Therefore  $S_\sU^* A^* A S_\sU \leq A^*A$. But then, according to Lemma \ref{lem:outer1}  in  Subsection  \ref{sec-out}, the function
\begin{equation}\label{defPhi6}
\Phi(\l):=E_\sU^* (I-\l S_\sU^*)^{-1}(I-A^*A)^{-1}E_\sU, \quad \l\in \BD,
\end{equation}
is outer. Because $R_\circ$ is invertible, it follows  that the function $\Upsilon_{22}(\lambda)=\Phi(\l)R_\circ$ is outer too.

\smallskip
\noindent\textsc{Step 3.}
Let $\Phi$ be given by \eqref{defPhi6}. Since $\Upsilon_{22}(\lambda)$ is invertible for each $\l \in \BD$ and $R_\circ$ is invertible,  the operator $\Phi(\l)$  is also invertible for each $\l \in \BD$. But then the final part  of Lemma \ref{lem:outer1} tells us that the function $\Phi(\l)^{-1}$ belongs to $H^\iy(\sU,\sU)$. But then $\Upsilon_{22}(\l)^{-1}=R_\circ^{-1}\Phi(\l)^{-1}$ also belongs to $H^\iy(\sU,\sU)$.

\smallskip
\noindent\textsc{Step 4.}
Finally, assume $\sZ$ is finite dimensional. Since $Z^*(\Lambda +  \widetilde{B}\widetilde{B}^*)^{-1}\Lambda$ is similar to $T=\la Z^*(\Lambda +  \wt{B}\wt{B}^*)^{-1}$, we have $\spec(Z^*(\Lambda +  \widetilde{B}\widetilde{B}^*)^{-1}\Lambda)=\spec(T)<1$; note that $\spec(T)<1$ follows from Theorem \ref{thm:centrsol}. Furthermore, $Z^*$ is pointwise stable, by part (ii) of Lemma \ref{L:useobs1}, which implies all eigenvalues of $Z^*$ are contained in $\BD$. Hence $\spec(Z)=\spec(Z^*)<1$. This yields that $\Upsilon_{22}$ is an invertible outer function.
\end{proof}

The next proposition shows that for the strictly positive case   the definition  of the  central solution $F_\circ$ to the LTONP interpolation problem  given in    Remark \ref{rem:centrsol} coincides with the one given in the paragraph  directly after Theorem \ref{thm:main1}.  The proposition also justifies  the title of this subsection.

\begin{proposition}\label{quotientf}
Let $F_\circ$ be the  central solution  of the  LTONP problem with data set $\{W, \widetilde{W}, Z\}$. If the Pick operator  $\Lambda$ is strictly positive, then  $F_\circ$  is given by the quotient formula:
\begin{equation}\label{quotientid}
F_\circ(\l)=\Upsilon_{12}(\l) \Upsilon_{22}(\l)^{-1}, \quad \l \in \BD.
\end{equation}
In other words, when the free parameter $X$ in \eqref{allsol1a} is zero, then the resulting function is the central solution.
\end{proposition}

\begin{proof}[\bf Proof] By using  \eqref{defUp12} and \eqref{defUp22}, we obtain
\begin{align*}
\Upsilon_{12}(\l) \Upsilon_{22}(\l)^{-1}&=
B^*(I -\l Z^*)^{-1}\la^{-1}\wt{B}
\left(I  +\wt{B}^*(I -\l Z^*)^{-1}\la^{-1}\wt{B}\right)^{-1}\\
&= B^*\left(I  +(I -\l Z^*)^{-1}\la^{-1}\wt{B}\wt{B}^*\right)^{-1}
(I -\l Z^*)^{-1}\la^{-1}\wt{B}\\
&=B^*\left(\Lambda - \l \Lambda Z^*  + \wt{B}\wt{B}^*\right)^{-1}\wt{B}\\
&= B^*(\Lambda + BB^*)^{-1}
\left(I - \l \Lambda Z^*(\Lambda + BB^*)^{-1}\right)^{-1}\wt{B}\\
&= F_\circ(\lambda).
\end{align*}
The last equality follows from  formula \eqref{Fcentral} for the central solution
$F_\circ(\lambda)$  in Theorem \ref{thm:centrsol}.
\end{proof}

\begin{proposition}\label{prop:Upp22b}
Let $F_\circ$ be the  central solution  of the  LTONP problem with data set $\{W, \widetilde{W}, Z\}$, with  the Pick operator  $\Lambda$ being  strictly positive, and let $\Upsilon_{22}^{-1}$ be given by \eqref{Upsilon22}. Then the functions $F_\circ$ and $\Upsilon_{22}^{-1}$ are both uniformly bounded on $\BD$  in operator norm, and   the corresponding Toeplitz operators satisfy the following identity:
\begin{equation}\label{eq:specfac}
I-T_{F_\circ}^*T_{F_\circ}=T_{\Upsilon_{22}^{-1}}^*T_{\Upsilon_{22}^{-1}}
\end {equation}
Furthermore, both $F_\circ$ and $\Upsilon_{22}^{-1}$ are Schur class functions.
\end{proposition}
\begin{proof}[\bf Proof]
Since $F_\circ$ is a solution to the LTONP interpolation problem,  $F_\circ$ is a Schur class function. In particular, the function $F_\circ$ is uniformly bounded on $\BD$ in operator norm. The latter  also holds true for $\Upsilon_{22}^{-1}$ by Proposition \ref{prop:Ups22a}.

Let us assume that \eqref{eq:specfac}  is proved. Since $F_\circ $ is  a Schur class function, it follows that $T_{F_\circ}$ is a contraction. But then the identity \eqref{eq:specfac} implies that $\|T_{\Upsilon_{22}^{-1}}^*T_{\Upsilon_{22}^{-1}}\|\leq1$. Hence  the Toeplitz operator $T_{\Upsilon_{22}^{-1}}$    is a contraction  too.   The latter implies that $\Upsilon_{22}^{-1}$ is a Schur class function.  Thus the final statement of the proposition is proved.

It remains to prove  \eqref{eq:specfac}. Recall that $\Upsilon_{22} = \Phi R_\circ$, where the function  $\Phi$  is  given by \eqref{defPhi6} and
$R_\circ = \left(E_{\sU}^* (I - A^*A)^{-1} E_{\sU}\right)^{-\frac{1}{2}}$.
Here   $A = W^* P^{-1} \widetilde{W}$, and hence   $W A = \widetilde{W}$.
 We claim that
\begin{equation}\label{int6}
\langle S_\sY A h, A f\rangle  =  \langle A  S_\sU h, A f\rangle,\quad
 h,f \in \ell_+^2(\sU).
\end{equation}
Using $Z W = W S_\sY$ and  $Z \widetilde{W} = \widetilde{W} S_\sU$, we obtain
\begin{align*}
\langle S_\sY A h, A f\rangle &= \langle W S_\sY A h, P^{-1} \widetilde{W} f\rangle
=\langle Z W   A h, P^{-1} \widetilde{W} f\rangle\\
&= \langle Z\widetilde{W}  h, P^{-1} \widetilde{W} f\rangle =
\langle \widetilde{W} S_\sU h, P^{-1} \widetilde{W} f\rangle\\
&= \langle W A  S_\sU h, P^{-1} \widetilde{W} f\rangle =
\langle A  S_\sU h, A f\rangle.
\end{align*}
This yields \eqref{int6}.

Next, let  $x\in \ell^2_+(\sU)$ be of compact support, that is, $x$ has only a finite number of non-zero entries. We shall  show that for any such $x$ we have
\begin{equation}\label{ImpId}
\|T_\Phi x\|^2 - \|T_{F_\circ} T_\Phi x\|^2 = \|T_{\Upsilon_{22}^{-1}} T_\Phi x\|^2.
\end{equation}
Recall that the central solution  $F_\circ$ is given by the quotient formula \eqref{quotientid} $F_\circ(\lambda) =  \Upsilon_{12}(\lambda)\Upsilon_{22}(\lambda)^{-1}$, where $\Upsilon_{12}$ and $\Upsilon_{22}$ are defined in
\eqref{Upsilon12} and \eqref{Upsilon22}, respectively. Thus
$F_\circ(\lambda) \Upsilon_{22}(\lambda) =  \Upsilon_{12}(\lambda)$ for each $\l\in \BD$. By eliminating $R_\circ$ in the definitions of $\Upsilon_{12}$ and $\Upsilon_{22}$,
we see that
\[
F_\circ(\lambda)E_\sU^*(I - S_\sU^*)^{-1}D_A^{-2} E_\sU =
E_\sY^*(I - S_\sY^*)^{-1}AD_A^{-2} E_\sU,
\]
where $D_A=(I-A^*A)^{\frac{1}{2}}$. So for $x=\{x_n\}_{n=0}^\infty$ in $\ell_+^2(\sU)$ with compact support, we have
\small\begin{align*}
\|T_{F_\circ} T_\Phi x\|^2
&=\|T_{F_\circ} \sum_{n=0}^\infty S_\sU^n D_A^{-2}E_\sU x_n\|^2 =
\|\sum_{n=0}^\infty S_\sY^n A D_A^{-2}E_\sU x_n\|^2\\
&= \langle\sum_{n=0}^\infty S_\sY^n A D_A^{-2}E_\sU x_n,
\sum_{m=0}^\infty S_\sY^m A D_A^{-2}E_\sU x_m\rangle \\
&= \sum_{n \geq m}\langle S_\sY^n A D_A^{-2}E_\sU x_n,S_\sY^m A D_A^{-2}E_\sU x_m\rangle +\\
&\qquad\qquad\qquad + \sum_{n < m}\langle S_\sY^n A D_A^{-2}E_\sU x_n,S_\sY^m  A D_A^{-2}E_\sU x_m\rangle.\\
&= \sum_{n \geq m}\langle S_\sY^{n-m} A D_A^{-2}E_\sU x_n, A D_A^{-2}E_\sU x_m\rangle +\\
&\qquad\qquad\qquad + \sum_{n < m}\langle A D_A^{-2}E_\sU x_n,S_\sY^{m-n}  A D_A^{-2}E_\sU x_m\rangle.
\end{align*}
Using the fact that $A^*A D_A^{-2}=(D_A^{-2}-I)$ we obtain
\begin{align*}
&\sum_{n > m}\langle S_\sY^{n-m} A D_A^{-2}E_\sU x_n, A D_A^{-2}E_\sU x_m\rangle
=\sum_{n > m}\langle A S_\sU^{n-m}D_A^{-2}E_\sU x_n, A D_A^{-2}E_\sU x_m\rangle\\
&\qquad=\sum_{n > m}\langle A S_\sU^{n-m}D_A^{-2}E_\sU x_n, A D_A^{-2}E_\sU x_m\rangle\\
&\qquad =\sum_{n > m}\langle  S_\sU^{n-m}D_A^{-2}E_\sU x_n,  (D_A^{-2} - I)E_\sU x_m\rangle\\
&\qquad
= \sum_{n > m}\langle S_\sU^{n-m}D_A^{-2}E_\sU x_n,  D_A^{-2}E_\sU x_m\rangle
= \sum_{n > m}\langle S_\sU^{n}D_A^{-2}E_\sU x_n, S_\sU^{m} D_A^{-2}E_\sU x_m\rangle.
\end{align*}
A similar computation gives
\[
\sum_{n < m}\langle S_\sY^n A D_A^{-2}E_\sU x_n,S_\sY^m  A D_A^{-2}E_\sU x_m\rangle=
\sum_{n < m}\langle S_\sU^n D_A^{-2}E_\sU x_n,S_\sU^m D_A^{-2}E_\sU x_m\rangle.
\]
For $m=n$ we have
\begin{align*}
\langle A D_A^{-2}E_\sU x_n, A D_A^{-2}E_\sU x_n\rangle
&=\langle A D_A^{-2}E_\sU x_n, A D_A^{-2}E_\sU x_n\rangle\\
&=\langle D_A^{-2}E_\sU x_n, (D_A^{-2}-I)E_\sU x_n\rangle\\
&=\langle D_A^{-2}E_\sU x_n, D_A^{-2}E_\sU x_n\rangle
-\langle D_A^{-2}E_\sU x_n,E_\sU x_n\rangle\\
&=\langle D_A^{-2}E_\sU x_n, D_A^{-2}E_\sU x_n\rangle
-\langle R_\circ^{-1} x_n, x_n\rangle.
\end{align*}
Putting the above computations together gives
\begin{align*}
\|T_{F_\circ} T_\Phi x\|^2
&=\sum_{n,m=0}^\infty\langle S_\sU^n D_A^{-2}E_\sU x_n,S_\sU^m D_A^{-2}E_\sU x_m\rangle
-\sum_{n=0}^\infty\langle R_\circ^{-2} x_n, x_n\rangle\\
&=\langle T_\Phi x , T_\Phi x \rangle -\sum_{n=0}^\infty\| R_\circ^{-1} x_n\|^2
=\| T_\Phi x\|^2 - \| T_{R_\circ^{-1}} x\|^2\\
&=\| T_\Phi x\|^2 - \| T_{R_\circ^{-1}}T_{\Phi^{-1}} T_\Phi x\|^2
=\| T_\Phi x\|^2 - \|T_{\Upsilon_{22}^{-1}} T_\Phi x\|^2.
\end{align*}
Here $T_{R_\circ^{-1}}$ denotes the diagonal Toeplitz operator with the operator $R_\circ^{-1}$ on the main diagonal. We proved \eqref{ImpId} for all $x$ in $\ell_+^2(\sU)$ with compact support. The fact that $\Phi$ is outer implies that $T_\Phi$ maps the compact support sequences in $\ell^2_+(\sU)$ to a dense subset of $\ell^2_+(\sU)$. Therefore
\[
\|v\|^2 - \|T_{F_\circ} v\|^2 = \|T_{\Upsilon_{22}^{-1}} v\|^2,\quad v\in\ell^2_+(\sU).
\]
In other words,
$I - T_{F_\circ}^*T_{F_\circ} = T_{\Upsilon_{22}}^{-*}T_{\Upsilon_{22}}^{-1}$,
and  \eqref{eq:specfac}.
\end{proof}

\subsection{$J$-contractiveness of the coefficient matrix}\label{ssec:Jcontact}

Throughout this section let $\{W,\wtil{W},Z\}$ be a data set for a LTONP interpolation problem. Assume $\la=P-\wtil{P}$ is strictly positive. Define $\Up_{ij}$, $i,j=1,2$, as in \eqref{defUp11}--\eqref{defUp22}. Now set
\begin{equation}\label{Up}
\Up(\l)=\mat{cc}{\Up_{11}(\l)&\Up_{12}(\l)\\ \Up_{21}(\l) & \Up_{22}(\l)}\quad (\l\in\BD).
\end{equation}
Furthermore, set
\[
J_1=\mat{cc}{I_\sY&0\\0& -I_{\sU}}\ands
J_2=\mat{cc}{I_\sE&0\\0& -I_{\sU}}.
\]
The following theorem is the main result of this section.

\begin{theorem}\label{T:J-cont}
Let $\{W,\wtil{W},Z\}$ be a data set for a LTONP interpolation problem. Assume $\la=P-\wtil{P}$ is strictly positive. Then for each $\l\in\BD$ the operator  $\Up(\l) $  is $J$-contractive, that is, $\Up(\l)^* J_1 \Up(\l)\leq J_2$. More precisely, {for each $\l\in\BD$}  we have
\begin{align}\label{UpJineq}
&\Up(\l)^* J_1 \Up(\l)
=J_2+\\
&\quad -(1-|\l|^2)\mat{c}{{Q_\circ} CP\la^{-1}\\ R_\circ\wtil{B}^*\la^{-1}Z}(I-\l Z^*)^{-*}\la (I-\l Z^*)^{-1}\times \notag\\
&\hspace{5cm} \times \mat{cc}{\la^{-1}P C^*Q_\circ & Z^*\la^{-1}\wtil{B}R_\circ}.\notag
\end{align}
Furthermore,   for each $\l$ on the unit circle that  is  not in the spectrum of $Z$ the operator $\Up(\l) $  is $J$-unitary, that is, $\Up(\l)^* J_1 \Up(\l)= J_2$.
\end{theorem}

\begin{remark} Theorem \ref{T:J-cont} can be also used to
show that $\Upsilon_{22}^{-1}$ is a function in $\sS(\sU,\sU)$.
Indeed, the inequality  $\Up(\l)^* J_1 \Up(\l)\leq J_2$,
implies that
\[
\Upsilon_{12}(\lambda)^* \Upsilon_{12}(\lambda) -
\Upsilon_{22}(\lambda)^* \Upsilon_{22}(\lambda) \leq -I
\qquad(\lambda \in \mathbb{D}).
\]
Thus $I \leq \Upsilon_{22}(\lambda)^* \Upsilon_{22}(\lambda)$
for each $\lambda$ in $\mathbb{D}$. Proposition \ref{prop:Ups22a}
shows that $\Upsilon_{22}(\lambda)$ is invertible for $\lambda$ in $\mathbb{D}$.
Hence  $\Upsilon_{22}(\lambda)^{-*}\Upsilon_{22}(\lambda)^{-1} \leq I$
for $\lambda$ in $\mathbb{D}$.  Therefore  $\Upsilon_{22}(\lambda)^{-1}$
is a contraction for all $\lambda$ in $\mathbb{D}$. In other words,
$\Upsilon_{22}(\lambda)^{-1}$ is a function  in $\sS(\sU,\sU)$.
\end{remark}

Before we prove this result it is useful to first {derive} the following two lemmas. The first lemma provides a state space realization for the coefficient matrix-function $\Up$, the second lemma derives a number of useful identities of the operators involved in the realization.

\begin{lemma}\label{L:Uprealform}
The function $\Up$ in \eqref{Up} is given by
\begin{equation}\label{Uprealform}
\Up(\l)= \left(\what{D}+\l \what{C}(I-\l Z^*)^{-1}\what{B}\right) \mat{cc}{Q_\circ & 0\\0 & R_\circ},\quad  \l\in\BD,
\end{equation}
where $\what{B}$, $\what{C}$ and $\what{D}$ are the operators given by
\[
\what{B}=\mat{cc}{\la^{-1}PC^* & Z^* \la^{-1}\wtil{B}},\
\what{C}=\mat{c}{B^*\\ \wtil{B}^*},\
\what{D}=\mat{cc}{D^*& B^*\la^{-1}\wtil{B}\\ 0 & I+\wtil{B}^*\la^{-1}\wtil{B}}.
\]
\end{lemma}

\begin{proof}[\bf Proof]
By writing out the right-hand side of \eqref{Uprealform} in $2 \times 2$ block matrix form, we see that the left upper block and left lower block coincide with $\Upsilon_{11}$ and $\Upsilon_{21}$ in \eqref{defUp11} and \eqref{defUp21}, respectively. It remains to show that $\Upsilon_{12}$ in \eqref{defUp12} and $\Upsilon_{22}$ in \eqref{defUp22} can be written as
\begin{align*}
\Upsilon_{12}(\lambda)& =
 (B^*\la^{-1}\wtil{B}+\l B^*(I-\l Z^*)^{-1}Z^*\la^{-1}\wtil{B})R_\circ;\\
\Upsilon_{22}(\lambda)& =
(I+ \wtil{B}^*\la^{-1}\wtil{B}+\l \wtil{B}^*(I-\l Z^*)^{-1}Z^*\la^{-1}\wtil{B})R_\circ.
\end{align*}
In both cases this is a direct consequence of the fact that
\[
(I-\l Z^*)^{-1}=I+\l (I-\l Z^*)^{-1}Z^*.\qedhere
\]
\end{proof}

\begin{lemma}
With  $\what{B}$, $\what{C}$, $\what{D}$ and $J_1$ defined as above, we have the following identities:
\begin{align}
\label{Ids1}
& \what{D}^* J_1\what{D}=\mat{cc}{Q_\circ^{-2}&0\\0& -R_\circ^{-2}}
-\what{B}^*\la\what{B}\\
\label{Ids2}
& \what{C}^* J_1\what{C}=\la-Z\la Z^*\ands
\what{D}^* J_1\what{C}=-\what{B}^*\la Z^*,
\end{align}

\end{lemma}

\begin{proof}[\bf Proof]
Recall that
\[
BB^*-\wtil{B}\wtil{B}^*=\la-Z\la Z^*.
\]
The identities in \eqref{Ids2} follow from this identity and the following straightforward computations:
\begin{align}
\what{C}^* J_1\what{C} &=BB^*-\wtil{B}\wtil{B}^*=\la-Z\la Z^*,\label{Ids3}\\
\what{D}^* J_1\what{C}
&=\mat{cc}{D&0\\ \wtil{B}^*\la^{-1}B&  I+\wtil{B}^*\la^{-1}\wtil{B}}\mat{c}{B^*\\ -\wtil{B}^*}\nn\\
& =\mat{c}{DB^*\\ {\wtil{B}^*}\la^{-1}(BB^* -\la-\wtil{B}\wtil{B}^*)}\nn\\
&=\mat{c}{-CPZ^*\\-{\wtil{B}^*}\la^{-1}Z\la Z^*}=-\mat{c}{CP\la^{-1}\\\wtil{B}^*\la^{-1} Z}\la Z^*
=-\what{B}^*\la Z^*.\label{Ids4}
\end{align}
In establishing the first identity on  the last line  we used $DB^*+CPZ^*=0$, which follows from \eqref{semiunit1}.  Using
$DD^* + C P C^* =I$ from \eqref{semiunit1}, we have
\begin{align*}
\what{D}^* J_1\what{D}
&=\mat{cc}{D&0\\ \wtil{B}^*\la^{-1}B& {I+}\wtil{B}^*\la^{-1}\wtil{B}} \mat{cc}{D^*& B^*\la^{-1}\wtil{B}\\ 0 &-(I+\wtil{B}^*\la^{-1}\wtil{B})}  \\
&=\mat{cc}{DD^* & DB^*\la^{-1}\wtil{B}\\
\wtil{B}^*\la^{-1}BD^* &  \wtil{B}^*\la^{-1}B B^*\la^{-1}\wtil{B} -(I+\wtil{B}^*\la^{-1}\wtil{B})^2}\\
&=\mat{cc}{I-CPC^* & -CPZ^*\la^{-1}\wtil{B}\\ -\wtil{B}^*\la^{-1}ZPC^* &
\wtil{B}^*\la^{-1}B B^*\la^{-1}\wtil{B} -(I+\wtil{B}^*\la^{-1}\wtil{B})^2}.
\end{align*}
Next observe that
\[
I-CPC^*=I+CP(\la^{-1}-P^{-1})PC^*-CP\la^{-1}PC^*=Q_\circ^{-2} -CP\la^{-1}PC^*
\]
and
\begin{align*}
(I+\wtil{B}^*\la^{-1}\wtil{B})^2
&=(I+\wtil{B}^*\la^{-1}\wtil{B})+ \wtil{B}^*\la^{-1}\wtil{B}(I+\wtil{B}^*\la^{-1}\wtil{B})\\
&=R_\circ^{-2}+\wtil{B}^*\la^{-1}(\la +\wtil{B}\wtil{B}^*)\la^{-1}\wtil{B}\\
&=R_\circ^{-2}+\wtil{B}^*\la^{-1}(Z\la Z^*+BB^*)\la^{-1}\wtil{B}\\
&={R_\circ^{-2}+\wtil{B}^*\la^{-1}Z\la Z^*\la^{-1}\wtil{B}+
\wtil{B}^*\la^{-1}BB^*\la^{-1}\wtil{B}.}
\end{align*}
Hence
\[
{(I+\wtil{B}^*\la^{-1}\wtil{B})^2- \wtil{B}^*\la^{-1}B B^*\la^{-1}\wtil{B}}=
R_\circ^{-2}+\wtil{B}^*\la^{-1}Z\la Z^*\la^{-1}\wtil{B}.
\]
Using these identities we obtain that
\begin{align*}
\what{D}^* J_1\what{D} & =
\mat{cc}{Q_\circ^{-1} & 0\\ 0& -R_\circ^{-2}}
-\mat{cc}{CP\la^{-1}PC^* & CPZ^*\la^{-1}\wtil{B}\\ \wtil{B}^*\la^{-1}ZPC^* &{\wtil{B}^*\la^{-1}Z\la Z^*\la^{-1}\wtil{B}}}\\
&{=\mat{cc}{Q_\circ^{-1} & 0\\ 0&-R_\circ^{-2}}
-\mat{c}{CP\la^{-1}\\\wtil{B}^*\la^{-1}Z}\la \mat{cc}{\la^{-1}PC^* & Z^* \la^{-1} \wtil{B}}}.
\end{align*}
This shows that \eqref{Ids1} holds as well.
\end{proof}

\begin{proof}[\bf Proof of Theorem \ref{T:J-cont}] Fix  a  $\l \in \BD$.  In order to prove \eqref{UpJineq}, we multiply the left hand side of \eqref{UpJineq}  from  both sides by $\sbm{Q_\circ& 0\\ 0& R_\circ}^{-1}$.   Then, by using
\eqref{Uprealform},  we  obtain
\begin{align*}
&\mat{cc}{Q_\circ& 0\\ 0& R_\circ}^{-1} \Up(\l)^* J_1 \Up(\l)\mat{cc}{Q_\circ& 0\\ 0& R_\circ}^{-1} =\\
&\quad =(\what{D}^*+\overline{\l}\what{B}^*(I-\overline{\l} Z)^{-1}\what{C}^*)J_1(\what{D}+\l \what{C}(I-\l Z^*)^{-1}\what{B})\\
&\quad
= \what{D}^*J_1\what{D}
+\overline{\l}\what{B}^*(I-\overline{\l} Z)^{-1}\what{C}^* J_1 \what{D}
+\l \what{D}^* J_1 \what{C}(I-\l Z^*)^{-1}\what{B}+\\
&\hspace{3cm}+|\l|^2\what{B}^*(I-\overline{\l} Z)^{-1}\what{C}^* J_1 \what{C}(I-\l Z^*)^{-1}\what{B}\\
&\quad= \what{D}^*J_1\what{D}
-\overline{\l}\what{B}^*(I-\overline{\l} Z)^{-1}Z\la \what{B}
-\l \what{B}^*\la Z^*(I-\l Z^*)^{-1}\what{B}+\\
&\hspace{3cm}+|\l|^2\what{B}^*(I-\overline{\l} Z)^{-1}(\la -Z\la Z^*)(I-\l Z^*)^{-1}\what{B}\\
&\quad= \what{D}^*J_1\what{D}-\what{B}^*(I-\overline{\l} Z)^{-1}\times\\
&\quad \times\Big(\overline{\l}Z\la (I-\l Z^*)+\l(I-\overline{\l}Z)\la Z^* -|\l|^2(\la-Z\la Z^*)\Big)
(I-\l Z^*)^{-1}\what{B}.
\end{align*}
Note that
\begin{align*}
& \overline{\l}Z\la (I-\l Z^*)+\l(I-\overline{\l}Z)\la Z^* -|\l|^2(\la-Z\la Z^*)=\\
& \qquad= -(|\l|^{2}\la -\overline{\l}Z\la-\l\la Z^*+|\l|^2Z\la Z^*)\\
& \qquad= -(I-\overline{\l} Z)\la (I-\l Z^*)+({I}-|\l|^2)\la.
\end{align*}
Inserting this identity into the above computation yields
\begin{align*}
&\mat{cc}{Q_\circ& 0\\ 0& R_\circ}^{-1} \Up(\l)^* J_1 \Up(\l)\mat{cc}{Q_\circ& 0\\ 0& R_\circ}^{-1} =\\
&\qquad =\what{D}^*J_1\what{D}+\what{B}^*\la\what{B}
-(1-|\l|^2)\what{B}^*(I-\overline{\l} Z)^{-1}\la (I-\l Z^*)^{-1}\what{B}\\
&\qquad =\mat{cc}{Q_\circ^{-2}&0\\0& -R_\circ^{-2}}
-(1-|\l|^2)\what{B}^*(I-\overline{\l} Z)^{-1}\la (I-\l Z^*)^{-1}\what{B}.
\end{align*}
Multiplying {the resulting identity from both sides} by  $\sbm{Q_\circ&0\\0&R_\circ}$  {yields} \eqref{UpJineq}.

{By taking limits the  final statement directly follows  from   \eqref{UpJineq}.}
\end{proof}

\subsection{The case when $Z$ is exponentially stable }\label{ssec: strongstable}
%%%%%%%%%%%%%%%%%%%%%%%%%%%%%%%%
We conclude this section with a few remarks about the case  when $Z$ is exponentially stable . Note that this happens when $\sZ$ is finite dimensional.  Recall that $Z$ is
 exponentially stable  if     $\spec(Z)$, the spectral radius of $Z$, is strictly less than one.

\begin{proposition}\label{prop:stable}Assume that the operator $Z$ is exponentially stable .  Then the  operator $Z^*(\la +  \wt{B}\wt{B}^*)^{-1}\la$ is also exponentially stable , and therefore  the functions  $\Upsilon_{ij}(\lambda)$, $i,j=1,2$,  the central solutions $F_\circ$  and the function $\Upsilon_{22}(\lambda)^{-1}$ are analytic on $|\l| <1+\epsilon$ for some $\epsilon>0$.
Furthermore,
\begin{equation}\label{specfact}
I-F_\circ(\l)^*F_\circ(\l)=\Upsilon_{22}(\lambda)^{-*}\Upsilon_{22}(\lambda)^{-1},  \quad |\l |=1.
\end{equation}
Finally, $T_{F_\circ}$ is a strict contraction.
\end{proposition}

\begin{proof}[\bf Proof]
We first show that   $Z^*(\la +  \wt{B}\wt{B}^*)^{-1}\la$ is exponentially stable . Notice that
\begin{align*}
\Lambda^{\frac{1}{2}}\left(Z^*(\la +  \wt{B}\wt{B}^*)^{-1}\la\right)\Lambda^{-\frac{1}{2}}&=
\la^{\frac{1}{2}} Z^*\Lambda^{-\frac{1}{2}}\Lambda^{\frac{1}{2}}
(\la +  \wt{B}\wt{B}^*)^{-1}\la^{\frac{1}{2}}\\
&=\la^{\frac{1}{2}} Z^*\la^{-\frac{1}{2}}
\left(I+  \la^{-\frac{1}{2}} \wt{B}\wt{B}^*\la^{-\frac{1}{2}}\right)^{-1}.
\end{align*}
Hence $\Lambda^{\frac{1}{2}}\left(Z^*(\la +  \wt{B}\wt{B}^*)^{-1}\la\right)\Lambda^{-\frac{1}{2}}$
and $\la^{\frac{1}{2}} Z^*\la^{-\frac{1}{2}}
\left(I+  \la^{-\frac{1}{2}} \wt{B}\wt{B}^*\la^{-\frac{1}{2}}\right)^{-1}$ \!
are similar. In particular, they have the same spectrum. Furthermore,
$\Lambda - Z \Lambda Z^* = BB^* - \widetilde{B}\widetilde{B}^*$
can be rewritten as
\[
I-\la^{-\frac{1}{2}}Z\la^{\frac{1}{2}} \la^{\frac{1}{2}}Z^*\la^{-\frac{1}{2}}=
\la^{-\frac{1}{2}}BB^*\la^{-\frac{1}{2}}-\la^{-\frac{1}{2}}\wt{B}\wt{B}^*\la^{-\frac{1}{2}}.
\]
Replacing  $\la^{-\frac{1}{2}} Z \la^{\frac{1}{2}}$ by $Z$ and  $\la^{-\frac{1}{2}}B$ by $B$ and  $\la^{-\frac{1}{2}}\wt{B}$ by $\wt{B}$, we see that without loss of generality we may assume that $\la =I$.

So we  assume that
\begin{equation}\label{condspec1}
\spec(Z)  < 1\ands I- ZZ^*+\wt{B} \wt{B}^*  = BB^*\geq 0 .
\end{equation}
We have to show that   $Z^*(I+\wt{B}\wt{B}^*)^{-1}$
is exponentially stable . By consulting \eqref{OmPF} with $\Lambda =I$, we see that
$\Pi_\sZ \omega P_\sF|\sZ = Z^*(I+\wt{B}\wt{B}^*)^{-1}$.
Hence $Z^*(I+\wt{B} \wt{B}^*)^{-1}$ is a contraction, and thus, \begin{equation} \label{radiusspec}
 r_{\rm spec}(Z^*(I+\wt{B}\wt{B}^*)^{-1}) \leq 1.
\end{equation}

Next consider  the auxiliary operator
\begin{equation} \label{defX}
Y = I + \wt{B}^*(I-Z^*)^{-1}\wt{B}: \sU
  \to \sU.
\end{equation}
We shall show that $Y$ is invertible. The idea of the proof is
taken from \cite{Kos}, page 128. One computes that
\begin{align*}
 YY^* & = \{I+\wt{B}^*(I-Z^*)^{-1} \wt{B} \} \{I+\wt{B}^*(I-Z)^{-1} \wt{B} \} \\
  & = I + \wt{B}^* (I-Z^*)^{-1}\wt{B} + \wt{B}^* (I-Z)^{-1} \wt{B}+\\
  &\hspace{3cm}+
  \wt{B}^*(I-Z^*)^{-1}\wt{B} \wt{B}^*(I-Z)^{-1} \wt{B}  \\
  & = I + \wt{B}^*(I-Z^*)^{-1} \{(I-Z)+(I-Z^*)+\wt{B} \wt{B}^* \} (I-Z)^{-1} \wt{B}.
\end{align*}
Now use the second part of \eqref{condspec1} and
\[
ZZ^*-I = (I-Z)(I-Z^*)-(I-Z^*)-(I-Z).
\]
It follows that
\begin{align*}
(I-Z)+(I-Z^*)+\wt{B} \wt{B}^* &=(I-Z)(I-Z^*)+I-ZZ^*+\wt{B} \wt{B}^*\\
&= (I-Z)(I-Z^*)+BB^*\geq 0.
\end{align*}
Hence  $YY^* \geq I$, and $YY^*$  is strictly positive. In a similar fashion one computes that
\begin{align*}
Y^*Y&=I+\wt{B}^* (I-Z)^{-1} \{(I-Z^*)+(I-Z)+\wt{B}\wt{B}^* \} (I-Z^*)^{-1}\wt{B}\\
%&=I+\wt{B}^* (I-Z)^{-1} \{(I-Z)(I-Z^*)+I-ZZ^*+\wt{B} \wt{B}^*\}(I-Z^*)^{-1}\wt{B}\\
&=I+\wt{B}^* (I-Z)^{-1} \{(I-Z)(I-Z^*)+BB^*\}(I-Z^*)^{-1}\wt{B}\geq I.
\end{align*}
Thus $Y^*Y$ is   also strictly positive.  Since both $Y^*Y$ and $YY^*$ are strictly positive, we  conclude that the operator $Y$ defined by \eqref{defX} is invertible.

%XXXXXXXXXXXXXXXXXX

Since $Y$  defined by \eqref{defX} is invertible,
it follows that $I+(I-Z^*)^{-1}\wt{B} \wt{B}^*$ is invertible. Here we used the fact that the nonzero spectrum of the product of two operators are the same.
Multiplying by  $I -Z^*$  on the left shows that  $I-Z^*+\wt{B} \wt{B}^*$ is also invertible.  Multiplying by $(I + \widetilde{B}\widetilde{B}^*)^{-1}$ on the right
 implies  that $I - Z^*(I+\wt{B} \wt{B}^*)^{-1}$ is invertible. In other words,
\begin{equation} \label{onenot}
 1 \notin \sigma\left(Z^*(I+\wt{B} \wt{B}^*)^{-1}\right).
\end{equation}
Recall that  $\sigma(A)$ denotes the spectrum of an operator $A$.  Now take $\lambda \in {\BT}$, and notice that the conditions in \eqref{condspec1} remain valid if $Z$ is replaced by $\l  Z$. Thus \eqref{onenot} yields
\[
 1 \notin \sigma\left(\lambda^{-1} Z^*(I+\wt{B} \wt{B}^*)^{-1}\right).
\]
It follows that $\lambda \notin \sigma(Z^*(I+\wt{B} \wt{B}^*)^{-1}).$ Since $\lambda$ is an arbitrary element of ${\BT}$, we conclude that
$\sigma(Z^*(I+\wt{B} \wt{B}^*)^{-1})\cap\BT$
is empty, and hence using  \eqref{radiusspec} we obtain that
the spectral radius of $Z^*(I+{\wt{B}}{\wt{B}}^*)^{-1}$
is strictly less than one.

Since  both $Z$ and    $Z^*(\la +  \wt{B}\wt{B}^*)^{-1}\la$ are exponentially stable ,  it is clear    from  \eqref{defUp11}  to  \eqref{defUp22},  \eqref{Fcentral} and \eqref{invu22} that the functions  $\Upsilon_{ij}(\lambda)$, $i,j=1,2$,  the central solutions $F_\circ$  and the function $\Upsilon_{22}(\lambda)^{-1}$ are analytic on $|\l| <1+\epsilon$ for some $\epsilon>0$.

Next we prove \eqref{specfact}. Fix $\l \in \BT$. Since  $\spec(Z)<1$,  the final statement of Theorem \ref{T:J-cont} tells is that
 \[
\Upsilon_{12}(\l)^*\Upsilon_{12}(\l)-\Upsilon_{22}(\l)^*\Upsilon_{22}(\l)=-I_\sU.
\]
Multiplying the latter identity from the right by  $\Upsilon_{22}(\l)^{-1}$ and from the left by $\Upsilon_{22}(\l)^{-*}$ and using  the quotient formula \eqref{quotientid} we see that
\[
F_\circ(\l)^* F_\circ(\l)-I=-\Upsilon_{22}(\l)^{-*}\Upsilon_{22}(\l)^{-1},
\]
 which proves \eqref{specfact}.

Finally, using \eqref{specfact}, we see that $\|T_{F_\circ} \|= \sup_{\l \in \BD}\|F_\circ(\l)\|<1$, and hence $T_{F_\circ}$ is a strict contraction.
\end{proof}

 \begin{corollary} Let $\{W,\wtil{W},Z\}$ be a data set for a LTONP interpolation problem, and let $\la=P-\wtil{P}$ be strictly positive.  If in addition  $\sZ$ is finite dimensional, then the operator $Z$ is exponentially stable , and the functions  $\Upsilon_{ij}(\lambda)$, $i,j=1,2$,  the central solutions $F_\circ$,  and the function $\Upsilon_{22}(\lambda)^{-1}$ are  rational operator  functions  with no poles on the closed unit disk and the factorization in \eqref{specfact} is a right canonical factorization.  in the sense of  \cite[Section XXIV3]{GGK1}. In other words, $\Upsilon_{22}$  is invertible outer, that is, $T_{\Upsilon_{22}}$ is invertible  and its inverse  is $T_{\Upsilon_{22}^{-1}}$.
\end{corollary}

\begin{proof}[\bf Proof]
From Theorem \ref{thm:main1} we know that $Z$ is exponentially stable , But for  a finite dimensional space pointwise  stable is equivalent to exponentially stable . Furthermore, since  $\sZ$ is finite dimensional, formulas  \eqref{defUp11} -- \eqref{defUp22} imply that the functions $\Upsilon_{ij}(\lambda)$, $i,j=1,2$, are rational. Similarly, \eqref{Fcentral} and \eqref{invu22} show that  $F_\circ$   and  $\Upsilon_{22}(\l)^{-1}$ are  rational operator  functions. Recall (see Proposition \ref{prop:stable}) that  $\Upsilon_{22}(\l)$ and $\Upsilon_{22}(\l)^{-1}$ are both analytic at each point of the closed unit disc, which implies that the factorization in \eqref{specfact} is a right canonical factorization and  $\Upsilon_{22}$  is invertible outer.
\end{proof}

%%%%%%%%%%%%%%%%%%%%%%%%%
%%%%%%%%%%%%%%%%%%%%%%%%%%

%:Sec7
\setcounter{equation}{0}

\section{Maximal entropy principle}\label{sec:max}

For a function $F\in \sS(\sU, \sY)$ we define the {\em entropy} to be the  cost function $\sigma_{F}$  defined by   the following optimization problem:
\begin{equation}\label{optme}
\sigma_{F}(u)
 = \inf\left\{\|u - E_\sU^*  T_F^* h\|^2 +
 \lg\left(I - T_F T_F^*\right)h,h\rg
\mid h\in \ell_+^2(\mathcal{Y})\right\},
\end{equation}
where $u$ is a vector  in $\sU$.  Note that the above problem  is precisely the optimization problem in \eqref{optc}  with   $C = T_F$.  Due to the equivalence of the optimization problems in \eqref{optc} and \eqref{optclt}, the  entropy  $\s_F$ is also given by
\begin{equation}\label{optmeclt}
\sigma_{F}(u)
 = \inf\left\{\left\|D_{T_F}\left(E_\sU u - S_\sU e\right)\right\|^2
\mid  e\in \ell_+^2(\mathcal{U})\right\}, \quad u\in \sU.
\end{equation}
This is precisely the notion  of  entropy that is used in  the commutant lifting setting   presented in \cite[Section IV.7]{FFGK98}.  Furthermore, if   $\|F\|_\infty = \|T_F\| < 1$,  then by \eqref{eq:cost3} the  entropy for $F$ is determined by
\begin{equation}\label{optme11}
\sigma_F(u) = \left\lg\left(E_\sU^*(I - T_F^*T_F)^{-1}E_\sU\right)^{-1}u,u\right\rg
\end{equation}

In the band method theory on the  maximal entropy principle the operator $E_\sU^*(I - T_F^*T_F)^{-1}E_\sU$ appears as the multiplicative diagonal of  the function  $I-F(\l)^*F(\l)$, $\l\in \BT$, assuming the Fourier coefficients of $F$   are summable in operator norm; see Sections I.3 and II.3 in \cite{GKW91}, and  Section XXXIV.4 in \cite{GGK2}. For further  information on the multiplicative diagonal we refer to Subsection \ref{ssec:multdiag}.

In this section the function  $F$  is assumed to belong to the set of all solutions to a  LTONP interpolation problem. The following theorem is the  maximal entropy principle for this set of $F$'s.

\begin{theorem}\label{thm:maxent}
Assume that the LTONP interpolation problem with  given data set $\{W,\widetilde{W},Z\}$ is solvable, i.e.,  the Pick matrix $\la$ is nonnegative. Let $F_\circ$  in  $\sS(\sU, \sY)$ be the central solution  to this  LTONP interpolation problem.
Then $F_\circ$ is the unique maximal entropy solution, that is,
if $F\in \sS(\sU, \sY)$ is any  other solution to the LTONP interpolation problem, then
\begin{equation}\label{optME}
\sigma_F(u) \leq \sigma_{F_\circ}(u) \qquad \quad (u \in \sU).
\end{equation}
Moreover, we have $\sigma_F(u) = \sigma_{F_\circ}(u)$
for all $u \in \sU$ if and only if $F = F_\circ$, and the
entropy for the central solution is given by
\begin{equation}\label{ent2a}
\sigma_{F_\circ}(u) =
\left\lg P_\sG (u \oplus 0), (u \oplus 0)\right\rg
\qquad \quad (u \in \sU),
\end{equation}
where $\sG$ is the Hilbert space given by the first part of \eqref{defGGprime}. Finally, if   $\Lambda$ is strictly positive, then the entropy
for the central solution is also determined by
\begin{equation}\label{ent2aa}
\sigma_{F_\circ}(u) =
 \left\lg\left(I + \widetilde{B}^* \la^{-1} \widetilde{B} \right)^{-1}u,u\right\rg
\qquad \quad (u \in \sU).
\end{equation}
\end{theorem}

The above theorem is a more detailed version of  Theorem IV.7.1 in \cite{FFGK98} specialized for the LTONP interpolation problem. For related earlier results see \cite{GKW91} \cite{GGK2} and Section XXXV in \cite{GGK2}.

The proof of Theorem \ref{thm:maxent} is new.  It will be  given    after the next result, which characterizes the entropy function $\s_F$  of any  $F\in S(\sU, \sY)$ in terms of an observable co-isometric realization.

\begin{lemma}\label{L:RealEnt}
Let $\si=\{\a,\b,\g,\d\}$ be an observable co-isometric realization of  $F\in S(\sU,\sY)$,  and let $M_\si$ be the associated system matrix. Set $\sM_\si=\im M_\si^*$. Then
\begin{equation}\label{RealEnt}
\s_F(u)=\lg P_{\sM_\si^\perp}\tau_\sU u, \tau_\sU u\rg\quad (u\in\sU).
\end{equation}
Here $\tau_\sU$ is the embedding  operator  of $\sU$ into $\sU\oplus\sX$.\end{lemma}

\begin{proof}[\bf Proof]
Fix  $F\in\sS(\sU,\sY)$, and let $\si=\{\a,\b,\g,\d\}$ be an observable  co-isometric realization of $F$ with system matrix $M_\si$,  and put $\sM= \im M_\si^*$ where $M_\si $ is given by \eqref{Msi}. Since $M_\si$ is a co-isometry, the  range of  $M_\si^*$ is closed.  Thus $\sM$ is a subspace of $\sU\oplus \sX$.  We set
\begin{equation}\label{def:rhoF}
\rho_F(u)=\lg P_{\sM_\si^\perp}\tau_\sU u, \tau_\sU u\rg\quad (u\in\sU).
\end{equation}
We have to prove $\s_F=\rho_F$.  Since all observable  co-isometric realizations of $F$ are unitarily equivalent, see Theorem \ref{thm:co}, the definition of $\rho_F$ is independent of the choice of  the  observable co-isometric realization of $F$. Hence it suffices to show $\s_F=\rho_F$ for a particular choice of $\si$.

Observe  that $F$ is   a solution to the LTONP interpolation problem with data set $\{I_{\ell^2_+(\sY)},T_F,S_\sY\}$. Indeed, with
\[
W=I_{\ell^2_+(\sY)}, \quad \wt{W}=T_F, \quad  Z=S_\sY
\]
the identities \eqref{data1} and \eqref{defNP} are automatically fulfilled. Moreover, in this case  $F$ is the unique solution, and hence $F$ is the central solution associated with the  data set $\{I_{\ell^2_+(\sY)},T_F,S_\sY\}$. But then we can apply Lemma \ref{L:centreal} to obtain a special  observable co-isometric realization of $F$.  To do this let us  denote the  subspaces $\sF$ in \eqref{defFFprime} and $\sG$ in \eqref{defGGprime} associated with our data set $\{I_{\ell^2_+(\sY)},T_F,S_\sY\}$ by $\what{\sF}$ and $\what{\sG}$, respectively. In this case the associate Pick operator $\wh{\la}$ is given by $\what{\la}=I-T_F T_F^*=D_{T_F^*}^2$. Note that $\what{\sF}$ is given by
\begin{equation}\label{Fhat}
\what{\sF}=\overline{\im\mat{c}{E_\sU^* T_F^*\\ D_{T_F^*}}}.
\end{equation}
Now let $\wh{\si}$ be the observable co-isometric realization obtained by applying Lemma \ref{L:centreal}. Then \eqref{sys-max2} tells us that
$
(M_{\wh{\si}}^*)^\perp=\kr M_{\wh{\si}}= \wh{\sG}.
$
Thus $\rho_F(u)=\lg P_{\wh{\sG}}\t_u, \t_u \rg$. Using \eqref{Fhat} and the projection theorem we then obtain  for each $u\in\sU$ that
\begin{align*}
\rho_F(u)
&=\lg P_{\what{\sG}}\tau_\sU u, \tau_\sU u\rg
=\inf\left\{\|\tau_\sU u -f\| \mid  f\in\what{\sF} \right\}\\
&=\inf\left\{\left\|\mat{c}{u\\0} -\mat{c}{E_\sU^* T_F^*\\ D_{T_F^*}}h\right\| \mid h\in\ell^2_+(\sY) \right\}\\
&=\inf\left\{\|u - E_\sU^*T_F^* h\|^2 + \lg\left(I - T_F T_F^*\right)h,h\rg
\mid h\in \ell_+^2(\mathcal{Y})\right\}=\s_F(u).
\end{align*}
Thus we proved $\s_F=\rho_F$ for a particular choice of $\si$, which  completes the proof.
\end{proof}

\begin{remark} Note that the formula for $\s_F$ in \eqref{def:rhoF} can be rewritten directly in terms of the system matrix $M_\si$ as
\[
\s_F(u)=\inf  \left\{\left\|  \begin{bmatrix}u\\0 \end{bmatrix} - M_\si^* h\right\|^2\mid
 h \in \mathcal{Y} \oplus \mathcal{X} \right\} \qquad  (u\in \sU).
\]
\end{remark}

\begin{proof}[\bf Proof of Theorem \ref{thm:maxent}]
We shall prove Theorem \ref{thm:maxent} using the formula for $\s_F$ given in Lemma \ref{L:RealEnt}.

First we derive the formula \eqref{optme11} for the central solution. From the proof of Lemma \ref{L:RealEnt}, using Lemma \ref{L:centreal}, we know that
\[
\s_{F_\circ}(u)= \lg P_\sG \t_\sU u, \t_\sU u\rg=
\lg P_\sG\begin{bmatrix}u\\0\end{bmatrix}, \begin{bmatrix}u\\0\end{bmatrix}\rg, \quad u\in \sU,
\]
which yields  \eqref{optme11}.

Let $F\in\sS(\sU,\sY)$ be a solution to the LTONP interpolation problem with data $\{W,\wtil{W},Z\}$, and let $\si=\{\a,\b,\g,\d\}$ be a $\la$-preferable, observable, co-isometric realization of $F$. Then $\s_F$ is given by \eqref{RealEnt} with $\sM_\si^\perp=\kr M_\si$, the null space  of the system matrix $M_\si$.
The fact that $\si$ is $\la$-preferable implies that
$M_\si^* |\mathcal{F}^\prime = \omega^*$.
Hence $\sF=\im \o^*\subset \im M_\si^*$, so that
$\sM^\perp\subset \sF^\perp=\sG\oplus\sV$ with $\sV=\sX\ominus \sZ_\circ$.
Hence $P_{\sM^\perp}\leq P_{\sG\oplus\sV}$.
Since $\sU\perp \sV$, both seen as subspaces of $\sU\oplus\sX$, we have
\begin{align*}
\s_F(u)
&=\left\lg P_{\sM_\si^\perp}\tau_\sU u, \tau_\sU u\right\rg
\leq \left\lg P_{\sG\oplus\sV}\tau_\sU u ,\tau_\sU u\right\rg=\\
&\hspace{2cm}=\left\lg P_{\sG}\tau_\sU u, \tau_\sU u\right\rg=\s_{F_\circ}(u)\quad (u\in\sU).
\end{align*}
Hence the entropy $\s_{F_\circ}(u)$  of   the central solution $F_\circ$ is maximal among all solutions to the LTONP interpolation problem for the data set $\{W,\wtil{W},Z\}$.

Next we show that $F_\circ$ is the unique solution to the LTONP interpolation problem for the data set $\{W,\wtil{W},Z\}$ that maximizes the entropy. Hence, assume that the entropy of the solution $F$ is maximal, that is, $\s_F(u)=\lg P_\mathcal{G}  \t_\sU u,\t_\sU u  \rg$ for each $u\in\sU$. Then
\begin{align*}
\|P_{\sM_\si^\perp} \tau_\sU u\|^2&=\lg P_{\sM_\si^\perp} \tau_\sU u,\tau_\sU u\rg = \s_F(u)=\\
&\hspace{1cm}=\lg P_\mathcal{G} \tau_\sU u,\tau_\sU u\rg=
\|P_\mathcal{G} \tau_\sU u\|^2\quad (u\in\sU).
\end{align*}
We will first show that $\kr M_\si=\sM_\si^\perp=\sG$.
Observe that for $u$ in $\mathcal{U}$ we have
\[
\|P_\mathcal{F} \tau_\sU u\|^2 = \|u\| - \|P_\mathcal{G} \tau_\sU u\|^2 =
 \|u\| - \|P_{\sM_\si^\perp} \tau_\sU u\|^2 = \|P_{\sM_\si} \tau_\sU u\|^2.
\]
Because $M_\si | \mathcal{F} = \omega$, it follows that $\mathcal{F}$ is a subspace of $\im M_\si^* =\sM_\si$.
This yields
\[
\|P_{\mathcal{L}^\perp} \tau_\sU u\|^2=
\|P_\mathcal{F} \tau_\sU u\|^2+\|P_{\sL^\perp \ominus \sF} \tau_\sU u\|^2.
\]
Thus $P_{\sL^\perp \ominus \sF} \tau_\sU u=0$.
Hence $P_\mathcal{F} \tau_\sU u = P_{\mathcal{L}^\perp} \tau_\sU u$
holds for all $u\in\sU$. Then
\[
P_\mathcal{G} \tau_\sU u = \tau_\sU u - P_\mathcal{F} \tau_\sU u=
 \tau_\sU u - P_{\sM_\si}  \tau_\sU u
= P_{\sM_\si^\perp} \tau_\sU u \qquad (u\in\sU).
\]
In what follows the symbol  $\sH\bigvee\sK$ stands  for closed linear hull of the spaces $\sH$ and $\sK$.  By consulting \eqref{defFFprime} and noting that $\sZ_\circ$ is the closure of ${\im \la^\half}$, we see that
\[
\mathcal{U} \bigvee \mathcal{F}=\mat{c}{\sU\\ 0}\bigvee \mat{c}{\wtil{B}^*\\\la^\half}\sZ =\sU\oplus\sZ_\circ.
\]
Hence $\mathcal{F} \oplus \mathcal{G} = \mathcal{U} \oplus \mathcal{Z}_\circ
 = \mathcal{U} \bigvee \mathcal{F}$ and we obtain that
\[
\mathcal{G} = P_\mathcal{G}(\mathcal{F}\oplus  \mathcal{G})
= \overline{P_\mathcal{G} \mathcal{U}} = \overline{P_{\sM_\si^\perp} \mathcal{U}}
\subset \sM_\si^\perp.
\]
Therefore $\mathcal{G}$ is a subset of $\sM_\si^\perp$.
Set $\sV=\sX\ominus \sZ_\circ$, with $\sX$ being
the state space of $\si$. Write $\sM_\si^\perp=\sG\oplus\sL$.
Since $\sF\perp \sM_\si^\perp$, we have
\[
\sL\subset (\sU\oplus\sX)\ominus (\sF\oplus\sG)=
(\sU\oplus\sX)\ominus (\sU\oplus\sZ_\circ)=\sV.\]
Because  $\sG\subset \sM_\si^\perp=\kr M_\si$, we have
\[
M_\si|\left(\sU\oplus\sZ_\circ\right)
=M_{\si}|\left( \sF\oplus\sG\right) = \omega P_{\sF}.
\]
Therefore, $M_\si$ has a block operator decomposition of the form
\[
M_\si=\mat{cc}{\d & \g\\ \b & \a}=\mat{c|cc}{\d_\circ &\g_\circ& M_1\\ \hline \b_\circ & \a_\circ& M_2\\ 0&0&M_3} :\mat{c}{\sU\\ \sZ_\circ\\ \sV}\to\mat{c}{\sY\\ \sZ_\circ\\ \sV}
\]
where $\{\alpha_\circ, \beta_\circ, \gamma_\circ, \delta_\circ\}$ form
the system   matrix for $\omega P_\sF$; see \eqref{sys-max}.
Let $x\in\sL\subset\sV$. We have $M_\si x=0$, and thus,
$M_{j}x=0$ for $j=1,2,3$. But then $\a x=0$ and $\g x=0$.
Hence $\g \a^k x=0$ for each $k$. The fact that $\si$
is an observable realization then implies that $x=0$.
Thus $\sL=\{0\}$ and we obtain that $\kr M_\si=\sM_\si^\perp=\sG$.

Using the fact that $M_\si^*$ is an isometry with $M_\si^*| \mathcal{F}^\prime = \omega^*$ and
$\mathcal{G} = \sM_\si^\perp=\kr M_\si$, we see that $M_\si^*$ admits
a matrix decomposition of the form
\[
M_\si^* = \begin{bmatrix}
        \omega^*P_{\mathcal{F}^\prime}                       & 0 \\
         P_{\mathcal{G}^\prime}   & U_+ \\
      \end{bmatrix}: \begin{bmatrix}
     \mathcal{Y} \oplus \mathcal{Z}_\circ   \\
        \mathcal{V} \\
      \end{bmatrix} \rightarrow
      \begin{bmatrix}
     \mathcal{U} \oplus \mathcal{Z}_\circ   \\
        \mathcal{V} \\
      \end{bmatrix}.
\]
Because $M_\si^*|\left(\mathcal{Y} \oplus \mathcal{Z}_\circ\right)$ is an isometry,
without loss of generality we can assume that the lower
left hand corner of $M_\si^*$ is given by
$P_{\mathcal{G}^\prime}$.
Moreover, $U_+$ is an isometry on $\mathcal{V}$.
Since $M_\si^*$ is an isometry and $\sG = \kr M_\si$,
we have
\begin{equation}\label{wander}
\mathcal{V}  =
\mathcal{G}^\prime
\oplus \im(U_+).
\end{equation}
In particular, $\mathcal{G}^\prime$ is a wandering subspace for
the isometry $U_+$ and we have
$\oplus_{n=0}^\infty  U_+^n \mathcal{G}^\prime \subset \mathcal{V}$.
Because the systems matrix $M_\si$ is observable,
$\mathcal{Z}_\circ \oplus \mathcal{V} =
\bigvee_{n=0}^\infty  \alpha^{*n} \gamma^* \mathcal{Y}$.  Observe that $\alpha^*$ admits a lower triangular matrix decomposition of the form:
\[
\alpha^* = \begin{bmatrix}
        \star                                                & 0 \\
        P_{\mathcal{G}^\prime} & U_+ \\
      \end{bmatrix} \mbox{ on }
      \begin{bmatrix}
        \mathcal{Z}_\circ \\
         \mathcal{V} \\
      \end{bmatrix}.
\]
Furthermore,   $\gamma^* \mathcal{Y}$ is a subset  of
$\mathcal{Z}_\circ \oplus \mathcal{G}^\prime$.
For $y$ in $\mathcal{Y}$, we have
\[
\alpha^{*n} \gamma^* y = \begin{bmatrix}
                           \star  \\
                            \sum_{k=0}^{n-2} U_+^k P_{\mathcal{G}^\prime} \star +
                            U_+^{n-1} P_{\mathcal{G}^\prime} \gamma^* y \\
                         \end{bmatrix}.
\]
The observability condition implies that
$\mathcal{V} = \oplus_{n=0}^\infty U_+^n  \mathcal{G}^\prime$. Therefore $U_+$ can be
viewed as the unilateral shift $S_{\mathcal{G}^\prime}$.
In other words, the realization $\si$ of $F$ is unitarily equivalent to the realization of the central solution obtained in Lemma \ref{L:centreal}. Hence $F=F_\circ$. So the maximal solution is unique.

To conclude the proof it remains to show that \eqref{ent2aa} holds. Assume that $\Lambda$ is strictly positive.
Recall that the operator $\tau_1$ in  \eqref{tau12}
is an isometry from $\sU$ into $\sU\oplus \sZ$ whose range equals
$\sG$. Hence $\tau_1 \tau_1^* = P_\sG$ is the orthogonal
projection onto $\sG$. In other words,
\begin{equation}\label{P_F}
P_\mathcal{G} = \tau_1 \tau_1^* = \begin{bmatrix}
                  I\\
                  - \Lambda^{-\frac{1}{2}} \widetilde{B} \\
                \end{bmatrix} R_\circ^2 \begin{bmatrix}
                 I &
                - \widetilde{B} \Lambda^{-\frac{1}{2}} \\
                \end{bmatrix},
\hspace{.1cm}  \mbox{where}\hspace{.15cm} R_\circ^2 = (I+\widetilde{B}^* \Lambda\widetilde{B})^{-1}.
\end{equation}
So for $u$ in $\sU$, we have
\begin{align*}
\s_{F_\circ}(u) &= \lg P_\sG \t_\sU u,  \t_\sU u\rg =
\lg\tau_1 \tau_1^* (u \oplus 0), (u \oplus 0)\rg =
\lg R_\circ^2 u,u\rg.
\end{align*}
In other words, \eqref{ent2aa} holds.
\end{proof}

\begin{remark}\label{centralcostf}
\textup{Consider the LTONP interpolation problem with data
$\{W,\widetilde{W},Z\}$. Moreover,  assume that $\Lambda$ is strictly positive and $Z$ is exponentially stable . Let $F_\circ$ be the central solution.
Then, by Proposition \ref{prop:stable}, the operator $T_{F_\circ}$ is a strict contraction, and thus \eqref{optme11} holds with  $T_{F_\circ}$  in place of $T_F$. Using   \eqref{specfact} in Proposition \ref{prop:stable} we see that
\begin{align*}
\s_{F_\circ}(u)&=\left\lg (E_\sU^*T_{\Upsilon_{22}^{-1}}^* T_{\Upsilon_{22}^{-1}}E_\sU^{-1})^{-1}u, u \right\rg\\
&=\left\lg(\Upsilon_{22}^{-1}(0)^*\Upsilon_{22}^{-1}(0)u,u \right\rg=\|\Upsilon_{22}^{-1}(0)u \|^2, \quad u\in \BD
\end{align*}
On the other hand, according to \eqref{ent2aa}, we have
\[
\s_{F_\circ}(u)= \left\lg(I_{\sU} + \widetilde{B}^*\Lambda^{-1}\widetilde{B})^{-1}u,u\right\rg, \quad u\in \BD.
\]
Hence
\begin{equation*}
s_{F_\circ}(u) = \|\Upsilon_{22}(0)^{-1}u\|^2 = \left\lg(I_{\sU} + \widetilde{B}^*\Lambda^{-1}\widetilde{B})^{-1}u,u\right\rg, \quad u\in \BD.\end{equation*}
If  $\sU$ is finite dimensional, then the later identity can be rewritten as
\begin{equation*}
 \det [(I_{\sU} + \widetilde{B}^*\Lambda^{-1}\widetilde{B})^{-1}] =
\exp{\left({\frac{1}{2\pi}\int_0^{2\pi} \ln \det[I-F_\circ(e^{i\theta})^*F_\circ(e^{i\theta})] d \theta}\right)}.
\end{equation*}
For more details, in particular concerning the connections with spectral factorization,  we refer to Subsection \ref{ssec:multdiag}.}
\end{remark}

%%%%%%%%%%%%%%%%%%%%%%%%%
%%%%%%%%%%%%%%%%%%%%%%%%%

%:Sec8
\setcounter{equation}{0}
\section{Commutant lifting  as  LTONP interpolation}\label{CL-ONP}
In the second paragraph after Proposition  \ref{prop:propsUps} we have seen that in the strictly positive case the LTONP interpolation problem is a commutant lifting problem.  In this section we go in the reverse direction.  We consider a large  subclass of  commutant lifting problems, and we show that  this class of  problems is  equivalent to the class of  LTONP interpolation problems. This equivalence will allow us to reformulate Theorem \ref{thm:main1a} as a theorem describing all solutions of a suboptimal commutant lifting problem (see Theorem \ref{thm:main2} below).

 Our starting point is  the quadruple $\{A_\circ, S_\sU, T^\prime, S_\sY\}$ as the given commutant  lifting data set.  Here   $A_\circ$ is an operator mapping  $\ell_+^2(\sU)$ into $\sH^\prime$, where  $\sH^\prime$ is an invariant subspace for $S_\mathcal{Y}^*$. In particular, $\sH^\prime$ is a subspace of $\ell_+^2(\sY)$, and $ \ell_+^2(\sY)\ominus \sH^\prime $ is invariant under  $S_\sY$.   Furthermore,    $T^\prime$ is the compression of $S_\sY$ to $\sH^\prime$, that is,  $T^\prime=\Pi_{\sH^\prime}S_\sY \Pi_{\sH^\prime}^*$, where $\Pi_{\sH^\prime}$ is the orthogonal projection of $\ell_+^2(\mathcal{Y})$ onto $\sH^\prime$. The data set satisfies  the intertwining relation $A_\circ S_\sU=T^\prime A_\circ$. Note that  we do not assume the minimality condition $\bigvee_{n\geq 0}S_\sY^n \sH^\prime=\ell_+^2(\sY)$, which often plays a simplifying role in proofs.

Given the lifting data set $\{A_\circ, S_\sU, T^\prime, S_\sY\}$, the commutant  lifting problem is to find all $F\in \sS(\sU, \sY)$ such that
\[
T_F= \begin{bmatrix}A_\circ\\ \star \end{bmatrix}: \ell_+^2(\sU)\to \begin{bmatrix}  \sH^\prime\\ \ell_+^2(\sY)\ominus \sH^\prime \end{bmatrix}.
\]
If the problem is solvable, then necessarily $A_\circ$ is a contraction.

 To reformulate this commutant lifting problem as a LTONP interpolation problem, put
\begin{equation}\label{CL-ONP1}
\sZ=\sH^\prime, \quad Z = T^\prime, \quad
W = \Pi_\sZ: \ell_+^2( \mathcal{Y})\to\sZ,
\quad
\wt{W}  = A_\circ: \ell_+^2( \mathcal{U})\to \sZ.
\end{equation}
Here  $\Pi_\sZ$ is the orthogonal projection of $\ell_+^2(\mathcal{Y})$ onto    $\sZ=\sH^\prime$.   With $W$, $\wt{W}$ and $Z$  given by \eqref{CL-ONP1}  it is straightforward to check that $ZW=WS_\sY$  and $Z\wt{W}=\wt{W}S_\sU$. Thus the conditions in \eqref{data1} are satisfied. Moreover, the solutions to the LTONP interpolation problem with this data set with data $\{W,W,Z\}$ are precisely the solutions to the commutant lifting problem with data set with data $\{W,W,Z\}$; see the second paragraph after Proposition \ref{prop:propsUps}.
Since $S_\sY^*$ is pointwise stable, it is also clear that $Z^*$ is pointwise stable.  Note that in this case
\begin{align}
&P =\Pi_\sZ \Pi_\sZ^*= I_\sZ, \quad \wt{P}= A_\circ A_\circ^*  \ands \la = P - \wt{P} = I - A_\circ A_\circ^* ,\label{CL-ONP2a}\\
&\hspace{2.6cm}B = \Pi_\sZ E_\sY  \ands \wt{B} = A_\circ E_\sU.\label{CL-ONP2b}
\end{align}
So the commutant lifting problem with data $\{A_\circ, S_\sU, T^\prime, S_\sY\}$ is solvable if and only if $\Lambda$ is positive, or equivalently, $A_\circ$ is a contraction. Finally, it is noted that one can use Theorem \ref{thm:allsol1}
to find all solutions to this commutant lifting problem when $\|A_\circ\| \leq 1$.

Notice that $\kr W=\ell_+^2(\sY)\ominus \sH^\prime$. By  the Beurling-Lax-Halmos theorem  there exists an inner function $\tht\in \sS(\sE, \sY)$ such that  $\ell_+^2(\sY)\ominus \sH^\prime =\kr W=\im T_\tht$, which allows us to define:
\begin{equation}
\label{defCD2}
C=E_\sE^* T_\tht^* S_\sY  \Pi_\sZ^*  : \sZ\to \sE \ands  D=\tht(0)^*:\sY\to \sE.
\end{equation}
Note $C$ and $D$ defined above  are precisely equal to the operators $C$ and $D$  defined by \eqref{defCD} provided the data set $\{W,\wt{W},  Z \}$ is the one defined by the commutant lifting setting  \eqref{CL-ONP1}. It follows that the operators $C$ and $D$ in \eqref{defCD2} is an admissible pair of complementary operators determined  by  the data set  $\{W,\wt{W}, Z\}$  defined by   \eqref{CL-ONP1}.

Using the above connections we can apply Theorem \ref{thm:main1a} to obtain the following theorem which describes all  solutions of the  commutant lifting problem with data $\{A_\circ, S_\sU, T^\prime,  S_\sY\}$ for the case when the operator $A_\circ$ is a strict contraction. Note that in this case the operator $A$ defined by \eqref{defA} is equal to the operator
\begin{equation}\label{defA-CL}
A=\Pi_{\sH^\prime}^* A_\circ=\Pi_{\sZ}^*A_\circ: \ell_+^2(\sU)\to \ell_+^2(\sY).
\end{equation}
Hence using  $\Pi_{\sZ}\Pi_{\sZ}^* =I_\sZ$, we also have $\Pi_{\sZ} A = A_\circ$.

\begin{theorem}\label{thm:main2} Let $\{A_\circ, S_\sU, T^\prime, S_\sU\}$  be a commutant lifting data set. Assume $A_\circ$ is a strict contraction.  Then all solutions $F$ to  the commutant lifting problem for the  data set   $\{A_\circ, S_\sU, T^\prime, S_\sU\}$  are given by
\begin{equation}\label{allsol2a}
F(\l)  = \big(\Upsilon_{11}(\l)X(\l)+\Upsilon_{12}(\l) \big)\big(\Upsilon_{21}(\l)   X(\l)+\Upsilon_{22}(\l)   \big)^{-1}, \quad \l \in \BD,
\end{equation}
where the free parameter $X$ is an arbitrary Schur class function,  $X\in \sS(\sU, \sE)$, and the coefficients $\Upsilon_{i,j}$, $1\leq i, j\leq 2$, are the analytic functions on $\BD$ defined by
\begin{align}
\Upsilon_{11}(\l)  &= D^*Q_\circ + \lambda  E_{\sY}^*
(I - \l S_{\sY}^*)^{-1} \Pi_{\sH^\prime}^*(I - A_\circ A_\circ^*)^{-1}  C^* Q_\circ,\label{Upsilon11b}\\
\Upsilon_{12}(\l)   &=E_{\sY}^*\big(I - \l  S_{\sY}^*\big)^{-1} \Pi_{\sH^\prime}^*A _0^* (I - A_\circ^*A_\circ)^{-1}E_{\sU}R_\circ,\label{Upsilon12b}\\
\Upsilon_{21}(\l)    &=   \lambda E_{\sU}^*
(I  -\l  S_{\sU}^*)^{-1} A_\circ^* (I - A_\circ A_\circ^*)^{-1}  C^*Q_\circ, \label{Upsilon21b}\\
\Upsilon_{22}(\l)  &=
E_{\sU}^*\big(I - \l  S_{\sU}^*\big)^{-1} (I - A_\circ^*A_\circ)^{-1}E_{\sU}R_\circ. \label{Upsilon22b}
\end{align}
Here  $C$ and $D$ are the operators  defined by \eqref{defCD2}, and
\begin{align}
Q_\circ  &=   \Big(I_\sE+CA_\circ \big(I-A_\circ^*  A_\circ\big)^{-1}A_\circ^* C^* \Big)^{-\frac{1}{2}}, \label{CLTQ0}\\
R_\circ & = \Big(E_{\sU}^*\big(I - A_\circ^*A_\circ\big)^{-1}E_{\sU}\Big)^{-\frac{1}{2}},\label{CLTR0}
\end{align}
and these operators are strictly positive.
\end{theorem}

\begin{proof}[\bf Proof]
The above theorem is a direct corollary of  Theorems \ref{thm:main1}  and \ref{thm:main1a}.   Indeed,  in the present setting  $A = \Pi_\sZ^* A_\circ$  and $\Pi_\sZ A =  A_\circ$ while  the operator $W^*=\Pi_\sZ^*=\Pi_{\sH^\prime}^*$. This   implies that
\[
(I-A^*A)^{-1}= (I-A_\circ^*A_\circ)^{-1}, \quad WA(I-A^*A)^{-1}A^* W^*
=A_\circ(I-A_\circ^*A_\circ)^{-1}A_\circ^*.
\]
It follows that in this case the operators $Q_\circ$ and $R_\circ$ in Theorem \ref{thm:main1a} are given by \eqref{CLTQ0} and \eqref{CLTR0},  respectively.   Furthermore,
\begin{align*}
&A(I-A^*A)^{-1}=\Pi_{\sH^\prime}^* A_\circ (I-A_\circ^*A_\circ)^{-1}, \\[.1cm]
&(I-AA^*)^{-1}W^*=(I-AA^*)^{-1} \Pi_{\sH^\prime}^* = \Pi_{\sH^\prime}^*  (I-A_\circ A_\circ^*)^{-1}, \\[.1cm]
&A^*(I-AA^*)^{-1}W^*= A_\circ^*\Pi_{\sH^\prime}\Pi_{\sH^\prime}^*
 (I-A_\circ A_\circ^*)^{-1}=A_\circ^*  (I-A_\circ A_\circ^*)^{-1}.
\end{align*}
The latter  identities  show that in this case the formulas for the function $\Upsilon_{ij}$, $1\leq i, j \leq 2$, in Theorem \ref{thm:main1a} can be rewritten as in \eqref{Upsilon11b} -- \eqref{Upsilon22b}, which completes the proof.
\end{proof}

%%%%%%%%%%%%%%%%%%%%%%%%%%%%%%%%%%%
%%%%%%%%%%%%%%%%%%%%%%%%%%%%%%%%%%%

%:Sec9
 \setcounter{equation}{0}
\section{The Leech problem revisited}\label{sec:Leech}

In this section we discuss the Leech problem and show how it appears as a special case of our LTONP interpolation  problem. We will also show that our first main result, Theorem \ref{thm:main1}, after some minor computations, provides the `infinite dimensional state space' characterization of the solutions to the Leech problem given in Theorem 3.1 in \cite{FtHK15}, without any `minimality' condition.
It is noted that  in \cite{FtHK15} these formulas are used  to derive
algorithms in the rational case.  The paper by R.B. Leech \cite{L14} where this problem originated from was eventually published in 2014; see \cite{KR14} for some background on the history of this paper.

The data set for the Leech problem consists of two functions $G\in H^\infty(\sY,\sV)$ and $K\in H^\infty(\sU,\sV)$, for Hilbert spaces $\sU$, $\sY$ and $\sV$, and the aim is to find Schur class functions $F\in \sS(\sU,\sY)$ such that $GF=K$. In terms of Toeplitz operators, we seek $F\in \sS(\sU,\sY)$ such that $T_GT_F=T_K$. To convert the  Leech problem to a LTONP interpolation problem,  set $\sZ=\ell^2_+(\sV)$ and define
\begin{equation}\label{WwtWZ-Leech}
W=T_G:\ell^2_+(\sY)\to\sZ,\quad \wt{W}=T_K:\ell^2(\sU)\to\sZ,\quad Z=S_\sV:\sZ\to\sZ.
\end{equation}
In this setting,
\begin{equation}\label{LeechBB}
P = T_G T_G^* \ands \widetilde{P} = T_K T_K^*.
\end{equation}
Since $T_G$ and $T_K$ are analytic Toeplitz operators they intertwine the unilateral forward shifts on the appropriate $\ell^2_+$-spaces. This shows that the triple $\{W,\wt{W},Z\}$ satisfies the conditions of being a LTONP
data set; see \eqref{data1}.  Moreover, the solutions to the LTONP interpolation problem associated with the data set $\{W,\wt{W},Z\}$  coincide with the solutions to the Leech problem for the functions $G$ and $K$. Furthermore, note that $Z^*=S_\sV^*$ is pointwise stable, but does not  have spectral radius less than one, as required in Section 1.4 of \cite{FFGK98}. The solution criterion $WW^*-\wt{W}\wt{W}^*\geq 0$ from the LTONP interpolation problem translates to the known solution criterion for the Leech problem, namely $T_GT_G^*-T_KT_K^*\geq0$.

Note that in this setting $B = T_G E_\sY$ and $\widetilde{B} = T_K E_\sU$. On can use Theorem \ref{thm:allsol1} to find a parametrization  of all solutions to the Leech problem when $\Lambda =T_GT_G^*-T_KT_K^*\geq0$.   From Theorem \ref{thm:main1}, we now obtain the following characterization of the solutions to the Leech problem under the condition that $\Lambda =T_GT_G^*-T_KT_K^*$  is strictly positive.

\begin{theorem}\label{thm:Leech}
Let $G\in H^\infty(\sY,\sV)$ and $K\in H^\infty(\sU,\sV)$,  and assume that $T_GT_G^*-T_KT_K^*$ is strictly positive. Let  $\Theta\in \sS(\sE,\sY)$,  for some Hilbert space $\sE$, be the inner function such that $\im T_\Theta=\kr T_G$. Then the solutions $F$ to the Leech problem associated with $G$ and $K$ are given by
\begin{equation}\label{solsLeech}
F(\lambda)= \Big(\Upsilon_{11}(\l)  X(\l)+\Upsilon_{12}(\l)\Big)\Big(\Upsilon_{21}(\l) X(\l) +\Upsilon_{22}(\l)\Big)^{-1},
\end{equation}
where the free parameter $X$ is an arbitrary Schur class function,  $X\in \sS(\sU, \sE)$, and the coefficients in \eqref{solsLeech} are the analytic functions on $\BD$  given by
\begin{align*}
\Upsilon_{11}(\l)
&=
\Theta(0)^*Q_\circ
  -\l E_\sY^*(I -\l S_\sY^*)^{-1}T_G^*(T_GT_G^*-T_KT_K^*)^{-1}N
Q_\circ, \\
\Upsilon_{12}(\l)
&=E_\sY^*(I -\l S_\sY^*)^{-1}T_G^*(T_GT_G^*-T_KT_K^*)^{-1}T_KE_\sU R_\circ,
\\
\Upsilon_{21}(\l)
&=-\l E_\sU^*(I -\l S_\sU^*)^{-1}T_K^*(T_GT_G^*-T_KT_K^*)^{-1}NQ_\circ, \\
\Upsilon_{22}(\l)&=R_\circ  +
E_\sU^*(I -\l S_\sU^*)^{-1}T_K^*(T_GT_G^*-T_KT_K^*)^{-1}T_K E_\sU R_\circ.
\end{align*}
Here $N=-T_GS_\sY^* T_\Theta E_\sE=S_\sV^* T_G E_{\sY}\Theta(0)$ and $Q_\circ$ and  $R_\circ$ are the strictly positive operators given by
\begin{align*}
Q_\circ&=\left(I_\sE+N^*((T_GT_G^*-T_KT_K^*)^{-1}-(T_GT_G^*)^{-1})N\right)^{-\frac{1}{2}}: \sE\to \sE,\\
R_\circ&=(I_\sU+E_\sU^*T_K^*(T_GT_G^*-T_KT_K^*)^{-1}T_K E_\sU)^{-\frac{1}{2}}: \sU\to \sU.
\end{align*}
Moreover, the parametrization given by \eqref{allsol1a}  is proper, that is, the map $X\mapsto F$ is one-to-one.
\end{theorem}

\begin{proof}[\bf Proof]
The formulas for  $\Upsilon_{i,j}$, $1\leq i,j\leq 2$,   follow directly from those in \eqref{defUp11}--\eqref{defUp22} after translation to the current setting, that is, using $\Lambda = T_G T_G^* - T_K T_K^*$ with $P = T_G T_G^*$ and $B = T_G E_\sY$ and $\widetilde{B} = T_K E_\sU$. Using \eqref{def:CD3}
with $W = T_G$ we arrive at  $PC^*=T_GS_\sY^* T_\Theta E_\sE=-N$. For the second formula for $N$, namely $N=S_\sV^* T_G E_{\sY}\Theta(0)$, see Lemma 2.1 in \cite{FtHK15}.
\end{proof}

This characterization of the solutions to the Leech problem is almost identical to that obtained in Theorem 3.1 in \cite{FtHK15}, for the case $\sU=\BC^p$, $\sY=\BC^p$, $\sV=\BC^m$ and under the `minimality' condition that for no nonzero $x\in\BC^p$ the function $z\mapsto G(z)x$ is identically equal to zero. Note that the operators $Q_\circ$  and $R_\circ$ above coincide with $\Delta_1^{-1}$ and
$\Delta_0^{-1}$  of Theorem 3.1 in \cite{FtHK15},
respectively.  However,  in  the definition of $\Delta_1$ in \cite[eqn.\ (3.7)]{FtHK15} it should have been  $((T_GT_G^*-T_KT_K^*)^{-1}-(T_GT_G^*)^{-1})$ rather than $((T_GT_G^*-T_KT_K^*)^{-1}-(T_GT_G^*)^{-1})^{-1}$. To see that $\Upsilon_{12}$ and $\Upsilon_{22}$ in Theorem \ref{thm:Leech} indeed coincide with those in Theorem 3.1 in \cite{FtHK15}, use that $(I -\l S_\sV^*)^{-1}=I+\l (I -\l S_\sV^*)^{-1}S_{\sV}^*$, so that
\begin{align*}
\Upsilon_{12}(\l)
&=E_\sY^*T_G^*(T_GT_G^*-T_KT_K^*)^{-1}T_KE_\sU R_\circ +\\
&\qquad\l E_\sY^*T_G^*(I -\l S_\sV^*)^{-1}S_\sV^*(T_GT_G^*-T_KT_K^*)^{-1}T_KE_\sU R_\circ \\
\Upsilon_{22}(\l)
&=R_\circ  +E_\sU^*T_K^*(T_GT_G^*-T_KT_K^*)^{-1}T_K E_\sU R_\circ\\
&\qquad+\l E_\sU^*T_K^*(I -\l S_\sV^*)^{-1}S_\sV^*(T_GT_G^*-T_KT_K^*)^{-1}T_K E_\sU R_\circ\\
&=R_\circ^{-1} +\l E_\sU^*T_K^*(I -\l S_\sV^*)^{-1}S_\sV^*(T_GT_G^*-T_KT_K^*)^{-1}T_K E_\sU R_\circ,
\end{align*}
where the last identity follows because
\[
(I_\sU+E_\sU^*T_K^*(T_GT_G^*-T_KT_K^*)^{-1}T_K E_\sU) R_\circ= R_{\circ}^{-2}R_\circ=R_\circ^{-1}.
\]

The Toeplitz-corona problem corresponds to the special case of the Leech problem where $\sU=\sV$ and  $K= I_\sU$  is identically equal to the identity operator on $\sU$. In view of the connection made between the LTONP interpolation problem and the commutant lifting problem in Section \ref{CL-ONP}, we refer to Proposition A.5 in \cite{FtHK15}, where the Toeplitz-corona is identified as a special case of the commutant lifting problem discussed in Section \ref{CL-ONP}. Although Proposition A.5 in \cite{FtHK15} is proven only for the case where $\sU$ and $\sY$ are finite dimensional, one easily sees that the result caries over to the infinite dimensional case. We present the result here rephrased in terms of the LTONP interpolation problem, and add a proof for completeness. Note that with $K$ is identically equal to $I_\sU$ we have $\wt{W}=I_{\ell^2_+(\sU)}$. Hence $\wt{W}$ is invertible.
The converse is also true.

\begin{proposition}\label{TC}
Let $\{W,\wt{W},Z\}$ as in \eqref{data1} be a data set for a LTONP interpolation problem where $\wt{W}$ is invertible. Then there exists a function \ $G\in H^\infty(\sY,\sU)$ such that with $K\equiv I_{\sY}$ the operators $W$, $\wt{W}$ and $Z$ are given by \eqref{WwtWZ-Leech}, with $\sV=\sY$, up to multiplication with an invertible operator from $\sZ$ to $\ell^2_+(\sU)$. In fact, $G$ is defined by $T_G=\wt{W}^{-1}W$, or equivalently,  $W = \widetilde{W} T_G$ and
$\widetilde{W} = \widetilde{W} T_I$ and $Z = \widetilde{W}S_\sU \widetilde{W}^{-1}$.
\end{proposition}

\begin{proof}[\bf Proof]
Let $\{W,\wt{W},Z\}$   be a data set for a  LTONP interpolation problem with
$\wt{W}$ invertible. Then $ZW=W S_\sY$ and
$S_\sU \wt{W}^{-1}=\wt{W}^{-1}Z$, so that
\[
S_\sU \wt{W}^{-1} W= \wt{W}^{-1}Z W= \wt{W}^{-1} W S_\sY.
\]
This shows $\wt{W}^{-1} W$ is a Toeplitz operator $T_G$ with defining function  $G\in H^\infty(\sY,\sU)$.  It is also clear that for $K\equiv I_{\sY}$ we have
\[
T_K=I_{\ell^2_+(\sU)}=\wt{W}^{-1}\wt{W}\ands
\wt{W}^{-1}Z\wt{W}=\wt{W}^{-1}\wt{W} S_{\sU}=S_\sU.\qedhere
\]
\end{proof}

%:secApp
\setcounter{equation}{0}
\appendix

\section{}\label{sec:App}
\renewcommand{\theequation}{A.\arabic{equation}}

\setcounter{equation}{0}
 This appendix  consists of  seven  subsections containing  standard background material  that is used throughout the paper.  Often we added  proofs for the sake of completeness.

\subsection{Stein equation}
In this section, we present some standard results
concerning discrete time Stein equations.

\begin{lemma}\label{lem_lyap} Let $Z$ be an operator on $\mathcal{Z}$
such that $Z^*$ is pointwise stable. Let
$\alpha$ be an operator   on $\mathcal{X}$ such that $\sup_{n\geq 0}\|\a^n\|<\iy$ while $\Xi$ is an operator mapping $\mathcal{X}$ into $\mathcal{Z}$.  Assume that the Stein equation
\begin{equation}\label{lyap00}
  \Omega - Z \Omega \alpha =\Xi
\end{equation}
has a solution  $\Omega$  mapping $\mathcal{X}$ into $\mathcal{Z}$.  Then the solution to this Stein equation is unique.
\end{lemma}

\begin{proof}[\bf Proof]
If $\Omega_1$ is another operator satisfying $\Omega_1 - Z \Omega_1 \alpha = \Xi$, then subtracting these two Stein equations yields
 \[
 \Omega - \Omega_1 = Z \big(\Omega - \Omega_1\big) \alpha.
 \]
 Applying this identity recursively,
 we have $\Omega - \Omega_1 = Z^n \big(\Omega - \Omega_1\big) \alpha^n$
 for all integers $n \geq 0$. By taking the adjoint, we obtain
 $\Omega^* - \Omega_1^* = \alpha^{*n} \big(\Omega^* - \Omega_1^*\big) Z^{*n}$. Since $Z^*$ is pointwise stable and $\sup_{n\geq 0}\|\a^n\|<\infty$,  for each $z\in\sZ$  we have
 \[
 \|(\Omega^* - \Omega_1^*) z\|= \|\alpha^{*n} \big(\Omega^* - \Omega_1^*\big) Z^{*n}z\|
 \leq \|\alpha^{*n}\| \|\big(\Omega^* - \Omega_1^*\big)\| \|Z^{*n}z\|\to 0.
 \]
 Hence   $\Omega^* = \Omega_1^*$, or equivalently, $\Omega = \Omega_1$.
 Therefore the solution to the Stein equation
 $\Omega = Z \Omega \alpha + \Xi$ is unique.
\end{proof}

Let $Z$ be an operator on $\sZ$ such that  $Z^*$ is pointwise  stable. Assume that $W$ is an operator mapping $\ell_+^2(\sY)$ into $\sZ$ such that $Z W = W S_\sY$.  Let $B$ be the operator mapping $\sY$ into $\sZ$ defined by $B = W E_\sY$. Then  $P = W W^*$ is the unique solution to the Stein equation
 \begin{equation}\label{lyap_app}
 P = Z P Z^* + BB^*.
\end{equation}
Lemma \ref{lem_lyap} guarantees that the solution to this Stein
equation is unique. Moreover, using $Z W = W S_\sY$, we obtain
\[
P = W W^* = W\left(S_\sY S_\sY^* + E_\sY E_\sY^*\right)W^*
= Z W W^* Z^* + B B^* = Z P Z^* + B B^*.
\]
Hence $P = WW^*$ satisfies the Stein equation \eqref{lyap_app}.
Notice that
\[
\begin{bmatrix}
  E_\sY  & S_\sY E_\sY  & S_\sY^2 E_\sY & S_\sY^3 E_\sY & \cdots\\
\end{bmatrix} = I,
\]
the identity {operator} on $\ell_+^2(\sY)$. Using this with $Z W = W S_\sY$, we see that
\[
W = W \begin{bmatrix}
  E_\sY  & S_\sY E_\sY  & S_\sY^2 E_\sY &  \cdots\\
\end{bmatrix} =
\begin{bmatrix}
  B  & Z B & Z^2 B & \cdots\\
\end{bmatrix}.
\]
In particular, $P = W W^* = \sum_{{n=0}}^\iy Z^n B B^* Z^{*n}$.
Motivated by this analysis we present the following result.

\begin{lemma}\label{lem:lyapP}
Let $Z$ be an operator on $\sZ$ such that $Z^*$ is pointwise stable.
Let $B$ be an operator mapping $\sY$ into $\sZ$. If $P$ is a solution
to the Stein equation $P = Z P Z^* + B B^*$, then $P$ is the only
solution to this Stein equation. Moreover, $P = W W^*$
where $W$ is the operator mapping $\ell_+^2(\sY)$ into $\sZ$
given by
\begin{equation}\label{cont_opt}
 W = \begin{bmatrix}
  B  & Z B & Z^2 B & \cdots\\
\end{bmatrix}: \ell_+^2(\sY) \rightarrow \sZ.
\end{equation}
Finally, $Z W = W S_\sY$ and $W E_\sY = B$.
\end{lemma}

\begin{proof}[\bf Proof]
By recursively using $P = Z P Z^* + B B^*$, we obtain
\begin{align*}
P &= BB^*  + Z P Z^* =
BB^*  + Z\left(BB^*  + Z P Z^* \right) Z^*\\
&=  BB^* + Z BB^* Z^* + Z^2\left(BB^*  + Z P Z^* \right) Z^{*2}+\cdots \\
&=  \sum_{j=0}^n Z^jBB^* Z^{*j}  + Z^{n+1} P Z^{* n+1},
\end{align*}
where  $n$ is any positive integer. Because  $Z^*$ is pointwise stable,
the uniform boundedness principle implies that $\sup\{\|Z^n\|:n\geq 0\} <\infty$.
Thus $Z^{n+1} P Z^{* n+1}$ converges to zero pointwise as $n$ tends to infinity. Therefore $P = \sum_{j=0}^\infty  Z^jBB^* Z^{*j}$ with pointwise convergence.  Moreover, $W$ in \eqref{cont_opt} is a well defined bounded operator and $P = W W^*$.
Clearly, $Z W = W S_\sY$ and $B = WE_\sY$.
\end{proof}

\subsection{The Douglas factorization lemma for $K_1 K_1^*=K_2 K_2^*$}\label{Assec:KK12}

In this subsection we review a variant of   the  Douglas factorization lemma; for the full lemma see, e.g., \cite[Lemma XVII.5.2]{GGK2}. The results
presented are used in Sections \ref{sec:ONP-coiso} and \ref{sec:strictlypos}. Consider the two Hilbert space operators and related subspaces given by:
\begin{align}
&K_1:\sH_1\to \sZ \ands \sF=\overline{\im K_1^*}\subset \sH_1,  \label{DK1K2w1}\\
&K_2:\sH_2\to \sZ \ands \sF^\prime=\overline{\im K_2^*}\subset \sH_2.\label{DK1K2w2}
\end{align}
The following two lemmas are  direct corollaries of  the Douglas factorisation lemma.

\begin{lemma}\label{lem:KK1} Let $K_1$ and $K_2$ be two operators of the form \eqref{DK1K2w1} and \eqref{DK1K2w2}. Then the following are equivalent.
\begin{itemize}
  \item[(i)] The operators $K_1 K_1^*=K_2 K_2^*$.
  \item[(ii)] There exists a unitary operator $\omega:\sF\to \sF^\prime$
  such that
  \begin{equation}\label{propsom}
\o K_1^*=K_2^* \quad\mbox{or equivalently} \quad  K_2\o=K_1|\sF.
\end{equation}
  \item[(iii)] There exists an operator $\omega:\sF\to \sF^\prime$ such that
\begin{equation}\label{defom-app1}
K_2 K_2^*=K_2\o K_1^* \ands K_2\o K_1^* = K_1 K_1^*.
\end{equation}
In this case $\omega$ is unitary.
\end{itemize}
If Part (ii) or (iii) holds, then the operator
$\omega$ is uniquely determined. Finally,
each of the identities in \eqref{propsom}
separately can be used as the definition of $\omega$.
\end{lemma}

\begin{remark}
\textup{The operator products in \eqref{propsom} and \eqref{defom-app1} have to be understood pointwise. For instance, the first identity   in \eqref{propsom} just means that $\o K_1^*x=K_2^*x$ for each $x\in \sZ$. Note that for each $x\in \sZ$ we have   $K_1^*x\in \sF$, and thus $\o K_1^*x$ is well defined and belongs to $\sF^\prime$.  On the other hand,  $K_2^*x$ also belongs to $\sF^\prime$, and hence $\o K_1^*x=K_2^*x$ makes sense. This remark  also  implies to the other identities in this subsection.}
\end{remark}

Let us sketch a proof of Lemma \ref{lem:KK1}. One part of  the Douglas factorization lemma says that if $A$ and $B$ are two operator acting between the appropriate spaces, then $AA^* \leq BB^*$ if and only if there exists
a contraction $C$ from the closure of the range of $B^*$ to
the closure of the range of $A^*$ satisfying $A^* = C B^*$.
Moreover, in this case, the operator $C$ is unique.
If $K_1 K_1^* = K_2K_2^*$, then there exists a contraction
$\omega$ such that $K_2^* = \omega K_1^*$. Because
$K_1 K_1^* = K_2K_2^*$, it follows that $\omega$ is an isometry
from $\sF$ onto $\sF^\prime$. Since $\omega$ is onto, $\omega$
is unitary. On the other hand, if $K_2^* = \omega K_1^*$
where $\omega$ is unitary, then $K_1 K_1^* = K_2K_2^*$.
Therefore Parts (i) and (ii) are equivalent.

Clearly, Part (ii) implies that Part (iii) holds. Assume that  Part (iii) holds. Then by the first identity in \eqref{defom-app1} and the fact that $K_2$ is zero on $\sH_2\ominus \sF^\prime$, we see that $\o K_1^*=K_2^*$. Similarly, using the second identity in \eqref{defom-app1} and $\sF=\overline{\im K_1^*}$, we obtain $K_2\o=K_1|\sF$.
This yields Part (ii). Therefore Parts (i) to (iii) are equivalent.

\begin{lemma}\label{lem:KK2}
Let $K_1$ and $K_2$ be two operators of the form \eqref{DK1K2w1} and \eqref{DK1K2w2}. Assume $K_1K_1^*=K_2K_2^*$ and let  $\omega:\sF\to\sF'$  be the unitary map  uniquely determined by \eqref{defom-app1}.  Let $\t_1:\sU_1\to \sH_1$ and  $\t_2:\sU_2\to \sH_2$ be isometries such that
$\im \t_1=\sH_1 \ominus \sF$ and $\im \t_2=\sH_2 \ominus \sF^\prime$.
 Then all contractions $Y:\sH_1 \to \sH_2$  such that
\begin{equation}\label{defom-app2}
K_2 K_2^*=K_2YK_1^*\ands K_2YK_1^*=K_1 K_1^*.
\end{equation}
are given by
$Y=\t_2 X\t_1^*+ \Pi_{\sF^\prime}^*\o \Pi_\sF$ where $X$ is any contraction mapping $\sU_1$ into $\sU_2$.  Moreover,  the map $X \to Y$ is one-to-one.
\end{lemma}

Recall that $V$ is a right inverse of $U$ if $UV = I$. Next we assume that $N:=K_1K_1^*=K_2K_2^*$  is strictly positive. Then both $K_1$ and $K_2$ are right invertible, the operator  $K_1^*N^{-1}$ is a right inverse of  $K_1$
and the operator  $K_2^*N^{-1}$ is a right inverse of  $K_2$. Indeed, we have
\begin{align*}
K_1K_1^*N^{-1}&=K_1K_1^*(K_1K_1^*)^{-1}=I_\sZ,\\
K_2K_2^*N^{-1}&=K_2K_2^*(K_2K_2^*)^{-1}=I_\sZ.
\end{align*}
Furthermore,  a direct computation shows that the orthogonal projections $P_\sF$ onto $\sF$ and $P_{\sF^\prime}$ onto $\sF^\prime$ are respectively given by
\begin{equation}\label{projFFprime}
P_\sF =K_1^*N^{-1}K_1  \ands  P_{\sF^\prime} =K_2^*N^{-1}K_2.
\end{equation}

\begin{lemma}\label{lem:KK3}
Let $K_1$ and $K_2$ be two operators of the form \eqref{DK1K2w1} and \eqref{DK1K2w2}. Assume that  $K_1K_1^*=K_2K_2^*$ and $N=K_1K_1^*=K_2K_2^*$ is strictly positive. Then the unique operator $\o:\sF\to \sF^\prime$ satisfying \eqref{defom-app1} is given by
\begin{equation}\label{eq:om}
\omega P_\sF  = K_2^*N^{-1}K_1.
\end{equation}
\end{lemma}

\begin{proof}[\bf Proof]
Using  the first identity in \eqref{projFFprime} and next  the first identity in \eqref{propsom} we see that
\[
\o P_\sF  h= \o (K_1^*N^{-1}K_1) h =(\o K_1^*)N^{-1}K_1 h =K_2^*N^{-1}K_1 h,  \quad h\in \sH_1,
\]
and \eqref{eq:om} is proved.
\end{proof}

\subsection{Construction of complementary operators}\label{Assec:CD}

This subsection deals with the construction of operators $C$ and $D$ satisfying \eqref{semiunit1} and \eqref{semiunit2} assuming the operators $Z$ and $B$ are given. As in Section \ref{sec:intro} the operators $Z$ and $B$ are Hilbert space operators, $Z:\sZ\to \sZ$ and   $B:\sY \to \sZ$. Moreover, we assume that $Z^*$ is pointwise stable, and  $P$ is  a strictly positive operator on $\sZ$ satisfying the Stein equation
\begin{equation}\label{lyap-A}
P-ZPZ^*=BB^*.
\end{equation}
The fact that $P$ is strictly   positive, $Z^*$ is pointwise stable  and  satisfies \eqref{lyap-A} implies  that
 \[
 W=\begin{bmatrix} B &ZB &Z^2B& \cdots\end{bmatrix}:
 \ell_+^2(\sY) \rightarrow \sZ
 \]
 defines a bounded linear operator and $P= WW^*$. Moreover,  as in Section \ref{sec:intro}, we have
\begin{equation}\label{dataW1}
 ZW=WS_\sY \ands B=WE_\sY:\sY\to \sZ.
\end{equation}
Finally, note   that  $P$ is not necessarily equal to  $WW^*$ when  $Z^*$ is not pointwise stable. For example, if $Z$ is unitary, and $P = I$, then $B=0$ and $W=0$.

 To see that $W$ is a well-defined  operator,  consider the auxiliary operators
\begin{equation}\label{def:BZ1}
B_1=P^{-\frac{1}{2}}B:\sY \to \sZ\ands  Z_1=P^{-\frac{1}{2}}ZP^{\frac{1}{2}}:\sZ\to \sZ.
\end{equation}
Multiplying   the Stein equation  $P-ZPZ^*=BB^*$  by $P^{-\frac{1}{2}}$ on the left and right yields $I-Z_1 Z_1^*=B_1B_1^*$, and hence
\begin{equation}\label{BZ1*}
\begin{bmatrix}B_1^*\\Z_1^* \end{bmatrix} :\sZ\to \begin{bmatrix}\sY\\\sZ \end{bmatrix}
\end{equation}
is an isometry.  In particular, the operator in \eqref{BZ1*} is a contraction.
But then we can apply  Lemma \ref{lem-obs}   to show that
\begin{equation}\label{obs-K}
K:=\begin{bmatrix} B_1^*\\ B_1^*Z_1^*\\ B_1^*(Z_1^*)^2\\ \vdots\end{bmatrix}:\ell_+^2(\sY)\to \sZ
\end{equation}
is a well defined bounded linear operator and $\|K\|\leq 1$. Note that the adjoint of $K$ is the operator $K^*$ given by
\[
K^*=\begin{bmatrix} B_1& Z_1 B_1& Z_1^2 B_1&\cdots \end{bmatrix}: \ell_+^2(\sY) \to \sZ.
\]
Using the definitions of  $B_1$ and $Z_1$ in \eqref{def:BZ1} we see that
\[
Z_1^n B_1=\left(P^{-\frac{1}{2}}ZP^{\frac{1}{2}}\right)^n P^{-\frac{1}{2}}B=\left(P^{-\frac{1}{2}}Z^nP^{\frac{1}{2}}\right)P^{-\frac{1}{2}}B=
P^{-\frac{1}{2}}Z^nB.
\]
Thus $P^{\frac{1}{2}}K^*=W$, and hence $W$ is a well defined operator from $\ell_+^2(\sY)$ into $\sZ$.  It is emphasized that because $P$ is strictly positive,
the operator  $Z^*$ must be  pointwise stable; see the first part of the proof of  Lemma 3.1.  The latter implies  that the solution of the Stein equation $P = Z P Z^* + BB^*$ is unique (see Lemma \ref{lem_lyap}), and thus $P=WW^*$.

\medskip
As mentioned in the Introduction (in the paragraph after formulas  \eqref{semiunit1} and \eqref{semiunit2}) there are various ways to construct admissible pairs of complementary operators. One such construction, using the Beurling-Lax-Halmos theorem, was given in the Introduction. The next proposition provides an alternative method   which has the advantage that it can be readily used in Matlab in the finite dimensional case.

\begin{proposition}\label{Aprop:CD1}  Let $Z:\sZ\to \sZ$ and $B:\sY \to \sZ$ be Hilbert space operators,   where  $Z^*$ is pointwise stable.
Moreover, assume {that  $P$ is   strictly positive operator
satisfying} the Stein equation $P= Z P Z^* + B B^*$.
Then there exists a Hilbert space $\sE$ and Hilbert space operators $C:\sZ\to \sE$ and $D: \sY\to \sE$  such that
\begin{align}
&\begin{bmatrix}
  D & C \\
  B & Z
\end{bmatrix}\begin{bmatrix}
I_\sY & 0 \\
  0 &  P
\end{bmatrix}\begin{bmatrix}
D^* & B^* \\
  C^* & Z^*
  \end{bmatrix} = \begin{bmatrix}
  I_\sE  & 0 \\
  0 & P
\end{bmatrix},\label{Asemiunit1}\\[.2cm]
&\begin{bmatrix}
D^* &   B^* \\
  C^* & Z^*
\end{bmatrix} \begin{bmatrix}
 I_\sE  & 0  \\
  0 &P^{-1}
\end{bmatrix}
\begin{bmatrix}
  D & C \\
  B & Z
\end{bmatrix}
 = \begin{bmatrix}
  I _\sY & 0  \\
  0 & P^{-1}
\end{bmatrix}. \label{Asemiunit2}
\end{align}
One such a pair of operators can be constructed in the following way.
Let $\varphi$ be any isometry from some space $\sE_0$ onto
the {null space} of
$\begin{bmatrix} B  &  Z P^{\frac{1}{2}}\end{bmatrix}$ of the form
\begin{equation}\label{def:E0}
\varphi = \begin{bmatrix}
            \varphi_1 \\
            \varphi_2 \\
          \end{bmatrix}:\sE_0 \rightarrow
\begin{bmatrix} \sY\\ \sZ \end{bmatrix}.
\end{equation}
 Define the operators $C_0$ and $D_0$ by
{\begin{equation}\label{def:C0D0}
C_0 = \varphi_2^* P^{-\frac{1}{2}}:\sZ\to \sE_0 \ands
D_0= \varphi_1^*:\sY\to \sE_0.
\end{equation}}
Finally, all  operators $C:\sZ\to \sE$ and $D: \sY\to \sE$  satisfying \eqref{Asemiunit1} and  \eqref{Asemiunit2} are given by
\begin{equation}\label{def:CD1}
C=UC_0 \!  \ands \!  D=UD_0 \quad \!\mbox{with $U:\sE_0\to\sE$  any unitary operator}.
\end{equation}
\end{proposition}

\begin{proof}[\bf Proof]
Let $Z_1$ and $B_1$ be the operators defined by equation  \eqref{def:BZ1}. Note that  $Z_1Z_1^*+ B_1B_1^*=I$, the  identity operator on $\sZ$.  Furthermore,  the two  identities \eqref{Asemiunit1} and  \eqref{Asemiunit2} are equivalent to the statement that the operator
\begin{equation}\label{Munit}
M:= \begin{bmatrix}     D   &   C P^{\frac{1}{2}} \\
                      B_1   &   Z_1   \end{bmatrix}:
\begin{bmatrix}\sY  \\ \sZ\end{bmatrix}
\to \begin{bmatrix}\sE  \\ \sZ \end{bmatrix}
\end{equation}
is unitary.  Notice that
$\begin{bmatrix}  B  &   Z P^{\frac{1}{2}}\end{bmatrix}$
and $\begin{bmatrix}  B_1  &   Z_1 \end{bmatrix}$
have the same {null space}. By construction the operator
\begin{equation}\label{unit_phi}
\begin{bmatrix}     \varphi_1^*   &   \varphi_2^* \\
                      B_1   &   Z_1   \end{bmatrix}:
\begin{bmatrix}\sY \\ \sZ\end{bmatrix}
\to \begin{bmatrix}\sE_0   \\ \sZ \end{bmatrix}
\end{equation}
is unitary. So choosing $D = \varphi_1^*$ and $C = \varphi_2^* P^{-\frac{1}{2}}$
yields a system $\{Z,B,C,D\}$ satisfying \eqref{Asemiunit1} and
\eqref{Asemiunit2}.  It easily follows that \eqref{Asemiunit1} and
\eqref{Asemiunit2} remain true when $C$ and $D$ are multiplied with a unitary operator on the left side. Hence \eqref{Asemiunit1} and \eqref{Asemiunit2} holds for $C$ and $D$ as in \eqref{def:CD1}.

Let $\{Z,B,C,D\}$ be any system satisfying
 \eqref{Asemiunit1} and \eqref{Asemiunit2}. Because $M$ is unitary {the two operators}
\[
\varphi = \begin{bmatrix}
            \varphi_1 \\
            \varphi_2 \\
          \end{bmatrix}:\sE_0 \rightarrow
\begin{bmatrix}\sY \\ \sZ\end{bmatrix}
\quad \mbox{and} \quad
V = \begin{bmatrix}
            D^* \\
            P^{\frac{1}{2}}C^* \\
          \end{bmatrix}:\sE \rightarrow
\begin{bmatrix}\sY \\ \sZ\end{bmatrix}
\]
are   isometries whose ranges are equal to the  null space  of
$\begin{bmatrix}  B  &   Z P^{\frac{1}{2}}\end{bmatrix}$.
Therefore,  $\varphi\varphi^* = VV^*$  {is equal to}  the orthogonal projection
onto the  null space of
$\begin{bmatrix}  B  &   Z P^{\frac{1}{2}}\end{bmatrix}$. Hence  there exists a unitary
operator $U$ from $\sE_0$ onto $\sE$ satisfying
\[
\begin{bmatrix}
            \varphi_1 \\
            \varphi_2 \\
          \end{bmatrix} = \begin{bmatrix}
            D^* \\
            P^{\frac{1}{2}}C^* \\
          \end{bmatrix}U;
\]
 use the special case of  the
 Douglas factorization  presented in Lemma \ref{lem:KK1}.
 Thus,  $U \varphi_1^* = D$ and $U\varphi_2^* P^{-\frac{1}{2}} =C$.
\end{proof}

\begin{proposition} \label{Aprop:CD2}
Let $Z:\sZ\to \sZ$ and $B:\sY \to \sZ$ be Hilbert space operators where  $Z^*$ is pointwise stable. Moreover, assume  that  $P$ is   strictly positive operator satisfying  the Stein equation $P= Z P Z^* + B B^*$.   Let $C:\sZ\to \sE$ and $D: \sY\to \sE$ be Hilbert space operators  such that \eqref{Asemiunit1} and \eqref{Asemiunit2} are satisfied. Put
\begin{equation}\label{def:Atheta}
\tht(\lambda)=D^* +\lambda B^*(I-\lambda Z^*)^{-1}C^*.
\end{equation}
Then $\tht \in \sS(\sE, \sY)$ and $\tht$ is inner. Moreover,
\begin{equation}\label{eq:CD2}
\kr W=\im T_\tht, \quad C=E_\sE^*T_\tht^*S_\sY W^*P^{-1}, \quad D=\tht(0)^*,
\end{equation}
where  $W=\begin{bmatrix} B &ZB &Z^2B& \cdots\end{bmatrix}$
mapping $\ell_+^2(\sY)$ into $\sZ$  is the operator determined by  \eqref{dataW1}.
\end{proposition}

\begin{proof}[\bf Proof]
The fact $\tht \in \sS(\sE, \sY)$ and $\tht$ is inner is a direct consequence of   \eqref{Asemiunit1} and the pointwise stability of $Z^*$. Indeed, from \eqref{Asemiunit1} we obtain that the realization of $\Theta$ given by the system matrix $M^*$, with $M$ as in \eqref{Munit}, has an isometric system matrix and a pointwise stable state matrix $Z_1^*=P^\half Z^*P^{-\half}$, so that the claim follows from Theorem III.10.1 in \cite{FFGK98}.  For completeness, we present a proof. Let $\Theta(\lambda) = \sum_{n=0}^\infty  \lambda_n \Theta_n$ be the Taylor series expansion for $\Theta$. Note that
$\Theta(0) = D^*$ and $\Theta_n = B^* (Z^*)^{n-1} C^*$ for all integers
$n \geq 1$. Let $\Phi$ be the operator defined by
\begin{equation}\label{TE100}
\Phi = \begin{bmatrix} \tht_0\\
 \begin{bmatrix}\tht_1 \\ \tht_2\\ \tht_3\\ \vdots\end{bmatrix}\end{bmatrix}
=\begin{bmatrix} D^*\\W^*C^*\end{bmatrix}:\sE \to
\begin{bmatrix} \sY \\ \ell_+^2(\sY) \end{bmatrix}.
\end{equation}
Because $W$ is a bounded operator mapping $\ell_+^2(\sY)$ into $\sZ$, it follows that $\Phi$ is a well defined operator. In fact, $\Phi$
is an isometry. To see this observe that \eqref{Asemiunit1} yields,
\[
\Phi^* \Phi = DD^* + C W W^* C^* = DD^* + C P C^* = I.
\]
Hence $\Phi$ is an isometry. Moreover, $ \Phi \sE$ is a wandering subspace for the unilateral shift $S_\sY$, that is, $\{S_\sY^n \Phi \sE\}_{n=0}^\infty$  forms a set of  orthogonal subspaces. To see this it is sufficient to show that
$ \Phi \sE$ is orthogonal to $S_\sY^n \Phi \sE$ for all integers
$n \geq 1$. Using $S_\sY^* W^* = W^* Z^*$, with $n \geq 1$, we obtain
\begin{align*}
\left(S_\sY^n \Phi\right)^* \Phi &= \Phi^* (S_\sY^*)^n \Phi =
\begin{bmatrix}
  D  & C W \\
\end{bmatrix} (S_\sY^*)^n \begin{bmatrix}
  D^*  \\ W^* C^* \\
\end{bmatrix}\\
 &=\begin{bmatrix}
  D  & C W \\
\end{bmatrix}  \begin{bmatrix}
  B^* (Z^*)^{n-1} C^*  \\ W^* (Z^*)^n C^* \\
\end{bmatrix}\\
&= D B^* (Z^*)^{n-1} C^* + C WW^* (Z^*)^n C^*\\
&= \left(D B^* + C P Z^*\right)(Z^*)^{n-1} C^* =0.
\end{align*}
The last equality follows from  \eqref{Asemiunit1}. Therefore   $\{S_\sY^n \Phi \sE\}_0^\infty$ forms a set of orthogonal subspaces.

The Toeplitz matrix $T_\Theta$ is determined  by
\[
T_\Theta = \begin{bmatrix}
             \Phi & S_\sY \Phi & S_\sY^2 \Phi & \cdots \\
           \end{bmatrix}.
\]
Because $\Phi$ is  an   isometry and $\Phi \sE$ is a wandering
subspace for $S_\sY$, it follow that all the columns $\{S_\sY^n \Phi\}_0^\infty$
are isometric and orthogonal. Therefore  $T_\Theta^* T_\Theta =I$
and $\Theta$ is an inner function.

Now let us  show that  $\kr W=\im T_\tht$. To this end, note that
\begin{equation}\label{TE1}
T_\tht E_\sE= \Phi
= \begin{bmatrix} D^*\\W^*C^*\end{bmatrix}:
 \sE \rightarrow \begin{bmatrix}
                   \sY \\
                   \ell_+^2(\sY) \\
                 \end{bmatrix}.
\end{equation}
 Because $P = WW^*$ is  strictly positive
the range of $W^*$ is closed. Moreover, one can directly verify that
$W^* P^{-1}W$ is the orthogonal
projection onto the range of $W^*$. Hence $I - W^* P^{-1}W$ is the orthogonal
projection onto $\kr W$. Since $T_\Theta$ is an isometry $T_\Theta T_\Theta^*$
is an orthogonal projection. We claim that $I - W^* P^{-1}W =  T_\Theta T_\Theta^*$,
and thus, $\kr W=\im T_\Theta$. To this end,  notice that
$T_\Theta T_\Theta^*$ is the unique solution to the Stein equation
\begin{equation}\label{LYAP00}
T_\Theta T_\Theta^* = S_\sY T_\Theta T_\Theta^* S_\sY^* +
T_\Theta E_\sE E_\sE^* T_\Theta^*.
\end{equation}
Because $S_\sY^*$ is pointwise stable, the solution $T_\Theta T_\Theta^*$
to this Stein equation
is unique; see Lemma \ref{lem_lyap}.  Moreover, using $W = \begin{bmatrix}
                                  B  & Z W \\
                                \end{bmatrix}$ with \eqref{Asemiunit2}, we have
\begin{align*}
& I - W^* P^{-1}W - S_\sY\big(I - W^* P^{-1}W\big)S_\sY^*\\
&=
E_\sY E_\sY^*+  S_\sY W^* P^{-1}W S_\sY^* - W^* P^{-1}W\\
&= \begin{bmatrix}
     I & 0 \\
     0 & W^* P^{-1}W \\
   \end{bmatrix} - \begin{bmatrix}
                                  B^*  \\ W^* Z^* \\
                                \end{bmatrix}P^{-1}\begin{bmatrix}
                                  B  & Z W \\
                                \end{bmatrix}\\
&=\begin{bmatrix}
     I - B^*  P^{-1} B   & - B^* P^{-1} Z W \\
     -W^* Z^* P^{-1} B   & W^* P^{-1}W - W^* Z^* P^{-1} Z W \\
   \end{bmatrix} \\
   &= \begin{bmatrix}
     D^* D            & D^* C W  \\
      W^* C^* D    & W^* C^* CW \\
   \end{bmatrix} = T_\Theta E_\sE E_\sE^* T_\Theta.
\end{align*}
So $I - W^* P^{-1}W$ is also the solution to the Stein equation
\eqref{LYAP00}. Because $S_\sY^*$ is pointwise stable, the solution
to this Stein equation is unique. Therefore
$T_\Theta  T_\Theta = I - W^* P^{-1}W$ and $\kr W = \im T_\Theta$.

It remains to prove the second and third identity in \eqref{eq:CD2}. Using \eqref{TE1} we see that
\[
E_\sE^*T_\tht^*S_\sY W^*P^{-1}=
\begin{bmatrix}  D& C W  \end{bmatrix}
\begin{bmatrix}  0 \\ W^*P^{-1}  \end{bmatrix} CWW^*P^{-1}=CPP^{-1}=C.
\]
This proves the second identity in \eqref{eq:CD2}. The third follows by taking
$\lambda=0$  in \eqref{def:Atheta}.
\end{proof}

\begin{proposition}\label{Aprop:CD3}  Let $Z:\sZ\to \sZ$ and $B:\sY \to \sZ$ be Hilbert space operators  where  $Z^*$ is pointwise stable.
Moreover, assume   that  $P$ is   strictly positive operator satisfying  the Stein equation $P= Z P Z^* + B B^*$.    Let $\Theta \in \sS(\sE, \sY)$ be any inner function such that $\kr W=\im T_\tht$, where $W$ is the operator appearing in \eqref{dataW1}. Then  the operators
\begin{equation}\label{def:CD3}
C:=E_\sE^* T_\tht^* S_\sY W^* P^{-1}: \sZ\to \sE \ands  D:=\tht(0)^*:\sY\to \sE.
\end{equation}
form an admissible pair  of complementary operators determined by $\{B, Z\}$,  that is, with this choice of $C$ and $D$ the identities  \eqref{Asemiunit1} and \eqref{Asemiunit2} are satisfied.
\end{proposition}

\begin{proof}[\bf Proof]
Notice that $S_\sY^* T_\Theta E_\sE$ is orthogonal to
$\im T_\Theta$. To see this simply observe that
\[
T_\Theta^* S_\sY^* T_\Theta E_\sE =S_\sE^* T_\Theta^*  T_\Theta E_\sE =
S_\sE^*   E_\sE =0.
\]
Because $\im T_\Theta = \kr W$, we see that  the range of  $S_\sY^* T_\Theta E_\sE$ is  contained  in the range of $W^*$. Since $P= WW^*$ is strictly positive, the range of $W^*$ is closed and $W^*$ is one to one.  Hence $\kr W^*=\{0\}$. By another implication of the Douglas factorization lemma, see e.g., \cite[Lemma XVII.5.2]{GGK2},  we obtain that there  exists a unique operator $C$ mapping $\sZ$ into $\sE$ such that $S_\sY^* T_\Theta E_\sE= W^* C^*$.
By taking the adjoint  we have  $C W = E_\sE^*T_\Theta^* S_\sY$. Hence
\[
C = C W W^* P^{-1} = E_\sE^*T_\Theta^* S_\sY W^* P^{-1}.
\]
In other words, $C$ is determined by the first equation in \eqref{def:CD3}.
By taking the Fourier transform we get
\begin{align*}
\Theta(\lambda) &= E_\sY^*(I-\l S_\sY^*)^{-1}T_\Theta E_\sE=
\Theta(0) + \lambda E_\sY^*(I-\l S_\sY^*)^{-1} S_\sY^* T_\Theta E_\sE\  \\
&= D^* +  \lambda E_\sY^*(I-\l S_\sY^*)^{-1} W^* C^* =
D^* + \lambda B^*(I- \lambda Z^*)^{-1}C^*.
\end{align*}
In other words, $\Theta(\lambda) = D^* + \lambda B^*(I- \lambda Z^*)^{-1}C^*$
and \eqref{eq:CD2} holds.

To derive \eqref{Asemiunit1}
 recall that $W^* C^* = S_\sY^* T_\Theta E_\sE$. Hence
\begin{align*}
DD^* + C P C^* &= \Theta(0)^*\Theta(0) + C WW^* C^*\\ &=
E_\sE^* T_\Theta^* E_\sY E_\sY^* T_\Theta E_\sE +
E_\sE^* T_\Theta^* S_\sY S_\sY^* T_\Theta E_\sE\\
&= E_\sE^* T_\Theta^* T_\Theta E_\sE =I.
\end{align*}
Hence $DD^* + C P C^* =I$. Moreover,
\begin{align*}
B D^* + Z P C^* &= \begin{bmatrix}
                    B & Z W \\
                  \end{bmatrix}\begin{bmatrix}
                    \Theta(0) \\ W^* C^* \\
                  \end{bmatrix}
                  =\begin{bmatrix}
                    B & Z W \\
                  \end{bmatrix}\begin{bmatrix}
                    \Theta(0) \\ S_\sY^* T_\Theta E_\sE \\
                  \end{bmatrix}\\
                  &=  W T_\Theta E_\sE =0.
\end{align*}
Thus $B D^* + Z P C^*=0$. This with $P = BB^* + Z P Z^*$,
yields \eqref{Asemiunit1}.

To obtain \eqref{Asemiunit2}, notice that $T_\Theta$ admits a decomposition of the form
\[
T_\Theta = \begin{bmatrix}
             D^*      & 0  \\
             W^* C^*  & T_\Theta \\
           \end{bmatrix}: \begin{bmatrix}
             \sE  \\
             \ell_+^2(\sE) \\
           \end{bmatrix}\rightarrow \begin{bmatrix}
             \sY  \\
             \ell_+^2(\sY) \\
           \end{bmatrix}.
\]
Because $\kr W = \im T_\Theta$ and $W^*P^{-1}W$ is the orthogonal
projection onto the range of $W^*$, we have
$T_\Theta T_\Theta^* = I - W^*P^{-1}W$. Using
$W = \begin{bmatrix}
       B  & Z W \\
     \end{bmatrix}$,   we obtain
\begin{align*}
&\begin{bmatrix}
  I - B^*P^{-1}B  & - B^*P^{-1}Z W  \\
  - W^* Z^* P^{-1} B & I - W^*Z^* P^{-1}Z W \\
\end{bmatrix} =I - \begin{bmatrix}
   B^*      \\
  W^* Z^*   \\
\end{bmatrix}P^{-1} \begin{bmatrix}
  B &  Z W  \\
\end{bmatrix}\\
&= I - W^*P^{-1}W = T_\Theta T_\Theta^* =
\begin{bmatrix}
  D^*D        & D^* C W  \\
  W^* C^* D   & W^* C^*C W + T_\Theta T_\Theta^* \\
\end{bmatrix}\\
&= \begin{bmatrix}
  D^*D        & D^* C W  \\
  W^* C^* D   & W^* C^*C W + I - W^* P^{-1} W \\
\end{bmatrix}.
\end{align*}
By comparing the upper left hand corner of the first and
last matrices, we have
$D^*D + B^* P^{-1} B =I$.  Because $W$ is onto,  comparing the
upper right hand corner shows that $D^* C + B^* P^{-1} Z = 0$.
Since $W^*$ is one to one, comparing the lower right hand
corner shows that $P^{-1} = Z P^{-1}Z^* + C^*C$.
This yields \eqref{Asemiunit2}. Therefore $\{C, D\}$ is
 an admissible pair of complementary operators.
\end{proof}

\begin{proof}[\bf Alternative proof of Proposition \ref{Aprop:CD3}]
To gain  some further insight, let us  derive Proposition \ref{Aprop:CD3} as a corollary of Proposition \ref{Aprop:CD2}  using the uniqueness part of the Beurling-Lax-Halmos theorem;  see \cite[Theorem 3.1.1]{FB10}.

Let $\wt{C}:\sZ\to \wt{\sE}$ and $\wt{D}: \sY\to \wt{\sE}$ be Hilbert space operators  such that \eqref{Asemiunit1} and \eqref{Asemiunit2} are satisfied with $\wt{C}$ and $\wt{D}$ in place of $C$ and $D$, respectively.  Set
\[
\wt{\Theta}(\lambda)=\wt{D}^* +\lambda B^*(I-\lambda Z^*)^{-1}\wt{C}^*.
\]
Then, by Proposition \ref{Aprop:CD2}, the function $\wt{\tht}$ is inner and $\kr W=\im T_{\wt{\tht}}$. Thus $\im T_{\wt{\tht}}=\im T_\tht$, and hence using the uniqueness part of the Beurling-Lax-Halmos theorem there exists a unitary operator $U$ from $\wt{\sE}$ onto $\sE$ such that
\[
\tht(\lambda)U=\wt{\tht}(\lambda) \qquad (\lambda\in \BD).
\]
Now put $C=U\wt{C}$ and $D=U\wt{D}$. From the final part of  Proposition \ref{Aprop:CD1} we know $\{C, D\}$  form an admissible pair  of complementary operators determined by $\{B, Z\}$.

It remains to show that $C$ and $D$ are given by  \eqref{def:CD3}.  From the second and third identity in \eqref{eq:CD2} we know that
\begin{equation}\label{wtCD}
 \wt{C}=E_{\wt{\sE}}^*T_{\wt{\tht}}^*S_\sY W^*P^{-1} \ands \wt{D}=\wt{\tht}(0)^*.
\end{equation}
Since $U:\wt{\sE}\to \sE$ is unitary we have $UE_{\wt{\sE}}^*T_{\wt{\tht}}^*=E_\sE^* T_\tht^* $. Thus the first identity in \eqref{wtCD} shows that $C=U\wt{C}$ is given by the first identity in \eqref{def:CD3}.  Similarly, we have
\[
D=U\wt{D}=U\wt{\tht}(0)^*=U\big(\tht(0)U\big)^*=\tht(0)^*,
\]
which proves the second identity in \eqref{def:CD3}.
\end{proof}

\paragraph{\bf An example}  Let $\sM$ be a subspace of $\ell_+^2(\sY)$ invariant under the block forward shift $S_\sY$. The Beurling-Lax-Halmos theorem   \cite[Theorem 3.1.1]{FB10} tells us that there exist a Hilbert space $\sE$ and an inner function
$\Theta\in \sS(\sE, \sY)$ such that $\sM=\im T_\tht$. Moreover, if $\Psi$ is an inner function
in $\sS(\sE_\circ, \sY)$ satisfying  $\sM=\im T_\Psi$,
then $\Theta(\lambda) U = \Psi(\lambda) $ where $U$ is a
constant unitary operator mapping $\sE_\circ$ into $\sE$.

We shall derive this result as a special case   of  Proposition \ref{Aprop:CD2}.  Put $\sZ=\ell_+^2(\sY)\ominus \sM$, and define
\begin{equation}
Z= \Pi_\sZ S_\sY \Pi_\sZ ^*: \sZ\to \sZ \ands
B=\Pi_\sZ  E_\sY:\sY\to \sZ. \label{BLH-ZB}
\end{equation}
Note that $Z$ is the compression of $S_\sY$  to  $\sZ$, and $\sZ$ is an invariant subspace for $S_\sY^*$.  Let $W$ be the operator mapping $\ell_+^2(\sY)$ onto   $\sZ$ defined by $W = \Pi_\sZ$. Since $\sM$ is   an invariant subspace for $S_\sY$, we have
\[
S_\sY=\begin{bmatrix}Z&0\\ \star &\star \end{bmatrix}:
\begin{bmatrix}\sZ\\ \sM  \end{bmatrix}\to \begin{bmatrix}\sZ\\ \sM  \end{bmatrix}
\]
where $\star$ represents an unspecified entry.  In particular, this implies that
\[
W S_\sY = \begin{bmatrix}
            I & 0 \\
          \end{bmatrix}\begin{bmatrix}Z&0\\ \star &\star \end{bmatrix}
= \begin{bmatrix}
            Z & 0 \\
          \end{bmatrix} = Z \begin{bmatrix}
            I & 0 \\
          \end{bmatrix} = Z W.
\]
Hence $Z W = W S_\sY$. By construction $B = W E_\sY$. Thus
$I = WW^*$ is the unique solution to the Stein equation
$P = Z P Z^* + BB^*$.

%
%It follows that $ZZ^*$ is the compression of  $S_\sY S_\sY^*$ to $\sZ$. Together with the definition of $B$ in the second part of \eqref{BLH-ZB} this yields\begin{align*}
%ZZ^*+BB^*&=\Pi_\sZ S_\sY S_\sY^* \Pi_\sZ^*+\Pi_\sZ  E_\sY E_\sY^*\Pi_\sZ^*\\
%&=\Pi_\sZ\left(S_\sY S_\sY^*+E_\sY E_\sY^* \right)\Pi_\sZ^*=\Pi_\sZ \Pi_\sZ^*=I_\sZ.
%\end{align*}
%Thus \eqref{lyap-A} is satisfied with $P=I_\sZ$.
%
%Next, in the present setting,  we compute the operator $W$ appearing in \eqref{dataW1}. Take $x\in \sZ$. Since
%\[
%S_\sY^*=\begin{bmatrix}Z^*&\star\\ 0 &\star \end{bmatrix}:
%\begin{bmatrix}\sZ\\ \sM  \end{bmatrix}\to \begin{bmatrix}\sZ\\ \sM  \end{bmatrix},
%\]
%we see that
%\[
%(S_\sY^*)^n x=\begin{bmatrix}(Z^*)^n &\star\\ 0 &\star \end{bmatrix}\begin{bmatrix}x\\ 0 \end{bmatrix}=\begin{bmatrix}(Z^*)^n x\\ 0 \end{bmatrix}, \quad x\in \sZ\quad (n=0, 1,2, \ldots).
%\]
%Since $B^*=E_\sZ^* \Pi_\sZ^*$, we conclude that
%\begin{equation}\label{ESBZ}
%E_\sZ^*(S_\sY^*)^n x=\begin{bmatrix}B^*(Z^*)^n \\ 0  \end{bmatrix}, \quad x\in \sZ\quad (n=0, 1,2, \ldots).
%\end{equation}
%Next observe that for $x=\col[x_n]_{n=0}^\iy \in \sZ\subset \ell_+^2(\sY)$ we have
%\[
%E_\sZ^*(S_\sY^*)^n x=x_n \quad x\in \sZ\quad (n=0, 1,2, \ldots).
%\]
%Using the latter equalities  in \eqref{ESBZ} we see that
%\[
%W^*= \begin{bmatrix}  B^*\\  B^*Z^* \\ B^*(Z^*)^2\\ \vdots \end{bmatrix}=\Pi_\sZ^*:\sZ\to  \ell_+^2(\sY)\ands W=\Pi_\sZ.
%\]
The fact that $W=\Pi_\sZ$ implies that $\kr W=\ell_+^2(\sY)\ominus \sZ=\sM$. But then Proposition \ref{Aprop:CD2} tells us that there exist a Hilbert space $\sE$ and an inner function $\tht\in \sS(\sE, \sY)$ such $\sM= \im T_\tht$ which is the Beurling-Lax-Halmos result.  Moreover, Propositions \ref{Aprop:CD1}  and \ref{Aprop:CD2} together provide a procedure to construct $\tht$ .

To prove the uniqueness, assume that
$\Psi$ is another  inner function
in $\sS(\sE_\circ, \sY)$ satisfying  $\sM=\im T_\Psi$.
Because $T_\Theta$ and $T_\Psi$ are two isometries
whose range equals $\sM$, it follows that
$T_\Theta T_\Theta^*= T_\Psi T_\Psi^* = P_\sM$,
the orthogonal projection onto
$\sM$. According to the  variant of the Douglas factorization lemma  discussed in the preceding subsection (see Lemma \ref{lem:KK1})
we have $T_\Theta V= T_\Psi$ where $V$ is a unitary operator
from $\ell_+^2(\sE_\circ)$ onto $\ell_+^2(\sE)$.
 Because $S_\sY T_\Theta = T_\Theta S_{\sE}$ and
 $S_\sY T_\Psi = T_\Psi S_{\sE_\circ}$, we see that
 $S_{\sE} V = V S_{\sE_\circ}$.
 So $V$ is a lower triangular unitary Toeplitz operator. Hence
 $V = T_U$ where $U$ is  a constant function on $\BD$ whose value is a unitary operator, also denoted by $U$, mapping $\sE_\circ$
 into $\sE$. Therefore  $\Theta(\lambda) U = \Psi(\lambda) $.

%%%%%%%%%%%%%%%%%%%%%%%%%
%%%%%%%%%%%%%%%%%%%%%%%%%%

\subsection{Construction of a co-isometric realization}\label{Assec:co-iso}

In Section \ref{sec:ONP-coiso} an important role is played by the classical fact  that an operator-valued function $F$ is a Schur class function  if and only if $F$ admits an observable co-isometric realization (see Theorem  \ref{thm:co}).  The ``if part''  in this theorem  is straightforward and holds true for any contraction. Indeed,  assume
that
\begin{equation}\label{M}
M = \begin{bmatrix}
      \delta & \gamma \\
      \beta  & \alpha \\
    \end{bmatrix}:
    \begin{bmatrix}
    \mathcal{U} \\
      \mathcal{X} \\
    \end{bmatrix}\rightarrow\begin{bmatrix} \mathcal{Y}\\\mathcal{X} \\ \end{bmatrix}
\end{equation}
 is a contraction.  Then $\alpha$ is a contraction, and thus  $(I - \lambda \alpha)^{-1}$ is well defined
for all $\lambda$ in the open unit disc $\mathbb{D}$.
Hence $F(\lambda) = \delta + \lambda \gamma (I - \lambda \alpha)^{-1} \beta$
is analytic in $\mathbb{D}$.
Now observe that for $u$ in $\mathcal{U}$, we have
\[
\begin{bmatrix}
  F(\lambda)u \\
   (I - \lambda \alpha)^{-1} \beta u \\
\end{bmatrix}
= \begin{bmatrix}
      \delta & \gamma \\
      \beta  & \alpha \\
    \end{bmatrix} \begin{bmatrix}
  u \\
  \lambda (I - \lambda \alpha)^{-1} \beta u \\
\end{bmatrix}.
\]
Using the fact that $M$ is contraction, we see that
\[
\|F(\lambda) u\|^2 \leq
\|F(\lambda) u\|^2 + \|(I - \lambda \alpha)^{-1} \beta u\|^2(1 - |\lambda|^2) \leq \|u\|^2.
\]
Hence  $\|F(\lambda)\| \leq 1$ for each $\lambda \in \mathbb{D}$.
Therefore $F$ is in the Schur class $\mathcal{S}(\mathcal{U},\mathcal{Y})$.

The only ``only if part''  is much less trivial and has a long and interesting history (see the paragraph directly after Theorem \ref{thm:co}).    Here we present an alternative proof  of  the ``only if part'' inspired by the proof of Theorem \ref{thm:allsol1}; see the end of this section for more details.

\begin{proof}[\bf Proof of the ``only if'' part of  Theorem \ref{thm:co}]
Let $F\in S(\sU,\sY)$, and let  $T= T_F$  be the block Toeplitz
 operator mapping $\ell_+^2(\sU)$ into $\ell_+^2(\sY)$
 defined by $F$. The fact that $F$ is a Schur class
 function implies that $T$ is a contraction, and
 hence the defect operator $D_{T^*}=(I-T T^*)^{\frac{1}{2}}$
 is well defined. With $T$ we associate the following two auxiliary operators:
\begin{align*}
K&=\begin{bmatrix} E_\sY & S_\sY D_{T^*}\end{bmatrix}:
\begin{bmatrix}\sY \\[.1cm]  \ell_+^2(\sY)\end{bmatrix}\to \ell_+^2(\sY),\\
L&=\begin{bmatrix} T E_\sU &   D_{T^*}\end{bmatrix}:
\begin{bmatrix}\sU \\[.1cm]  \ell_+^2(\sY)\end{bmatrix}\to \ell_+^2(\sY).
\end{align*}
Here $D_{T^*}$ is the positive square root of
$I - TT^*$.

\smallskip\noindent
\textsc{Part 1.} We first show that there exists a co-isometry $M$ mapping  $ \sU\oplus \ell_+^2(\sY)$ into $ \sY\oplus \ell_+^2(\sY)$ such that $KM=L$. To see this, note that
\begin{align*}
KK^*&=E_\sY E_\sY^*+S_\sY (I-T T^*)S_\sY^*=
E_\sY E_\sY^*+S_\sY S_\sY^*-TS_\sU S_\sU^*T^* \\
&=I_{\ell_+^2(\sY)} - TS_\sU S_\sU^* T^*;\\
LL^*&=T E_\sU E_\sU^* T^*+(I-T T^*)=I_{\ell_+^2(\sY)} -T(I- E_\sU E_\sU^* )T^*\\
&=I_{\ell_+^2(\sY)} - TS_\sU S_\sU^* T^*.
\end{align*}
Thus $KK^*=LL^*$. It follows (apply  Lemma  \ref{lem:KK1}  with $K_1=K$ and $K_2=L$) that there exists a unique unitary operator $\tau_1$ mapping $\overline{\im K^*}$  onto $\overline{\im L^*}$ such $\tau_1 K^*f=L^* f$ for each $f\in  \ell_+^2(\sY)$. Furthermore, $\begin{bmatrix} y&x \end{bmatrix}{}^\top \in \kr K$ if and only if   $y=0$ and  $x\in\kr D_{T^*}$. The latter implies that the operator $\tau_2$ from $ \kr K$ to $\kr L$ defined by
\[
\tau_2  \begin{bmatrix} 0\\ x \end{bmatrix}= \begin{bmatrix} 0\\ x \end{bmatrix}, \quad x\in \kr D_{T^*}
\]
is a well defined isometry from  $ \kr K$ to $\kr L$. Since
\[
\overline{\im K^*} \oplus \kr K= \sY\oplus \ell_+^2(\sY) \ands \overline{\im L^*} \oplus \kr L = \sU\oplus \ell_+^2(\sY).
\]
It follows that $N=\tau_1 \oplus \tau_2$ is an isometry from $ \sY\oplus \ell_+^2(\sY)$ into  $ \sU\oplus \ell_+^2(\sY)$ such that $NK^*=L^*$.
But then $M=N^*$ is a co-isometry from $ \sU\oplus \ell_+^2(\sY)$ into $ \sY\oplus \ell_+^2(\sY)$ such that $KM=L$.

We partition $M$ as a $2\ts 2$ operator matrix using the Hilbert space direct sums $\sU\oplus \ell_+^2(\sY)$ and  $\sY\oplus \ell_+^2(\sY)$, as follows:
\[
M= \begin{bmatrix}
      \delta & \gamma \\
      \beta  & \alpha \\
    \end{bmatrix}:  \begin{bmatrix}\sU \\[.1cm]  \ell_+^2(\sY)\end{bmatrix}\to \begin{bmatrix}\sY \\[.1cm]  \ell_+^2(\sY)\end{bmatrix}.
\]
Finally, using this decomposition with $K M = L$, we obtain
\[
\begin{bmatrix} E_\sY & S_\sY D_{T^*}\end{bmatrix}
\begin{bmatrix}
      \delta & \gamma \\
      \beta  & \alpha \\
    \end{bmatrix} = \begin{bmatrix} T E_\sU &   D_{T^*}\end{bmatrix}.
\]

\smallskip\noindent
\textsc{Part 2.} We show that $F$ is given by the state space realization
\begin{equation}\label{realFa}
F(\lambda) = \delta +  \lambda \gamma (I - \lambda \alpha)^{-1} \beta
\qquad (\lambda \in \mathbb{D}).
\end{equation}
Since $M$ is a co-isometry, $M$ is a contraction, and hence
 the operator $\begin{bmatrix}\gamma &\a \end{bmatrix}{}^\top$
is also a contraction. But then we can apply  Lemma  3.1 in \cite{FFK02} (see Lemma  \ref{lem-obs} below) to show that the  observability  operator
\begin{equation}\label{defGa1}
\ga:= \begin{bmatrix}\gamma\\ \gamma\a\\ \gamma\a^2\\ \vdots \end{bmatrix}:
\sZ  \to \ell_+^2(\sY)
\end{equation}
is well defined and a contraction. Note that
\begin{equation}\label{prGa1}
\Gamma-S_{\sY} \Gamma\a=E_\sY \gamma.
\end{equation}
Furthermore,  the identity $ KM=L$ is equivalent to
\begin{equation}\label{KML1}
T E_\sU= E_\sY \d + S_{\sY} D_{T^*}\beta \ands
D_{T^*}= E_\sY \gamma +S_{\sY} D_{T^*}\a .
\end{equation}
Using the second identity in \eqref{KML1} along with \eqref{prGa1} we see that
\begin{align}
D_{T^*}-\ga&=\left(E_\sY \gamma +S_{\sY} D_{T^*}\a \right)-\left(E_\sY \gamma +S_{\sY} \Gamma\a\right)=S_{\sY}(D_{T^*}-\ga)\a \nn \\[.1cm]
&=S_{\sY}^n(D_{T^*}-\ga)\a^n, \quad n=0, 1, 2, \cdots.\label{KLM2}
\end{align}
Since    $\a$ is a contraction and  $S_\sY^*$  is pointwise stable, it follows  that for each $f\in \ell_+^2(\sY)$ we have
\[
(D_{T^*}- \Gamma^*)f
= (\a^*)^n \Big(D_{T^*}-\ga^*\Big)(S_\sY^*)^n f \to 0 \quad (n\to \iy).
\]
But then we have $\ga= D_{T^*}$. Thus, by the first identity in
 \eqref{KML1},  we obtain
\begin{align*}
TE_\sU&=  E_\sY \d + S_{\sY} D_{T^*}\beta =E_\sY \d + S_{\sY}\ga \b\\[.1cm]
=&
\begin{bmatrix} \d \\ 0 \\0 \\0 \\ \vdots  \end{bmatrix} +
\begin{bmatrix} 0 \\ \g \\ \g\a\\ \g\a^2\\ \vdots  \end{bmatrix} \b.
\end{align*}
Since the first column of $T$ is given by the Fourier coefficients $F_0, F_1, F_2, \ldots $ of the Schur class function $F$, we conclude that
\[
F_0=\d \ands F_n=\g \a^{n-1}\beta, \quad n=1, 2, \ldots.
\]
This proves \eqref{realFa}.
\end{proof}

\begin{lemma}\label{lem-obs}
\textup{(\cite[Lemma 3.1]{FFK02})}
Assume that $\begin{bmatrix}\gamma &\alpha \end{bmatrix}{}^\top$ is a contraction mapping $\mathcal{Z}$ into $\sY \oplus \mathcal{Z}$. Then the observability operator $\Gamma = \col \big[\g \a^j\big]_{j=0}^\infty$ is also a contraction mapping $\mathcal{Z}$ into $\ell_+^2(\sY)$.
\end{lemma}

\begin{proof}[\bf Proof]
Because $\begin{bmatrix}\gamma &\alpha \end{bmatrix}{}^\top$ is a contraction,
$I \geq \gamma^* \gamma + \alpha^* \alpha$. By recursively using this fact,
we obtain
\begin{align*}
I &\geq \gamma^* \gamma + \alpha^* \alpha \geq
\gamma^* \gamma + \alpha^*\left(\gamma^* \gamma + \alpha^* \alpha\right) \alpha\\[.1cm]
&\geq \gamma^* \gamma + \alpha^*\gamma^* \gamma \alpha +
\alpha^{*2}\left(\gamma^* \gamma + \alpha^* \alpha\right) \alpha^2\cdots \\
&\geq \sum_{j=0}^n \alpha^{*j}\gamma^* \gamma \alpha^j + \alpha^{*n+1}\alpha^{n+1}, \quad
n=0, 1,2, \dots.
\end{align*}
In particular,
$I \geq \sum_{0}^n \alpha^{*j}\gamma^* \gamma \alpha^j$ for any integer
$n \geq 0$. Therefore $I \geq \Gamma^*\Gamma$ and $\Gamma$ is a contraction.
\end{proof}

\subsection{Outer functions}
\label{sec-out}
The first lemma presented in this section plays an important role in the proof of  Proposition~\ref{prop:Ups22a}.
 Recall that an operator-valued   function $\Phi$ whose values are operators mapping $\sU$ into $\sY$ is called \emph{outer} if  $\Phi$ is analytic on $\BD$, for each $u\in \sU$ the function $\Phi(\cdot)u$ is in  $H^2(\sY)$, and $\Phi(\cdot)\sU$ is cyclic with respect to the forward shift on $H^2(\sY)$. The latter is equivalent to the following condition:
\begin{equation}\label{def:outer}
\bigvee_{n\geq 0}S_\sY^n
 \begin{bmatrix}
\Phi_0\\ \Phi_1\\ \Phi_2\\ \vdots
\end{bmatrix}\sU =\ell_+^2(\sY) \qquad   \mbox{where }
\Phi(\lambda) = \sum_{j=0}^\infty \lambda^j \Phi_j.
\end{equation}
The following result has  its roots in  \cite{FFGK98} and its proof is presented for the sake of  completeness.

\begin{lemma}\label{lem:outer1} Let $A$ be a strict contraction  mapping $\ell_+^2(\sU)$ into an auxiliary Hilbert space $\sH^\prime$  satisfying the inequality $S_\sU^* A^* A S_\sU \leq A^*A$. Then
\[
\Phi(\lambda) = E_\sU^*(I - \lambda S_\sU^*)^{-1} (I- A^*A)^{-1}E_\sU, \quad \l\in \BD,
\]
is an outer function.  Furthermore, there exists a function $\Psi\in H^\infty(\sU,\sU)$ such that $\Psi(\l)\Phi(\l)u=u$ for each $u\in\sU$ and $\l\in\BD$. In particular, if $\Phi(\l)$ is invertible for each $\l\in\BD$, then $\Phi(\l)^{-1}$ is in $H^\infty(\sU,\sU)$.
\end{lemma}

We shall derive the above lemma as a corollary of  the following somewhat more general lemma.

\begin{lemma}\label{lem:outer2}Let $\Omega$ be a strictly positive operator on $\ell_+^2(\sU)$, and assume that $\Omega\leq S_\sU^* \Omega S_\sU$. Then
the function $\Phi(\l)=E_\sU^*(I-\l S_\sU^*)^{-1}\Omega^{-1}E_\sU$ is outer.
Furthermore, there exists a function $\Psi\in H^\infty(\sU,\sU)$ such that $\Psi(\l)\Phi(\l)u=u$ for each $u\in\sU$ and $\l\in\BD$. In particular, iif $\Phi(\l)$ is invertible for each $\l\in\BD$, then $\Phi(\l)^{-1}$ is in $H^\infty(\sU,\sU)$.
\end{lemma}

The additional invertibility condition appearing in the final sentences of the above two lemmas is always fulfilled if $\sU$ is finite dimensional; see Remark 3.2.3 in \cite{FB10}.  Moreover, this invertibility  condition is also satisfied if $\Phi=\Upsilon_{22}$, where  $\Upsilon_{22}$ is given by \eqref{Upsilon22}.

\begin{proof}[\bf Proof of Lemma \ref{lem:outer1}]
Put   $\Omega =I- A^*A$. Since $S_\sU^* A^* A S_\sU \leq A^* A$, we have
\[
\Omega = I - A^*A \leq I-S_\sU^* A^* A S_\sU = S_\sU^*\left(I- A^* A \right)S_\sU = S_\sU^* \Omega S_\sU.
\]
Applying the Lemma  \ref{lem:outer2} with $\Omega =I- A^*A$
yields  the desired result.
\end{proof}

\begin{proof}[\bf Proof of Lemma \ref{lem:outer2}]
Notice that
\[
\Omega^{\frac{1}{2}} \Omega^{\frac{1}{2}} =
\Omega\leq S_\sU^* \Omega S_\sU =
\left(\Omega^{\frac{1}{2}} S_\sU\right)^*\Omega^{\frac{1}{2}} S_\sU.
\]
According to the Douglas factorization lemma  there exists a contraction $C$  mapping the subspace
 $\sM = \overline{\Omega^{\frac{1}{2}}S_\sU \ell_+^2(\sU)}$ into
 $\ell_+^2(\mathcal{U})$ satisfying $C \Omega^{\frac{1}{2}} S_\sU = \Omega^{\frac{1}{2}}$.  We extend  $C$ to the whole space   $\ell_+^2(\sU)$ by setting $C|\sM^\perp =0$. So $C$ is a  well defined contraction on $\ell_+^2(\sU)$.
The remaining part of the proof is split into two parts.

\smallskip
\noindent \textsc{Part 1.}
In this part we show that the function $\Phi(\lambda)$ is outer. Assume that   $h$ is a vector in $\ell_+^2(\sU)$  which is orthogonal to $S_\sU^n \Omega^{-1}  E_\sU \sU$  for all integer $n \geq 0$. We have to show that $h=0$.  Since $h$ is orthogonal  $S_\sU^n \Omega^{-1} E_\sU \sU$ for all $n\geq 0$,  we obtain  $\Omega^{-1} S_\sU^{* n} h$ is orthogonal to $E_\sU \sU$ for all $n \geq 0$. So there exists a vector $h_n$ in $\ell_+^2(\sU)$ such that $\Omega^{-1} S_\sU^{* n} h = S_\sU h_n$. Multiplying on the left by $\Omega^{\frac{1}{2}}$ shows that $\Omega^{-\frac{1}{2}} S_\sU^{* n} h = \Omega^{\frac{1}{2}}  S_\sU h_n$ is a vector in $\sM$ for all $n \geq 0$. We claim that
\begin{equation}\label{np176}
C^* \Omega^{-\frac{1}{2}} S_\sU^{* n+1} h  = \Omega^{-\frac{1}{2}}  S_\sU^{* n} h
\qquad (\mbox{for all integers } n\geq 0).
\end{equation}
To see this notice that for $g$ in $\ell_+^2(\sU)$, we have
\begin{align*}
\lg C^* \Omega^{-\frac{1}{2}}  S_\sU^{* n+1} h, \Omega^{\frac{1}{2}}  S_\sU g \rg &=
\lg \Omega^{-\frac{1}{2}} S_\sU^{* n+1} h,C   \Omega^{\frac{1}{2}} S_\sU g\rg\\[.1cm]
&= \lg  \Omega^{-\frac{1}{2}} S_\sU^{* n+1} h,   \Omega^{\frac{1}{2}} g) = ( S_\sU^{* n+1} h, g \rg\\[.1cm]
&=\lg  S_\sU^{* n} h,S_\sU g\rg = \lg  \Omega^{-\frac{1}{2}} S_\sU^{* n} h,   \Omega^{\frac{1}{2}}S_\sU g\rg.
\end{align*}
Since $\Omega^{\frac{1}{2}} S_\sU \ell_+^2(\sU)$ is dense in $\sM$ and
$\Omega^{-\frac{1}{2}} S_\sU^{* n} h \in \sM$, we obtain \eqref{np176}.
The recursion relation in \eqref{np176} implies that
\[
 \Omega^{-\frac{1}{2}}  h = C^* \Omega^{-\frac{1}{2}} S_\sU^{*} h =
 C^{*2}  \Omega^{-\frac{1}{2}}  S_\sU^{* 2} h = \cdots
 = C^{*n} \Omega^{-\frac{1}{2}} S_\sU^{* n} h.
\]
In other words, $ \Omega^{-\frac{1}{2}}  h = C^{*n} \Omega^{-\frac{1}{2}} S_\sU^{* n} h$ for all integers $n \geq 0$.
Because $C$ is a contraction, we have
\[
\|\Omega^{-\frac{1}{2}}  h\| =
\|C^{*n} \Omega^{-\frac{1}{2}}S_\sU^{* n} h \| \leq\| \Omega^{-\frac{1}{2}} S_\sU^{* n} h\| \rightarrow 0\quad (n\to \infty).
\]
Since $ \Omega^{-\frac{1}{2}}  $ is invertible, $h =0$. So
 the closed linear span of $\{S_\sU^n \Omega^{-1} E_\sU  \sU\}_0^\infty$ equals $\ell_+^2(\sU)$ and the function $\Phi$ is outer.

%%%%%%%%%%%%%%%%%%%%%%%%%%

\smallskip
\noindent \textsc{Part 2.}
In this part we prove the remaining claims.
In order to do this, let $\sL$ be the linear space of all sequences  $u=\{u_j\}_{j=0}^\iy$, $u_j\in \sU$ for $j=0, 1, 2, \ldots$, with compact support. The latter means that $u_j\not =0$   for a finite number of  indices $j$ only. Note that $\sL\subset \ell^2_+(\sU) $ and that $\sL$ is invariant under the forward shift $S_\sU$.  Given  $\sL$ we consider the linear map    $M$ from $\sL$ into $\ell^2_+(\sU) $ defined by
 \begin{equation*}
M u = \begin{bmatrix}
            \Omega^{-1} E_\sU   &
            S_\sU \Omega^{-1} E_\sU  & S_\sU^2 \Omega^{-1} E_\sU  & \cdots \\
         \end{bmatrix} u =\sum_{j=0}^\infty S_\sU^j \Omega^{-1} E_\sU u_j.
\end{equation*}
If we identify $\ell_+^2(\sU)$ with the Hardy space $H^2(\sU)$ using the Fourier transform, then $\sL$ is just the space of all $\sU$-valued polynomials, and $M$ is the operator of multiplication by $\Phi$ acting on the $\sU$-valued polynomials.

We shall show that there exists $\epsilon>0$  such that $\|M u \|\geq \epsilon \|u\|$  for each $u=\{u_j\}_{j=0}^\iy$.  Note that
{\small
\begin{align*}
 &\|\Omega^\half M u\|^2= \left\|\Omega^{\frac{1}{2}}\sum_{j=0}^\infty S_\sU^j \Omega^{-1} E_\sU u_j\right\|^2=
 \left\langle \Omega \sum_{j=0}^\infty S_\sU^j \Omega^{-1} E_\sU u_j,
 \sum_{k=0}^\infty S_\sU^k \Omega^{-1} E_\sU u_k\right\rangle\\
 &\ =\left\langle \Omega \left(\Omega^{-1} E_\sU u_0 +
 S_\sU\sum_{j=0}^\infty S_\sU^j \Omega^{-1} E_\sU u_{j+1}\right),
 \Omega^{-1} E_\sU u_0 +
 S_\sU\sum_{k=0}^\infty S_\sU^k \Omega^{-1} E_\sU u_{k+1}\right\rangle.
\end{align*}}
Set $\Delta = E_\sU^*\Omega^{-1}E_\sU$. Using the fact that $E_\sU^* S_\sU=0$ and  $S_\sU^*\om S_\sU\geq \om$ we obtain that

\begin{align*}
\small\|\Omega^\half M u\|^2
&= \langle \Delta u_0,u_0\rangle
+\left\langle S_\sU^*\Omega S_\sU \sum_{j=0}^\infty S_\sU^j \Omega^{-1} E_\sU u_{j+1},
\sum_{k=0}^\infty S_\sU^k \Omega^{-1} E_\sU u_{k+1}\right\rangle\\
&\geq \langle \Delta u_0,u_0\rangle
+\left\langle  \Omega  \sum_{j=0}^\infty S_\sU^j \Omega^{-1} E_\sU u_{j+1},
\sum_{k=0}^\infty S_\sU^k \Omega^{-1} E_\sU u_{k+1}\right\rangle\\
&=\|\Delta^\half u_0\|^2+\|\Omega^\half M S_\sU^* u\|^2.
\end{align*}
Applying the above computation to $S_\sU^* u$ instead of $u$, and continuing recursively we obtain that
\begin{equation}\label{ineqM}
\|\Omega^\half M u\|^2\geq \sum_{j=0}^\infty \|\Delta^\half u_j\|^2.
\end{equation}
Since $\Delta$ is strictly positive, there exists  a $\epsilon_1>0$ such that
$ \| \Delta^\half u_j\| \geq \epsilon_1\|u_j\|$ for all $j=0,1,2, \ldots$. But then the inequality \eqref{ineqM} shows that
\begin{align}
\|M u\|^2 &\geq \|\Omega^\half\|^{-1}\| \Omega^\half M u\|^2\geq \|\Omega^\half\|^{-1}\sum_{j=0}^\infty \|\Delta^\half u_j\|^2\nn \\
&\geq \epsilon_1^2 \|\Omega^\half\|^{-1} \sum_{j=0}^\infty\|u_j\|^2 \geq \epsilon_1^2 \|\Omega^\half\|^{-1}\|u\|^2 \nn \\
&= \epsilon^2 \|u\|^2, \hspace{.15cm}  \mbox{where $\epsilon=\epsilon_1 \|\Omega^\half\|^{-\frac{1}{2}}$}.\label{ImpIneq}
\end{align}
We conclude that $M$ is bounded from below.

Next, put $\sR=M\sL \subset \ell_+^2(\sU)$. Then $M$ maps $\sL$ in a one-to-one way onto $\sR$. By $T$ we denote the corresponding inverse operator. Then the result of the previous paragraph tells us that $\|Tf\|\leq \epsilon^{-1}\|f\| $ for each $f\in \sR$. The fact that $\Phi$ is outer implies that $\sR$  is dense  in $\ell_+^2(\sU)$. It follows that $T$ extends to a bounded linear operator   from $\ell_+^2(\sU)$ into $\ell_+^2(\sU)$ which we also denote by $T$. Recall that  $\sL$ is invariant under the forward shift $S_\sU$. Since   $S_\sU Mu=M S_\sU u$ for each $u\in \sL$, we also have $S_\sU Tf=T S_\sU f$  for each $f\in \sR$. But then the fact that $T$ is a bounded linear operator on $\ell_+^2(\sU)$ implies by continuity that $S_\sU Tg=T S_\sU g$  for each $g\in \ell_+^2(\sU)$. It follows that $T$  is a (block) lower triangular Toeplitz operator. Let $\Psi \in H^\iy(\sU,\sU)$   be its defining function, i.e., $T=T_\Psi$. Since $TM u=u$ for each  $u\in \sL$, we have
\begin{equation}\label{leftinv}
\Psi(\l)\Phi(\l)u=u,  \qquad  u\in \sU,\, \l\in\BD.
\end{equation}
Now if $\Phi(\l)$ is invertible for each $\l\in \BD$, then it is clear that $\Phi(\l)^{-1}=\Psi(\l)$ is in $H^\infty(\sU,\sU)$.
\end{proof}

Observe that for the case when $\dim \,\sU <\infty$  the identity \eqref{leftinv} implies that $\Phi(\l)$ is invertible for each $\l\in\BD$ without using Remark 3.2.3. in \cite{FB10}.

\begin{remark}
\textup{It is interesting to consider the special case when    $\Omega$ is a strictly positive Toeplitz operator on $\ell_+^2(\sU)$. In this case
$\Omega = S_\sU^* \Omega S_\sU$, and the proof of Lemma \ref{lem:outer2}  yields  a classical   result on spectral factorization; see, e.g., Proposition 10.2.1 in \cite{FB10}. Indeed, put $\Psi(\lambda) =
\left(E_\sU^* \Omega^{-1} E_\sU\right)^{\frac{1}{2}}\Phi(\lambda)^{-1}$  where, as before,   $\Phi(\l)=E_\sU^*(I-\l S_\sU^*)^{-1}\Omega^{-1}E_\sU$. The fact that $\Omega$ is a strictly positive Toeplitz operator then implies that  $\Phi(\l) $  is invertible for each $\l \in \BD$, and
 $\Psi(\l)$ and $\Psi(\l)^{-1}$ are both functions in $H^\infty(\sU,\sU)$. Moreover,  $\Psi$ is the outer spectral factor for $\Omega$, that is, $\Omega = T_\Psi^* T_\Psi$ and $\Psi$ is an outer function. To prove the latter using  elements of  the proof of Lemma \ref{lem:outer2}, observe that in this setting,
 we have equality in  \eqref{ineqM}, that is,
\[
\|\Omega^{\frac{1}{2}} T_\Phi u \|^2 =  \sum_{j=0}^\iy \|\Delta^{\frac{1}{2}}u_j \|^2 \hspace{.15cm} \mbox{for all $u$ in $\ell_+^2(\sU)$ with compact support}.
\]
Because $T_\Phi^{-1}$ is a bounded  operator, we have
$\|\Omega^{\frac{1}{2}} u \|^2 = \|\Delta^{\frac{1}{2}}T_\Phi^{-1} u\|^2$ for
all $u$ in $\ell_+^2(\sU)$. In other words, $\Omega = T_\Psi^* T_\Psi$. Since
$\Omega$ is strictly positive and $\Phi$ is outer, $T_\Psi$  is well defined bounded invertible operator.  Hence $\Psi$ and $\Psi^{-1}$ are both functions in $H^\infty(\sU,\sU)$,  and $\Psi$ is the outer spectral factor for $\Omega$. See   Section 10.2 in \cite{FB10} for  further  details.}
\end{remark}

\subsection{An operator optimization problem}\label{ssec: optimization}
The results in this subsection  provide background  material  for Section \ref{sec:max}. We begin with an elementary optimization problem. Let $A_1: \sH\to \sU$ and $A_2: \sH\to \sR$ be a Hilbert space operators, where $\overline{\im A_2}=\sR$ and $\sR\subset \sH$. With these two operators we associate a  \emph{cost function}  $\sigma(u)$ on $\sU$,  namely
\begin{equation}\label{def:cost1}
\s(u)=\inf \{\|u-A_1 h\|^2 +\|A_2 h\|^2 \mid h\in \sH\},\qquad  u\in \sU.
\end{equation}
To understand the problem better  let $A$ be the operator given by:
\[
A=\begin{bmatrix}A_1\\ A_2  \end{bmatrix}:\sH\to \begin{bmatrix}\sU  \\ \sR\end{bmatrix} \ands \mbox{put $\sA=\overline{\im A}$}.
\]
Then by the projection theorem
\begin{equation*}
\s(u)=\inf  \left\{\left\|\begin{bmatrix}u\\ 0\end{bmatrix} -Ah\right\|^2 \mid h\in \sH\right\}=
\left\|(I- P_{\sA})\begin{bmatrix}u\\ 0\end{bmatrix}\right\|^2.
\end{equation*}
Here $P_\sA$ is the orthogonal projection on $\sU\oplus \sR$  with range $\sA=\overline{\im A}$. Next, let $\Pi_{\sU}$ be the orthogonal projection of $\sU\oplus\sR$ onto $\sU$, and thus $\Pi_{\sU}^*$ is the canonical embedding of $\sU$ into $\sU\oplus\sR$. Using this notation we see that
\begin{align}
\sigma(u)&= \left\|(I- P_\sA)\begin{bmatrix}u\\ 0\end{bmatrix}\right\|^2 = \|(I- P_\sA)\Pi_\sU^* u\|^2\nn\\
&= \lg \Pi_{\sU} P_{\sA^\perp}\Pi_{\sU}^* u, u\rg,\quad  u\in \sU.\label{eq:cost1}
\end{align}
 In particular, $\sigma(u) = \lg \Pi_{\sU} P_{\sA^\perp}\Pi_{\sU}^* u, u\rg$ is quadratic function in $u$.
Here $\sA^\perp$ is the orthogonal complement of $\sA$ in $\sU\oplus \sR$.

The case when $A_2^*A_2$ is strictly positive is of particular interest. In case $A_2^*A_2$ is strictly positive, $A^*A=A_1^*A_1+ A_2^*A_2$   is also strictly positive. It follows that $P_\sA=A(A^*A)^{-1}A^*$. Moreover, we have
\begin{align*}
 \Pi_{\sU} P_{\sA^\perp}\Pi_{\sU}^*&=I_{\sU}-\Pi_{\sU} P_{\sA}\Pi_{\sU}^*=I_{\sU}-\Pi_{\sU} A(A^*A)^{-1}A^*\Pi_{\sU}^*\\
&=I_{\sU}-A_1(A_1^*A_1+ A_2^*A_2)^{-1}A_1^*\\
&= I_{\sU}-A_1\Big(I_{\sH}+ ( A_2^*A_2)^{-1}A_1^*A_1\Big)^{-1}( A_2^*A_2)^{-1}A_1^* \\
&= I_{\sU}-\Big(I_{\sU}+ A_1( A_2^*A_2)^{-1}A_1^*\Big)^{-1}A_1( A_2^*A_2)^{-1}A_1^* \\
&=\Big(I_{\sU}+A_1( A_2^*A_2)^{-1}A_1^*\Big)^{-1}.
\end{align*}
Thus when $A_2$ is strictly positive,  then the cost function is given by
\begin{equation}\label{eq:cost2}
\s (u)=\lg (I_{\sU}+A_1 (A_2^*A_2)^{-1}A_1^*)^{-1} u, u\rg, \qquad  u\in \sU.
\end{equation}

%XXXXXXXXXXXX

\paragraph{A special choice of  $A_1$ and $A_2$.} Let $C$ be a contraction from the Hilbert space  $\sE$ into the Hilbert space $\sH$, let $\sU$  be a subspace of $\sE$, and let $\sR= \sD_{C^*}$ where  $\sD_{C^*}$  is the closure of the range of the defect operator $D_{C^*}=(I_{\sH} - CC^*)^{\frac{1}{2}}$. Put
\[
A_1=\t_{\sU}^* C^*:\sH\to \sU \ands A_2=D_{C^*}:\sH\to  \sR.
\]
Here $\t_{\sU}$ is the canonical embedding of  $\sU$ into $\sE$. Thus $C\t_{\sU}$ maps $\sU$ into $\sH$. In this case the cost function $\s$ is given by  \begin{equation}\label{optc}
\sigma(u) = \inf\{\|u - \t_{\sU}^* C^* h\|^2 + \lg (I- CC^*)h,h\rg
\mid  h \in \sH\}, \quad  u\in \sU.
\end{equation}
Furthermore, the operator  $A$ is given by
\begin{equation}\label{optcA}
A = \begin{bmatrix}
 \t_{\sU}^* C^* \\
 D_{C^*}  \\
 \end{bmatrix} \mid \sH \to \begin{bmatrix}
  \sU  \\
\sR \\
 \end{bmatrix}, \hspace{.15cm} \mbox{where} \hspace{.15cm}   \sR=\sD_{C^*}.
 \end{equation}
 Finally, if  $C$ is  a strict contraction,  then $D_{C^*}$ is invertible and $\sR=\sH$. Using \eqref{eq:cost2} it follows that
\begin{align*}
\s(u)&=\lg \left(I_{\sU}+\tau_{\sU}^*C^*(I_{\sH} -CC^*)^{-1}C
\t_{\sU}\right)u, u\rg \\
&=\lg \t_{\sU}^*\left(I_{\sE}+ C^*(I_{\sH} -CC^*)^{-1}C\right)\t_{\sU} u, u\rg\\
&=\lg\t_{\sU}^*\left(I_{\sE} + (I_{\sE} -C^*C)^{-1}C^*C \right)\t_{\sU} u, u\rg \\
&=\lg\t_{\sU}^*\left(I_{\sE} + (I_{\sE} -C^*C)^{-1}
(C^*C -I_{\sE} +I_{\sE})\right)\t_{\sU} u, u\rg\\
&=\lg\t_{\sU}^*(I_{\sE} -C^*C)^{-1}\t_{\sU} u, u\rg, \quad  u\in \sU.
 \end{align*}
 Thus in this case the cost function is given by
 \begin{equation}\label{eq:cost3}
 \s(u)=\lg(I_{\sE} -C^*C)^{-1}u , u\rg,\quad  u\in \sU \subset \sE.
 \end{equation}

 The next lemma shows that additional information on $\sE\ominus \sU$ yields alternative formulas for the cost function.

\begin{lemma} Let $V$ be  an isometry on $\sE$ such that $\im V=\sE\ominus \sU$. Then the cost function $\s$ defined by \eqref{optc} is also given by
\begin{equation}\label{optclt}
\sigma(u) = \inf\{\|D_C (\tau_{\sU} u -  V e )\|^2 \mid
u \in \sE\},\qquad  u\in \sU.
\end{equation}
\end{lemma}

\begin{proof}[\bf Proof]
To prove the lemma   we shall use the so-called \emph{rotation matrix} $R$ associated with the contraction $C$. Recall  (see, e.g., the paragraph after Proposition 1.2 in \cite[Section XXVII.1]{GGK2} that
\begin{equation}\label{rotation}
R = \begin{bmatrix}
      C^*       & D_C \\
      D_{C^*}   & -C
    \end{bmatrix}:\begin{bmatrix}
      \sH \\
      \sD_C \\
    \end{bmatrix} \rightarrow \begin{bmatrix}
      \sE \\
      \sD_{C^*}
    \end{bmatrix}
\end{equation}
is a unitary operator. As before, let $A$ be the operator given in \eqref{optcA}.
Using \eqref{optcA} one  sees  that $f \oplus g$ is a vector in $\sA^\perp$ if and only if $f \oplus g \in \sU \oplus \sD_{C^*}$ and $f \oplus g$ is orthogonal to $\sA$, that is,
\begin{align*}
0&= \left\lg \begin{bmatrix}f\\ g \end{bmatrix}, \begin{bmatrix}\t_{\sU}^*C^* \\ D_{C^*}\end{bmatrix}h\right\rg
=\left\lg \begin{bmatrix}\t_{\sU} f\\ g \end{bmatrix}, \begin{bmatrix}C^* \\ D_{C^*}\end{bmatrix}h\right\rg
=\left\lg \begin{bmatrix} f\\ g \end{bmatrix}, \begin{bmatrix}C^* \\ D_{C^*}\end{bmatrix}h\right\rg, \quad h\in \sH.
\end{align*}
Thus $f \oplus g$ is a vector in $\sA^\perp$ if and only if $f \oplus g \in \sU \oplus \sD_{C^*}$ and $f \oplus g$  is orthogonal to the  range of the first column of the  operator matrix $R$. Since $R$ is unitary, we conclude that $f \oplus g\in \sA^\perp$ if and only if  $f \oplus g \in \sU \oplus \sD_{C^*}$ and is contained in the range of the second column of $R$. In other words, $f \oplus g\in \sA^\perp$  if and only if
$f \oplus g \in \sU \oplus \sD_{C^*}$ and $f = D_C v$
and  $g = - C v$  for some $v\in \sD_C$.
Clearly, $D_C v \in \sU$ if and only if $D_C v$ is orthogonal to $\sU^\perp=\sE\ominus \sU$. However, $D_C v$ is orthogonal to $\sU^\perp$ if and only if $v$ is orthogonal to $D_C\sU^\perp$, or equivalently, $v$ is in $\sD_C \ominus  D_C \sU^\perp$. Since $D_C \sU^\perp = D_C V \sE$, we have
\[
(\im A)^\perp = \begin{bmatrix} D_C \\ -C \end{bmatrix}
\left(\sD_C \ominus D_C V \sE\right) =
\begin{bmatrix} D_C \\ -C \end{bmatrix}\mathfrak{D},
\]
where $\mathfrak{D} = \sD_C \ominus D_C V \sE$.  Therefore the orthogonal projection $P_{\sA^\perp}$ is given by
\[
P_{\sA^\perp}\begin{bmatrix} u\\ 0 \end{bmatrix} = \begin{bmatrix} D_C \\ -C \end{bmatrix}
P_\mathfrak{D}
 \begin{bmatrix} D_C& -C^*  \end{bmatrix}\begin{bmatrix} \t_{\sU} u\\ 0 \end{bmatrix}, \quad u\in \sU.
\]
Notice that  $\mathfrak{D}^\perp = \overline{D_C V \sE}$.
Using \eqref{eq:cost1} it follows  that
\begin{align*}
\sigma(u) &=  \left\lg P_{\sA^\perp}\begin{bmatrix} u\\ 0  \end{bmatrix}, \begin{bmatrix} u\\ 0   \end{bmatrix}\right\rg
=\lg P_\mathfrak{D} D_C \t_{\sU}  u, D_C \t_{\sU}  u\rg\\[.1cm]
&= \|P_\mathfrak{D} D_C \t_{\sU} u\|^2 =
\inf\{\|D_C \t_{\sU} u - d \|^2  \mid  d \in \mathfrak{D}^\perp\}\\[.1cm]
&= \inf\{\|D_C \t_{\sU} u - D_C V e\|^2 \mid  e \in \sE\}.
\end{align*}
Therefore the cost function $\sigma$ in the two optimization
problems \eqref{optc} and \eqref{optclt} are the same.
\end{proof}

\subsection{A connection to prediction theory and multiplicative diagonals}\label{ssec:multdiag}
Let $T_{R}$ be a non-negative  Toeplitz operator on $\ell_+^2(\sU)$ with
symbol $R$ in $L^\infty(\sU,\sU)$. Then a classical prediction problem is
solve the following optimization problem:
\begin{equation}\label{Pred1}
\sigma(u) = \inf\{\langle T_{R}(E_{\sU} u - S_{\sU}h),
E_{\sU} u - S_{\sU}h\rangle:
h \in \ell_+^2(\sU)\}
\end{equation}
where $u$ is a specified vector in $\sU$;
see Helson-Lowdenslager \cite{HL1,HL2}.

Recall that a non-negative  Toeplitz operator $T_R$ on $\ell_+^2(\sU)$ with
defining function $R$ in $L^\infty(\sU,\sU)$  admits an \emph{outer spectral factor} if there exists an outer function $\Psi$  in $H^\infty(\sU,\sE)$
such that $T_{R} = T_{\Psi}^* T_{\Psi}$, or equivalently,
$R(e^{i \theta}) = \Psi(e^{i \theta})^*\Psi(e^{i \theta})$
almost everywhere.  In this case, the outer spectral
factor $\Psi$ for $R$ is unique up to a
unitary constant on the left. In other words,  if  $\Phi$  in $H^\infty(\sU,\sV)$
 is another outer function satisfying $T_{R} =  T_{\Phi}^* T_{\Phi}$, then
$\Psi(\lambda) = U  \Phi(\lambda)$ where $U$ is a constant unitary operator
mapping $\sV$ onto $\sE$. Finally, it is noted that
not all non-negative  Toeplitz  operators  admit an outer spectral
factor. For example, if $R(e^{i \theta}) = 1$ for $0 \leq \theta \leq \pi$
and zero otherwise, then $T_{R}$ is a  non-negative  Toeplitz operator
on $\ell_+^2$ and does not admit an outer spectral factor. For
further results concerning outer spectral factorization
see \cite{sz-nf,FB10}. Following some ideas in Sz.-Nagy-Foias \cite{sz-nf}, we obtain the following result.

\begin{proposition}\label{prop:pred1} Assume that $T_{R}$ admits an outer spectral factorization $T_{R} = T_{\Psi}^* T_{\Psi}$ where $\Psi$ is an outer function in  $H^\infty(\sU,\sE)$. Then the function $\sigma$ in \eqref{Pred1} is also given by $\sigma(u) = \|\Psi(0)u\|^2$ for each $u\in \sU$. Moreover, the  cost function $\sigma$ is independent of the outer spectral factor $\Psi$ chosen for $R$.
\end{proposition}

\begin{proof}[\bf Proof] Observe that in this case
\begin{align*}
\sigma(u) &= \inf\{\langle T_{\Psi}^* T_{\Psi}
(E_{\sU} u - S_{\sU}h),E_{\sU} u - S_{\sU}h\rangle:
h \in \ell_+^2(\sU)\}\\
&= \inf\{\|T_{\Psi} (E_{\sU} u - S_{\sU}h)\|^2:h \in \ell_+^2(\sU)\}\\
&=  \inf\{\| E_{\sE}E_{\sE}^*T_{\Psi} E_{\sU} u
+ S_{\sE}S_{\sE}^*T_{\Psi} E_{\sU} u
- T_{\Psi}S_{\sU}h\|^2:h \in \ell_+^2(\sU)\}\\
&=  \inf\{\| E_{\sE}\Psi(0) u
+ S_{\sE}S_{\sE}^*T_{\Psi} E_{\sU} u
- S_{\sE}T_{\Psi}h\|^2:h \in \ell_+^2(\sU)\}\\
&= \inf\{\|\Psi(0)u\|^2+
\|S_{\sE} S_{\sE}^* T_{\Psi} E_{\sU} u - S_{\sE} T_{\Psi}  h\|^2:
h \in \ell_+^2(\sU)\}\\
&= \|\Psi(0)u\|^2+
\inf\{\|S_{\sE}^*T_{\Psi} E_{\sU} u -  T_{\Psi}  h\|^2:
h \in \ell_+^2(\sU)\} =\|\Psi(0)u\|^2.
\end{align*}
The last equality follows from the fact that $\Psi$ is outer, that is,
the range of $T_{\Psi}$ is dense in $\ell_+^2(\sE)$. Therefore
\begin{equation}\label{pred-out}
 \sigma(u) = \|\Psi(0)u\|^2 = \langle\Psi(0)^*\Psi(0)u,u\rangle, \qquad u \in \sU.
\end{equation}

The final statement  follows from the fact that the outer spectral factor $\Psi$ for $R$ is unique up to a unitary constant on the left.
\end{proof}

If $\sU$ is finite dimensional, then $R$ admits an outer
spectral factor $\Psi$ in $H^\infty(\sU,\sU)$ if and only if
\begin{equation}\label{szego00}
 \frac{1}{2\pi}\int_0^{2\pi} \ln \det[R(e^{i\theta})] d \theta > -\infty.
\end{equation}
In this case, the classical  Szeg\"{o} formula tells us that
\begin{equation}\label{szego}
 \det[\Psi(0)^*\Psi(0)] =
\exp\left({\frac{1}{2\pi}\int_0^{2\pi} \ln \det[R(e^{i\theta})] d \theta}\right)
\end{equation}
where $\det[T]$ is the determinant of a finite dimensional operator
with respect to any basis.

The following proposition is well known. The  equality in \eqref{Pred3} follows by a standard Schur complement computation.

\begin{proposition}\label{prop:pred2}
If $T_{R}$ is a strictly positive operator on $\ell_+^2(\sU)$, then
$T_{R}$ admits an outer spectral factor $\Psi$ in $H^\infty(\sU,\sU)$ and
\begin{equation}\label{Pred3}
 \sigma(u) = \|\Psi(0)u\|^2 = \langle\left(E_{\sU}^* T_R^{-1} E_{\sU}\right)^{-1}u,u\rangle, \quad u\in \sU.
\end{equation}
Moreover,   $\Psi(\lambda)^{-1}$ is also a function in $H^\infty(\sU,\sU)$.
\end{proposition}

When $T_{R}$ is strictly positive, then  $R$ also admits a factorization of the form:
\[
R(e^{i \theta}) = \Psi(e^{i \theta})^*\Psi(e^{i \theta})
=  \Psi_\circ(e^{i \theta})^*\Delta \Psi_\circ(e^{i \theta})
\]
where $\Psi_\circ$ is an outer function in $H^\infty(\sU,\sU)$
satisfying $\Psi_\circ(0) =I$ and $\Delta$ is a strictly  positive operator
on $\sU$. In fact, $\Delta = \Psi(0)^*\Psi(0)$ and
$\Psi_\circ(\lambda) = \Psi(0)^{-1} \Psi(\lambda)$.
The factorization $R(e^{i \theta}) = \Psi_\circ(e^{i \theta})^*\Delta \Psi_\circ(e^{i \theta})$
where $\Psi_\circ$ is an outer function in $H^\infty(\sU,\sU)$ satisfying
$\Psi_\circ(0)=I$ is unique. Moreover, $\Delta$ is called the
\emph{$($right$)$ multiplicative diagonal} of $R$. In this setting,
$\sigma(u) = \langle\Delta u,u\rangle$. Finally, it is noted that
the multiplicative diagonal is usually mentioned in the framework
of the Wiener algebra (see Remark \ref{rem:multipldiag} below).

Now assume that $F$ is a Schur function in $\sS(\sU,\sY)$.
Then $I - T_{F}^* T_{F}$ is a non-negative  Toeplitz operator on $\ell_+^2(\sU)$.
In this case, the optimization problem in \eqref{optclt}
with $V = S_{\sU}$ is equivalent to
\begin{equation}\label{Pred4}
\sigma(u) = \inf\{\langle (I - T_{F}^* T_{F})(E_{\sU} u - S_{\sU}h),
E_{\sU} u - S_{\sU}h\rangle:
h \in \ell_+^2(\sU)\}
\end{equation}
where $u$ is a specified vector in $\sU$.
Assume that $I - F^*F$ admits an outer spectral factor,
that is, $I - T_{F}^* T_{F} = T_{\Psi}^* T_{\Psi}$
for some outer function $\Psi$ in $H^\infty(\sU,\sE)$.
Then the corresponding
cost function $\sigma(u) = \|\Psi(0) u\|^2$.

If $T_{F}$ is a strict contraction,
or equivalently, $\|F\|_\infty <1$, then $I-T_F^* T_F$ is a strictly
positive operator on $\ell_+^2(\sU)$. Hence $I-T_F^* T_F$ admits an
outer spectral $\Psi$ factor in $H^\infty(\sU,\sU)$ and
$\Psi(\lambda)^{-1}$ is also in $H^\infty(\sU,\sU)$. Choosing $R = I - F^*F$ in
\eqref{Pred3}, yields
\begin{equation}\label{pred6}
\sigma(u) = \|\Psi(0) u\|^2 =
\langle \left(E_{\sU}^*(I - T_{F}^* T_{F})^{-1}E_{\sU}\right)^{-1} u,u\rangle.
\end{equation}
Finally, if $\sU$ is finite dimensional, then
\begin{equation}\label{szegoCC}
 \det[\Psi(0)^*\Psi(0)] =
\exp{\left(\frac{1}{2\pi}\int_0^{2\pi} \ln \det[I - F(e^{i\theta})^*F(e^{i\theta})] d \theta\right)}.
\end{equation}

%\omega\omega\omega\omega\omega
\medskip
\begin{remark}\label{rem:multipldiag}
\textup{Let $\sH$  be a Hilbert spaces, and  $W_\sH(\BT)$ we denote the operator Wiener algebra   on the unite circle which consists of  all $\sL(\sH, \sH)$-valued  functions on $\BT$ of the form
\[
F(\l)=\sum_{j=-\iy}^\iy \l^j F_j,\qquad \l\in \BT,
\]
where $F_j\in \sL(\sH, \sH)$ for each $j$ and $\sum_{j=-\iy}^\iy \| F_j\|<\iy$. By $W_{\sH, +}(\BT)$ we denote the subalgebra of  $W_\sH(\BT)$
consisting of all $F$ in $W_\sH(\BT)$  with $F_j=0$ for each $j\leq -1$. Now assume that $F(\l)$ is strictly positive  for each $\l\in \BD$. Then there exists a unique function $\Psi$ in $W_{\sH,+}(\BT)$ and a unique strictly positive operator $\de(F)$ on $\sH$  such that $\Psi$ is invertible in $W_{\sH, +}(\BT)$,  its index zero Fourier coefficient   $\Psi_0 =I_\sH$, and
\[
F(\l)=\Psi(\l)^* \de(F)\Psi(\l), \quad \l\in \BT.
\]
The operator $\de(F)$ is called the \emph{$($right$)$ multiplicative diagonal} of $F$. It is known that  $\de(F)$ is also given by
\[
\de(F)= \left(E_\sH^* T_F^{-1} E_\sH\right)^{-1}.
\]
See \cite{GKW91} where the notion of multiplicative diagonal is introduced  in a $\star$-algebra setting, and Sections XXXIV.4 and XXXV.1 in \cite{GGK2} for further  information.}
\end{remark}

\paragraph{Acknowledgement.}
We thank  Joseph A. Ball for his valuable comments on an earlier version of the present  paper. His observations  varied from remarks on the used terminology to relating  some of our results and proofs    to those in multivariable interpolation theory, in particular, in his work with Vladimir  Bolotnikov \cite{BB08}.

\paragraph{NRF statement disclaimer.}
 The present work is based on the research supported in part by the National Research Foundation of South Africa. Any opinion, finding and conclusion or recommendation expressed in this material is that of the authors and the NRF does not accept any liability in this regard.

%:refs

\end{document}